\documentclass[reqno,a4paper]{amsart}
\usepackage[utf8]{inputenc}
\usepackage{etoolbox}
\usepackage{dsfont}
\providetoggle{long_explanation}   
\settoggle{long_explanation}{true}

\title[Fourier Neural Operators]{
On universal approximation and error bounds for Fourier Neural Operators
}
\author{
Nikola Kovachki \and Samuel Lanthaler \and Siddhartha Mishra
}
\address[Nikola Kovachki]{Computing and Mathematical Sciences, California Institute of Technology, Pasadena, California 91125, USA}
\address[Samuel Lanthaler and S. Mishra]{Seminar for Applied Mathematics, ETH Z\"urich, R\"amistrasse 101, 8092 Z\"urich, Switzerland}

\usepackage{amsthm,amssymb}
\usepackage{mathtools,thmtools}
\usepackage{bm,esint}
\usepackage{color}
\usepackage{float,graphicx,subcaption}
\usepackage{cancel}
\usepackage[shortlabels]{enumitem}
\usepackage{mathrsfs}  
\usepackage{bbm}
\usepackage{euscript}
\usepackage{hyperref}
\usepackage{cleveref}
\usepackage{upgreek}

\newcommand{\step}[2]{
\vspace{10pt}
\textbf{Step #1:} 
#2
\vspace{5pt}
}

\usepackage[color=blue!20!white]{todonotes}
\newcommand{\todosl}[1]{\todo[size=\tiny,author=SL,backgroundcolor=orange!20!white]{#1}}
\newcommand{\todonk}[1]{\todo[size=\tiny,author=NK,backgroundcolor=blue!20!white]{#1}}

\newcommand{\size}{\mathrm{size}}
\newcommand{\width}{\mathrm{width}}
\newcommand{\depth}{\mathrm{depth}}
\newcommand{\lift}{\mathrm{lift}}

\renewcommand{\Re}{\mathrm{Re}}
\renewcommand{\Im}{\mathrm{Im}}

\newcommand{\fb}{\bm{\mathrm{e}}}

\newcommand{\sq}{\mathrm{sq}}

\renewcommand{\tilde}{\widetilde}
\renewcommand{\hat}{\widehat}

\newcommand{\Id}{\mathrm{Id}}

\newcommand{\Span}{\mathrm{span}}

\newcommand{\im}{\mathrm{Im}}
\newcommand{\embeds}{{\hookrightarrow}}

\renewcommand{\bar}{\overline}

\newcommand{\explain}[2]{\overset{\mathclap{\underset{\downarrow}{#2}}}{#1}}

\newcommand{\slot}{{\,\cdot\,}}

\newcommand{\T}{\mathbb{T}}
\newcommand{\R}{\mathbb{R}}
\newcommand{\C}{\mathbb{C}}

\newcommand{\N}{\mathbb{N}}

\newcommand{\Z}{\mathbb{Z}}

\renewcommand{\L}{\mathcal{L}}
\newcommand{\G}{\mathcal{G}}
\renewcommand{\div}{{\mathrm{div}}}

\newcommand{\Lip}{\mathrm{Lip}}

\renewcommand{\P}{\mathbb{P}}
\newcommand{\cA}{\mathcal{A}}

\newcommand{\cE}{\mathcal{E}}
\newcommand{\cF}{\mathcal{F}}
\newcommand{\cG}{\mathcal{G}}

\newcommand{\cI}{\mathcal{I}}
\newcommand{\cJ}{\mathcal{J}}
\newcommand{\cK}{\mathcal{K}}
\newcommand{\cL}{\mathcal{L}}

\newcommand{\cN}{\mathcal{N}}

\newcommand{\cQ}{\mathcal{Q}}
\newcommand{\cR}{\mathcal{R}}

\newcommand{\cT}{\mathcal{T}}
\newcommand{\cU}{\mathcal{U}}
\newcommand{\cV}{\mathcal{V}}

\newcommand{\tL}{\tilde{\cL}}
\newcommand{\tN}{\tilde{\cN}}

\newcommand{\hG}{\widehat{\cG}}

\newcommand{\hL}{\widehat{\cL}}

\newcommand{\hN}{\widehat{\cN}}

\newcommand{\hQ}{\widehat{\cQ}}
\newcommand{\hR}{\widehat{\cR}}

\newcommand\define[1]{{\boldmath\textbf{#1}}}

\renewcommand{\hat}{\widehat}

\newcommand{\set}[2]{{\left\{ #1 \,\middle|\, #2 \right\}}}

\newcommand{\dt}{{\tau}}

\declaretheoremstyle[
  headfont=\normalfont\bfseries\itshape,
  numbered=unless unique,
  bodyfont=\normalfont,
  spaceabove=1em plus 0.75em minus 0.25em,
  spacebelow=1em plus 0.75em minus 0.25em,
  qed={},
]{deflt}

\theoremstyle{deflt}
\newtheorem{theorem}{Theorem}[section]
\newtheorem{assumption}[theorem]{Assumption}
\newtheorem{setting}[theorem]{Setting}
\newtheorem{example}[theorem]{Example}
\newtheorem{algorithm}[theorem]{Algorithm}
\newtheorem{remark}[theorem]{Remark}
\newtheorem{definition}[theorem]{Definition}
\newtheorem{lemma}[theorem]{Lemma}

\numberwithin{equation}{section}
\numberwithin{theorem}{section}

\usepackage{tikz,tikz-cd}
\tikzset{every picture/.style={line width=0.75pt}} 
\usetikzlibrary{decorations.pathmorphing} 
\usetikzlibrary{shapes,arrows}
\usetikzlibrary{fit}

\newcommand{\rev}[1]{{\color{blue} #1}}

\begin{document}

\maketitle

\begin{abstract}
Fourier neural operators (FNOs) have recently been proposed as an effective framework for learning operators that map between infinite-dimensional spaces. We prove that FNOs are universal, in the sense that they can approximate any continuous operator to desired accuracy. Moreover, we suggest a mechanism by which FNOs can approximate operators associated with PDEs efficiently. Explicit error bounds are derived to show that the size of the FNO, approximating operators associated with a Darcy type elliptic PDE and with the incompressible Navier-Stokes equations of fluid  dynamics, only increases sub (log)-linearly in terms of the reciprocal of the error. Thus, FNOs are shown to efficiently approximate operators arising in a large class of PDEs. 
\end{abstract}

\section{Introduction}
Deep neural networks have been extremely successful in diverse fields of science and engineering including image classification, speech recognition, natural language understanding, autonomous systems, game intelligence and protein folding, \cite{DLnat} and references therein. Moreover, deep neural networks are being increasingly used successfully in scientific computing, particular in simulating physical and engineering systems modeled by partial differential equations (PDEs). Examples include the use of physics informed neural networks \cite{KAR1,KAR2,MM1,MM2} for solving forward and inverse problems for PDEs and supervised learning algorithms for high-dimensional parabolic PDEs \cite{HEJ1} and parametric elliptic \cite{Kuty,SchwabZech2019} and hyperbolic \cite{LMR1,LMRS1} PDEs, among others. 

The success of deep neural networks at a wide variety of learning tasks can be attributed to a confluence of several factors such as the availability of massive labeled data sets, the design of novel architectures and training algorithms as well as the abundance of high-end computing platforms such as GPUs \cite{DLbook}. Still, it is fair to surmise that this edifice of success partly rests on the foundation of \emph{universal approximation} \cite{BAR1,Cy1,HOR1}, i.e., the ability of neural networks to approximate any continuous (even measurable) function, mapping a finite-dimensional input space into another finite-dimensional output space,  to arbitrary accuracy.

However, many interesting learning tasks entail learning \emph{operators} i.e., mappings between an infinite-dimensional input Banach space and (possibly) an infinite-dimensional output space. A prototypical example in scientific computing is provided by nonlinear operators that map the initial datum into the (time series of) solution of a nonlinear time-dependent PDE such as the Navier-Stokes equations of fluid dynamics. A priori, it is unclear if neural networks can be successfully employed for learning such operators from data, given that their universality only pertains to finite-dimensional functions. 
\par The first successful use of neural networks in the context of such \emph{operator learning} was provided in \cite{ChenChen}, where the authors proposed a novel neural network based learning architecture, which they termed as \emph{operator networks} and proved that these operator networks possess a surprising universal approximation property for infinite-dimensional nonlinear operators. Operator networks are based on two different neural networks, a \emph{branch net} and a \emph{trunk net}, which are trained concurrently to learn from data. More recently, the authors of \cite{deeponets} have proposed using deep, instead of shallow, neural networks in both the trunk and branch net and have christened the resulting architecture as a \emph{DeepOnet}. In a recent article \cite{LMK2021}, the universal approximation property of DeepOnets was extended, making it completely analogous to universal approximation results for finite-dimensional functions by neural networks. The authors of \cite{LMK2021} were also able to show that DeepOnets can break the curse of dimensionality for a large variety of PDE learning tasks. Hence, in spite of the underlying infinite-dimensional setting, DeepOnets are capable of approximating a large variety of nonlinear operators efficiently. This is further validated by the success of DeepOnets in many interesting examples in scientific computing \cite{donet2,donet3,donet4} and references therein.  

An alternative operator learning framework is provided by the concept of \emph{neural operators}, first proposed in \cite{li2020neural}. Just as canonical artificial neural networks are a concatenated composition of multiple hidden layers, with each hidden layer composing an affine function with a scalar nonlinear activation function, neural operators also compose multiple hidden layers, with each hidden layer composing an affine \emph{operator} with a local, scalar nonlinear activation operator. The infinite-dimensional setup is reflected in the fact that the affine operator can be significantly more general than in the finite-dimensional case, where it is represented by a weight matrix and bias vector. On the other hand, for neural operators, one can even use \emph{non-local} linear operators, such as those defined in terms of an integral kernel. The evaluation of such integral kernels can be performed either with graph kernel networks \cite{li2020neural} or with multipole expansions \cite{li2020multipole}. 

More recently, the authors of \cite{fourierop2020} have proposed using convolution-based integral kernels within neural operators. Such kernels can be efficiently evaluated in the Fourier space, leading to the resulting neural operators being termed as \emph{Fourier Neural Operators} (FNOs). In \cite{fourierop2020}, the authors discuss the advantages, in terms of computational efficiency, of FNOs over the other neural operators mentioned above. Moreover, they present several convincing numerical experiments to demonstrate that FNOs can very efficiently approximate a variety of operators that arise in simulating PDEs. 

However, the theoretical basis for neural operators has not yet been properly investigated. In particular, it is unclear if neural operators such as FNOs are \emph{universal} i.e., if they can approximate a large class of nonlinear infinite-dimensional operators. Moreover in this infinite-dimensional setting, universality does not suffice to indicate computational viability or efficiency as the size of the underlying neural networks might grow exponentially with respect to increasing accuracy, see discussion in \cite{LMK2021} on this issue. Hence in addition to universality, it is natural to ask if neural operators can \emph{efficiently} approximate a large class of operators, such as those arising in the simulation of parametric PDEs. 

The investigation of these questions is the main rationale for the current paper. We focus our attention here on FNOs as they appear to be the most promising of the neural operator based operator learning frameworks. Our main result in this paper is to show that FNOs are \emph{universal} in possessing the ability to approximate a very large class of continuous nonlinear operators. This result highlights the potential of FNOs for operator learning. 

As argued before, a universality result is only a first step and by itself, does not constitute evidence for efficient approximation by FNOs. In fact, we show that in the worst case, the network size might grow exponentially with respect to accuracy, when approximating general operators. Hence, there is a need to derive explicit bounds on the network size in terms of the desired error tolerance. In this context, we consider a concrete computational realization of FNOs, that we term as \emph{pseudospectral FNO} or $\Psi$-FNO (for short). In addition to proving universality for $\Psi$-FNOs, we will suggest a mechanism through which $\Psi$-FNOs can approximate operators arising from PDEs, efficiently. We also derive explicit error bounds for this architecture in approximating PDEs, for two widely used prototypical examples of PDEs i.e, a Darcy type elliptic equation and the incompressible Navier-Stokes equations of fluid dynamics. In particular, we prove that the size of $\Psi$-FNOs in approximating the underlying operators for both these PDEs, under suitable hypotheses, only scales polynomially (log-linearly) in the error. Thus, FNOs can approximate these operators efficiently and these results validate some of the computational findings of \cite{fourierop2020}.  Together, these results constitute the first theoretical justification for the use of FNOs. 

The rest of the paper is organized as follows: in section \ref{sec:2}, we introduce FNOs and state the universality result. We also introduce $\Psi$-FNOs in this section. In section \ref{sec:3}, we show that $\Psi$-FNOs can efficiently approximate operators, stemming from the Darcy-type elliptic equation as well as the incompressible Navier-Stokes equations. In section \ref{sec:4}, we compare FNOs with DeepOnets and the results of the article are discussed in section \ref{sec:5}. The mathematical notation, used in this paper, is summarized in Appendix \ref{sec:glos} and we present all the technical details and proofs in other appendices. 
\section{Approximation by Fourier Neural Operators}
\label{sec:2}
In this section, we present Fourier Neural Operators (FNOs) and discuss their approximation of a class of nonlinear operators specified below:
\subsection{Setting for Operator Learning.}
\begin{setting} 
\label{setting}
We fix a spatial dimension $d\in \N$, and denote by $D \subset \R^d$ a domain in $\R^d$. We consider the approximation of operators $\G: \cA(D;\R^{d_a}) \to \cU(D;\R^{d_u})$, $a \mapsto u := \G(a)$, where the input $a \in \cA(D;\R^{d_a})$, $d_a\in\N$, is a function $a: D \to \R^{d_a}$ with $d_a$ components, and the output $u \in \cU(D;\R^{d_u})$, $d_u\in \N$, is a function $u: D \to \R^{d_u}$ with $d_u$ components. Here $\cA(D;\R^{d_a})$ and $\cU(D;\R^{d_u})$ are Banach spaces (or suitable subsets of Banach spaces). Typical examples of $\cA$ and $\cU$ include the space of continuous functions $C(D;\R^{d_u})$, or Sobolev spaces $H^s(D;\R^{d_u})$ of order $s \ge 0$  (see Appendix \ref{app:notn} for definitions.).  
\end{setting}
Concrete examples for operators $\G$, involving solution operators of PDEs, are given in section \ref{sec:3}. 
\subsection{Neural Operators}
With the above setting \ref{setting} and as defined in \cite{li2020neural}, a neural operator $\cN: \cA(D;\R^{d_a}) \to \cU(D;\R^{d_u})$, $a\mapsto \cN(a)$ is a mapping of the form 
\[
\cN(a) = \cQ \circ \cL_L \circ \cL_{L-1} \circ \dots \circ \cL_{1} \circ \cR(a),
\]
for a given depth $L\in \N$, where $\cR: \cA(D;\R^{d_a}) \to \cU(D;\R^{d_v})$, $d_v \ge d_u$, is a \emph{lifting} operator (acting locally), of the form
\begin{align} \label{eq:no-r}
\cR(a)(x) = R a(x), 
\quad 
R \in \R^{d_v\times d_a},
\end{align}
and $\cQ: \cU(D;\R^{d_v}) \to \cU(D;\R^{d_u})$ is a local \emph{projection} operator, of the form 
\begin{align} \label{eq:no-q}
\cQ(v)(x) = Qv(x), 
\quad
Q \in \R^{d_u \times d_v}.
\end{align}

\begin{remark}
In practice, it has been found that improved results can be obtained if the simple \emph{linear} lifting and projection operators $\cR$ \eqref{eq:no-r} and $\cQ$ \eqref{eq:no-q} are replaced instead by \emph{non-linear} mappings of the form 
\[
\hR(a)(x) = \hat{R}(a(x),x), \quad
\hQ(v)(x) = \hat{Q}(v(x),x),
\]
where $\hat{R}: \R^{d_a}\times D \to \R^{d_v}$ and $\hat{Q}: \R^{d_v} \times D \to \R^{d_u}$ are neural networks with activation function $\sigma$. Our error estimates will rely on the (more restrictive) linear choice of lifting and projection operators, given by \eqref{eq:no-r}, \eqref{eq:no-q}. The linear choice has the theoretical benefit of ensuring \emph{compositionality}, i.e. that a composition of neural operators can again be represented by a neural operator (cf. Lemma \ref{lem:composition}). Despite this technical distinction, we emphasize that all of our error and complexity estimates continue to hold also for neural operators with non-linear lifting and projections, since linear operators can always be approximated by non-linear ones (cp. Lemma \ref{lem:linear-to-nonlinear}). In fact, in the non-linear case, our results imply that $\hat{Q}$, $\hat{R}$ can be chosen to be \emph{shallow} networks.
\end{remark}

In analogy with canonical finite-dimensional neural networks, the layers $\cL_1, \dots, \cL_L$ are non-linear operator layers, $\L_\ell: \cU(D;\R^{d_v}) \to \cU(D;\R^{d_v})$, $v\mapsto \L_\ell(v)$, which we assume to be of the form 
\[
\cL_\ell(v)(x)
=
\sigma\bigg(
W_\ell v(x)
+
b_\ell(x)
 + \big(\cK(a;\theta_\ell) v\big)(x)
\bigg),
\quad
\forall \, x\in D.
\]
Here, the weight matrix $W_\ell \in \R^{d_v\times d_v}$ and bias $b_\ell(x)\in \cU(D;\R^{d_v})$ define an affine pointwise mapping $W_\ell v(x) + b_\ell(x)$. The richness of linear operators in the infinite-dimensional setting can partly be realized by defining the following \emph{non-local} linear operator,  
\[
\cK: \cA\times \Theta \to L\left(\cU(D;\R^{d_v}), \cU(D;\R^{d_v})\right),
\]
that maps the input field $a$ and a parameter $\theta \in \Theta$ in the parameter-set $\Theta$ to a bounded linear operator $\cK(a,\theta): \cU(D;\R^{d_v}) \to \cU(D;\R^{d_v})$, and the non-linear activation function $\sigma: \R \to \R$ is applied component-wise. As proposed in \cite{li2020neural}, the linear operators $\cK(a,\theta)$ are integral operators of the form 
\begin{align} \label{eq:no-k}
\big(\cK(a;\theta) v\big)(x)
=
\int_{D} \kappa_\theta(x,y;a(x),a(y)) v(y) \, dy,
\quad
\forall \, x\in D.
\end{align}
Here, the integral kernel $\kappa_\theta: \R^{2(d+d_a)} \to \R^{d_v\times d_v}$ is a neural network parametrized by $\theta \in \Theta$. Specific examples of the integral kernel \eqref{eq:no-k} include those evaluated with a graph kernel network as in \cite{li2020neural} or with a multipole expansion \cite{li2020multipole}.
\subsection{Fourier Neural Operators}
As defined in \cite{fourierop2020}, Fourier Neural operators (FNOs) are special cases of general neural operators \eqref{eq:no-k}, in which the kernel $\kappa_\theta(x,y;a(x),a(y))$ is of the form $\kappa_\theta = \kappa_\theta(x-y)$. In this case, \eqref{eq:no-k} can be written as a convolution
\begin{align} \label{eq:fno-k}
\big(\cK({\theta}) v\big)(x)
=
\int_{D} \kappa_{\theta}(x-y) v(y) \, dy,
\quad
\forall \, x\in D.
\end{align}
For concreteness, we consider the periodic domain $D =\T^d$ (which we identify with the standard torus $\T^d = [0,2\pi]^d$), although non-periodic, rectangular domains $D$ can also be handled in a straightforward manner.

Given this periodic framework, the convolution operator in \eqref{eq:fno-k} can be computed using the Fourier transform $\cF$ and the inverse Fourier transform $\cF^{-1}$ (see Appendix \ref{app:notn} \eqref{eq:ft} and \eqref{eq:ift} for notation and definitions), resulting in the following equivalent representation of the kernel \eqref{eq:no-k},
\begin{align} \label{eq:fnok}
(\cK(\theta)v)(x)
=
\cF^{-1}\Big( 
P_\theta(k) \cdot \cF(v)(k) 
\Big)(x),
\quad
\forall \, x \in \T^d.
\end{align}
Here, $P_\theta(k) \in \C^{d_v \times d_v}$ is a full matrix indexed by $k \in \Z^d$, and is related to the integral kernel $\kappa_\theta(x)$ in \eqref{eq:fno-k} via the Fourier transform, $P_\theta(k) = \cF(\kappa_\theta)(k)$. Note that we must impose that $P_\theta(-k) = P_\theta(k)^\dagger$ coincides with the Hermitian transpose for all $k\in \Z^d$, to ensure that the image function $(\cK(\theta)v)(x)$ is a real-valued function for real-valued $v(x)$. Consequently, the form of Fourier neural operators (FNOs) for the periodic domain $\T^d$ is that of a mapping $\cN: \cA(D;\R^{d_a}) \to \cU(D;\R^{d_u})$, of the form 
\begin{align} \label{eq:fno}
\cN(a)
:=
\cQ \circ \cL_{L} \circ \cL_{L-1} \circ \dots \circ \cL_{1} \circ \cR(a),
\end{align}
where the lifting and projection operators $\cR$ and $\cQ$ are given by \eqref{eq:no-r} and \eqref{eq:no-q}, respectively, and where the non-linear layers $\cL_\ell$ are of the form 
\begin{align} \label{eq:fno-layer}
\cL_\ell(v)(x)
=
\sigma
\bigg(
W_\ell v(x) + b_\ell(x) + \cF^{-1} \Big(P_{\ell}(k) \cdot \cF(v)(k)\Big)(x)
\bigg).
\end{align}
Here, $W_\ell \in \R^{d_v\times d_v}$ and $b_\ell(x)$ define a pointwise affine mapping (corresponding to weights and biases), and $P_{\ell}: \Z^d \to \C^{d_v\times d_v}$ defines the coefficients of a non-local, linear mapping via the Fourier transform.
\begin{remark}
\label{rem:loc}
The simplest example for a FNO, as defined by \eqref{eq:fno},\eqref{eq:fno-k} is as follows; let $\hN: \R^{d_a} \to \R^{d_u}$ be a canonical finite-dimensional neural network with activation function $\sigma$. We can associate to $\hN$ the mapping $\cN: L^2(\T^d; \R^{d_a}) \to L^2(\T^d;\R^{d_u})$, given by $a(x) \mapsto \hN(a(x))$. We easily observe that $\cN$ is a FNO as we can write it in the form,
\[
\hN = \hQ \circ \hL_L \circ \dots \circ \hL_1 \circ \hR,
\]
where $\hR(y) = Ry$ with $R \in \R^{d_v\times d_a}$, and each layer $\hL_\ell$ is of the form $\hL_\ell(y) = \sigma(W_\ell y + b_\ell)$ for some $W_\ell \in \R^{d_v \times d_v}$, $b_\ell \in \R^{d_v}$, with $\hQ$ being an affine output layer of the form $\hQ(y) = Qy + q$ with $Q\in \R^{d_u\times d_v}$, $q\in \R^{d_u}$. Replacing the input $y$ by a function $v(x)$, these layers clearly are a special case of the FNO lifting layer \eqref{eq:no-r}, the non-linear layers \eqref{eq:fno-layer} (with $P_\ell \equiv 0$ and constant bias $b_\ell(x)\equiv b_\ell$), and the projection layer \eqref{eq:no-q}. Thus, any finite-dimensional neural network can be identified with a FNO as defined above. 
\end{remark}

For the remainder of this work, we make the following assumption,

\begin{assumption}[Activation function]
Unless explicitly stated otherwise, the activation function $\sigma: \R \to \R$ in \eqref{eq:fno-layer} is assumed to be non-polynomial, (globally) Lipschitz continuous and $\sigma \in C^3$.
\end{assumption}

\subsection{Universal Approximation by FNOs}
Next, we will show that FNOs \eqref{eq:fno} are \emph{universal} i.e., given a large class of operators, as defined in setting \ref{setting}, one can find an FNO that approximates it to desired accuracy. To be more precise, we have the following theorem. 
\begin{theorem}[Universal approximation]
\label{thm:universal}
Let $s,s'\ge 0$. Let $\G: H^{s}(\T^d;\R^{d_a}) \to H^{s'}(\T^d;\R^{d_u})$ be a continuous operator. Let $K\subset H^{s}(\T^d;\R^{d_a})$ be a compact subset. Then for any $\epsilon > 0$, there exists a FNO $\cN: H^{s}(\T^d;\R^{d_a}) \to H^{s'}(\T^d;\R^{d_u})$, of the form \eqref{eq:fno}, continuous as an operator $H^s\to H^{s'}$, such that 
\[
\sup_{a\in K}
\Vert 
\G(a) - \cN(a)
\Vert_{H^{s'}}
\le \epsilon.
\]
\end{theorem}

\begin{proof}[Sketch of proof]
The detailed proof of this universal approximation theorem is provided in Appendix \ref{app:pfut} and we outline it here. For notational simplicity, we set $d_a=d_u=1$,
and first observe the following lemma, proved in Appendix \ref{app:pfut0}:
\begin{lemma} \label{lem:fut0}
Assume that the universal approximation Theorem \ref{thm:universal} holds for $s'=0$. Then it holds for arbitrary $s'\ge 0$.
\end{lemma}
The main objective is thus to prove Theorem \ref{thm:universal} for the special case $s'=0$; i.e. given a continuous operator $\cG: H^s(\T^d) \to L^2(\T^d)$, $K\subset H^s(\T^d)$ compact, and $\epsilon>0$, we wish to construct a FNO $\cN: H^s(\T^d) \to L^2(\T^d)$, such that ${\sup_{a\in K}\Vert \cG(a) - \cN(a) \Vert_{L^2} \le \epsilon}$. 

To this end, we start by defining the following operator,
\begin{equation}
    \label{eq:defGN}
\G_N: H^{s}(\T^d) \to L^2(\T^d), \quad
\G_N(a) := P_N\G(P_Na),
\end{equation}
with $P_N$ being the orthogonal Fourier projection operator, defined in Appendix \ref{app:notn} \eqref{eq:proj}. Thus, $\G_N$ can be thought of loosely as the \emph{Fourier projection} of the continuous operator $\G$.

Next, we can show that for any given $\epsilon > 0$, there exists $N\in \N$, such that 
\begin{align} \label{eq:GN}
\Vert \G(a) - \G_N(a) \Vert_{L^2} \le \epsilon, \quad
\forall \, a \in K.
\end{align}
Thus, the proof boils down to finding a FNO \eqref{eq:fno} that can approximate the operator $\G_N$ to any desired accuracy. 

To this end, we introduce a set of Fourier wavenumbers $k\in \cK_N$, by 
\begin{align} 
\label{eq:k}
\cK_N := \set{k\in \Z^d}{|k|_\infty \le N},
\end{align}
and define a \emph{Fourier conjugate} or \emph{Fourier dual} operator of the form
$\hat{\G}_N: \C^{\cK_N} \to \C^{\cK_N}$, 
\begin{equation}
\label{eq:Fconj}
\hat{\G}_N(\hat{a}_k)
:= 
\cF_N \left( \G_N \left( \Re\left(\cF^{-1}_N(\hat{a}_k)\right)\right)\right),
\end{equation}
such that the identity
\begin{align} \label{eq:decompG}
\G_N(a) = \cF_N^{-1} \circ \hat{\G}_N \circ \cF_N(P_Na),
\end{align}
holds for all \emph{real-valued} $a\in L^2(\T^d)$. Here, $\cF_N$ is the discrete Fourier transform and $\cF_N^{-1}$ is the discrete inverse Fourier transform, with both being defined in Appendix \ref{app:notn} \eqref{eq:dft} and \eqref{eq:dift}, respectively.

The next steps in the proof are to leverage the natural decomposition of the projection $\G_N$ in \eqref{eq:decompG} in terms of the discrete Fourier transform $\cF_N\circ P_N$, the discrete inverse Fourier transform $\cF^{-1}_N$ and the Fourier conjugate operator $\hat{\G}_N$ and approximate each of these operators by Fourier neural operators. 

We start by denoting,
\begin{align}\label{eq:RKN}
\R^{2\cK_N} = \left( \R^2 \right)^{\cK_N} (\simeq \mathbb{C}^{\cK_N}),
\end{align}
as the set consisting of coefficients $\{(v_{1,k},v_{2,k}) \}_{k\in \cK_N}$, where $v_{\ell,k} \in \R$ are indexed by a tuple $(\ell,k)$, $\ell \in \{1,2\}$, $k \in \cK_N$, and interpreting the operator $\cF_N\circ P_N$ as a mapping $\cF_N\circ P_N: a \mapsto \{(\Re(\hat{a}_k), \Im(\hat{a}_k))\}_{|k|\le N}$, with input $a \in L^2(\T^d)$ and the output $\{\Re(\hat{a}_k), \Im(\hat{a}_k)\}_{|k|\le N} \in \R^{2\cK_N}$ is viewed as a \emph{constant} function in  $L^2(\T^d;\R^{2 \cK_N})$. The approximation of this operator is a straightforward consequence of the following Lemma, proved in Appendix \ref{app:pfut1},
\begin{lemma} \label{lem:fno-ft}
Let $B > 0$ and $N \in \N$ be given. For all $\epsilon> 0$, there exists a FNO $\cN: L^2(\T^d) \to L^2(T^d;\R^{2\cK_N})$, $v \mapsto \{ \cN(v)_{\ell,k} \}$, with \emph{constant} output functions (constant as a function of $x\in \T^d$), and such that 
\[
\left.
\begin{aligned}
\Vert 
\Re(\hat{v}_k) - \cN(v)_{1,k}
\Vert_{L^\infty} &\le \epsilon
\\
\Vert 
\Im(\hat{v}_k) - \cN(v)_{2,k}
\Vert_{L^\infty} &\le \epsilon
\end{aligned}
\right\}
\quad \forall \, k\in \Z^d, \; |k|_\infty \le N,
\]
for all $\Vert v \Vert_{L^2} \le B$, and where $\hat{v}_k\in \C$ denotes the $k$-th Fourier coefficient of $v$.
\end{lemma}

In the next step, we approximate the (discrete) inverse Fourier transform $\cF_N^{-1}$ by an FNO. We recall that FNOs act on functions rather than on constants. Therefore, to connect $\cF_N^{-1}$ and FNOs, we are going to interpret the mapping 
\[
\cF_N^{-1}: [-R,R]^{2\cK_N} \subset \R^{2\cK_N} \to L^2(\T^d),
\]
as a mapping 
\begin{align*}
\cF_N^{-1}:
\left\{
\begin{aligned}
&L^2(\T^d; [-R,R]^{2\cK_N}) \to L^2(\T^d),
\\
&\{\Re(\hat{v}_k), \Im(\hat{v}_k)\}_{|k|\le N} \mapsto v(x),
\end{aligned}
\right.
\end{align*}
where the input $\{\Re(\hat{v}_k), \Im(\hat{v}_k)\}_{|k|\le N} \in [-R,R]^{2\cK_N}$ is identified with a \emph{constant} function in $L^2(\T^d;[-R,R]^{2 \cK_N})$. The existence of a FNO of the form \eqref{eq:fno} that can approximate \eqref{eq:ift0} to desired accuracy is a consequence of the following lemma, proved in Appendix \ref{app:pfut2},
\begin{lemma} \label{lem:fno-ift}
Let $B > 0$ and $N \in \N$ be given. For all $\epsilon> 0$, there exists a FNO $\cN: L^2(\T^d;\R^{2\cK_N}) \to L^2(\T^d)$, such that for any $v \in L^2_N(\T^d)$ with $\Vert v \Vert_{L^2} \le B$, we have
\[
\left\Vert 
v
-
\cN\left(
w
\right)
\right\Vert_{L^2} 
\le \epsilon,
\]
where $w(x) := \left\{(\Re(\hat{v}_k), \Im(\hat{v}_k))\right\}_{k\in \cK_N}$, i.e. $w\in L^2(\T^d;\R^{2\cK_N})$ is a constant function collecting the real and imaginary parts of the Fourier coefficients $\hat{v}_k$ of $v$.
\end{lemma}
Finally, by setting $\hat{K} := \cF_N(P_N K) \subset \C^{\cK_N}$ as the (compact) image of $K$ under the continuous mapping $\cF_N \circ P_N: L^2(\T^d) \to \C^{\cK_N}$ and identifying $\C^{\cK_N} \simeq \R^{2\cK_N}$, where $\hat{v}_{1,k} := \Re(\hat{v}_k)$ and $\hat{v}_{2,k} := \Im(\hat{v}_k)$ for $k\in \cK_N$, we can view $\hat{\G}_N$ as a continuous mapping
\[
\hat{\G}_N: \hat{K} \subset \R^{2\cK_N} \to \R^{2\cK_N},
\]
on a compact subset. Hence, by the universal approximation theorem for finite-dimensional neural networks \cite{BAR1,HOR1}, one can readily show that there exists an FNO, with only 
\emph{local} weights (see remark \ref{rem:loc}), which will approximate this continuous mapping $\hat{G}_N$ on compact subsets to desired accuracy. 

Hence, each of the component operators of the decomposition \eqref{eq:decompG} can be approximated to desired accuracy by FNOs and the universal approximation theorem follows by composing these FNOs and estimating the resulting error, with details provided in appendix \ref{app:pfut}. 
\end{proof}

In the following theorem, we will show that the universal approximation theorem \ref{thm:universal} can be extended to include operators defined on function spaces with Lipschitz domains. In fact, the Lipschitz condition can be relaxed to include all locally uniform domains using ideas from \cite{rogers2006degree}; we will, however, not pursue this for simplicity of the exposition. We show that one can construct a period extension of the input function and a FNO so that the restriction of the FNO's periodic output to the domain of interest gives a suitable approximation to any continuous operator. Similar ideas have been pursued in the design of numerical algorithms for solving PDEs and usually go by the name of \textit{Fourier continuations} \cite{bruno2010highorder1,bruno2010highorder2}. A major challenge for these methods is designing a suitable periodic function whose restriction gives the solution of interest. We show that FNOs can learn the output representation automatically.

\begin{theorem}
\label{thm:universal_extend}
Let \(s, s' \geq 0\) and \(\Omega \subset [0,2\pi]^d\) be a domain with Lipschitz boundary. Let $\G: H^{s}(\Omega;\R^{d_a}) \to H^{s'}(\Omega;\R^{d_u})$ be a continuous operator. Let $K\subset H^{s}(\Omega;\R^{d_a})$ be a compact subset. Then
there exists a continuous, linear operator \(\cE : H^s(\Omega;\R^{d_a}) \to H^s(\T^d;\R^{d_a}) \) such that \(\cE(a)|_\Omega = a\) for all \(a \in H^s(\Omega;\R^{d_a})\). Furthermore, 
for any $\epsilon > 0$, there exists a FNO $\cN: H^{s}(\T^d;\R^{d_a}) \to H^{s'}(\T^d;\R^{d_u})$ of the form \eqref{eq:fno}, such that 
\[
\sup_{a\in K}
\Vert 
\G(a) - \cN \circ \cE (a)|_\Omega
\Vert_{H^{s'}}
\le \epsilon.
\]
\end{theorem}
\begin{proof}
 Since \(\Omega\) is open we have that \({\text{dist}(\Omega, \partial [0,2\pi]^d) > 0}\) hence the conclusion of Lemma \ref{lem:periodic_extension} in Appendix \ref{app:notn}
follows with the hypercube \(B = [0,2\pi]^d\), in particular,
there exists a continuous, linear operator \(\cE : H^s(\Omega;\R^{d_a}) \to H^s([0,2\pi]^d;\R^{d_a})\) such that \(\cE(a)|_\Omega = a\) and \(\cE(a)\) is periodic on \([0,2\pi]^d\) for all \(a \in H^s(\Omega;\R^{d_a})\). Therefore
\(\cE : H^s(\Omega;\R^{d_a}) \to H^s(\T^d;\R^{d_a})\). Similarly, we can construct an extension operator $\cE': H^{s'}(\Omega;\R^{d_u}) \to H^{s'}(\T^d;\R^{d_u})$.

We can then associate to $\G: H^s(\Omega; \R^{d_a}) \to H^{s'}(\Omega;\R^{d_u})$ another continuous operator $\overline{\G}: H^s(\T^d; \R^{d_a}) \to H^{s'}(\T^d;\R^{d_u})$, by defining $\overline{\G}(a) := \cE' \circ \cG \circ \cR(a)$. Here $\cR(a) := a|_{\Omega}$ denotes the restriction to $\Omega$ which is clearly linear and continuous.
By the continuity of $\cE$, we have that $K' := \cE(K)$ is compact in $H^s(\T^d;\R^{d_a})$. By the universal approximation theorem \ref{thm:universal}, for any $\epsilon > 0$, there exists a FNO $\overline{\cN}: H^s(\T^d; \R^{d_a}) \to H^{s'}(\T^d;\R^{d_u})$, such that 
\[
\sup_{a'\in K'} \Vert \overline{\G}(a') - \overline{\cN}(a') \Vert_{H^{s'}} \le \epsilon.
\]
But then, using the fact that $\cR \circ \cE = \Id$, $\cR\circ \cE' = \Id$, the mapping $\cN: H^s(\Omega;\R^{d_a}) \to H^{s'}(\Omega;\R^{d_u})$, 
given by $\cN := \cR \circ \overline{\cN} \circ \cE$, satisfies
\begin{align*}
\sup_{a\in K} \Vert {\G}(a) - {\cN}(a) \Vert_{H^{s'}} 
&=
\sup_{a\in K} \Vert \cR \circ \cE' \circ {\G} \circ \cR \circ \cE(a) - \cR \circ \overline{\cN} \circ \cE(a) \Vert_{H^{s'}} 
\\
&=
\sup_{a\in K} \Vert \cR \circ \overline{\G} \circ \cE(a) - \cR \circ \overline{\cN} \circ \cE(a) \Vert_{H^{s'}} 
\\
&\le
\sup_{a\in K} \Vert \overline{\G} \circ \cE(a) - \overline{\cN} \circ \cE(a) \Vert_{H^{s'}}
\\
&=
\sup_{a'\in K'} \Vert \overline{\G}(a') - \overline{\cN}(a') \Vert_{H^{s'}}
\\
&\le \epsilon.
\end{align*}
\end{proof}

\begin{remark}
The form of the universal approximation theorem \ref{thm:universal} stated above, shows that any continuous operator $\G: H^s \to H^{s'}$ can be approximated to arbitrary accuracy by a FNO, \emph{on a given compact subset $K\subset H^s$}. The restriction to compact subsets may not always be very natural. For example, to train FNOs in practice, it might be more convenient to draw training samples from a measure $\mu$ such as the law of a Gaussian random field, which does not have compact support. Furthermore, the operator $\G$ may not always be continuous. To address these issues, one can follow the recent paper \cite{LMK2021}, where the authors prove a more general version for the universal approximation of operators for DeepOnets; for any input measure $\mu$, and a Borel measurable operator $\G$, such that $\int \Vert \G(a) \Vert_{L^2}^2 \, d\mu(a) < \infty$, it is shown that for any $\epsilon > 0$, there exists a DeepOnet $\cN(a) \approx \G(a)$ such that 
\[
\int \Vert \G(a) - \cN(a) \Vert_{L^2}^2 \, d\mu(a) < \epsilon.
\]
In particular, there are no restrictions on the topological support of $\mu$. The result of \cite{LMK2021} was for the alternative operator learning framework of DeepOnets, but the ideas and the proof can be analogously extended to FNOs. 
\end{remark}
\subsection{{$\Psi$}-Fourier neural operators.}
In practice, one needs to compute the FNO, of form \eqref{eq:fno}, both during training as well as for the evaluation of the neural operator. Thus, given any input function $a$, one should be able to readily calculate the FNO $\cN(a)$, requiring the efficient computation of the Fourier transform $\cF$ \eqref{eq:ft} and the inverse Fourier transform $\cF^{-1}$ \eqref{eq:ift}. In general, this is not possible as evaluating the Fourier transform \eqref{eq:ft} entails computing an integral exactly. Therefore, approximations are necessary to realize the action of FNOs on functions. Following \cite{fourierop2020}, one can efficiently approximate the Fourier transform and its inverse by the discrete Fourier transform \eqref{eq:dft} and the discrete inverse Fourier transform \eqref{eq:dift}, respectively. This amounts to performing a pseudo($\Psi$)-spectral Fourier projection between successive layers of the FNO and leading to the following precise definition, 
\begin{definition}[$\Psi$-FNO] \label{def:pfno}
A \define{$\Psi$-FNO} (or $\Psi$-spectral FNO) is a mapping 
\[
\cN: \cA(\T^d;\R^{d_a}) \to \cU(\T^d;\R^{d_u}), \quad a \mapsto \cN(a),
\]
of the form 
\begin{align} \label{eq:form-pfno}
\cN(a) 
=
\cQ \circ \cI_N \circ \cL_L \circ \cI_N  \circ \dots \circ \cL_1 \circ \cI_N \circ \cR(a),
\end{align}
where $\cI_N$ denotes the pseudo-spectral Fourier projection onto trigonometric polynomials of degree $N\in \N$ \eqref{eq:IN}, the lifting operator $\cR: \cA(\T^d;\R^{d_a}) \to \cU(\T^d;\R^{d_v})$, the projection $\cQ: \cU(\T^d;\R^{d_v}) \to \cU(\T^d;\R^{d_u})$ are defined as in \eqref{eq:no-r}, \eqref{eq:no-q}, and the non-linear layers $\cL_{\ell}$, for $\ell=1,\dots, N$, are of the form 
\begin{align*}
\cL_\ell(v)(x) &= 
\sigma\bigg(
W_\ell v(x) + b_{\ell}(x) + \cF^{-1}\Big(P_\ell(k) \cdot \cF(v)(k)\Big)(x)
\bigg).
\end{align*}
Here, $W_\ell\in \R^{d_v\times d_v}$ and $b_{\ell}(x) \in \cU(\T^d;\R^{d_v})$ define a pointwise affine mapping $v \mapsto W_\ell v(x) + b_{\ell}(x)$, and the coefficients $P_\ell(k) \in \R^{d_v\times d_v}$ $(k\in \cK_N)$ define a (non-local) convolution operator via the Fourier transform.

Note that a $\Psi$-FNO $\cN$ is uniquely defined, as an operator, by its restriction to the finite-dimensional subspace $L^2_N(\T^d; \R^{d_a}) \subset \cA(\T^d;\R^{d_a})$ (see Appendix \ref{app:notn} for the definition of $L^2_N$). Furthermore, we have that the image $\im(\cN) \subset L^2_N(\T^d;\R^{d_u})$. To indicate that a $\Psi$-FNO is of the form \ref{eq:form-pfno}, for some $N\in \N$, we shall thus more simply say that ``$\cN: L^2_N(\T^d;\R^{d_a})\to L^2_N(\T^d;\R^{d_u})$ is a $\Psi$-FNO''.
\end{definition}
At the level of numerical implementation, a {$\Psi$-FNO} can be naturally identified with a \emph{finite-dimensional} mapping
\[
\hN: \R^{d_a \times \cJ_N} \to \R^{d_u \times \cJ_N}, \quad \bm{a} \mapsto \cN(\bm{a}),
\]
with input $\bm{a} = \{a_j\}_{j\in \cJ_N} \in \R^{d_a\times \cJ_N}$ corresponding to the point-values $a_j = a(x_j)$ on the grid $\{x_j\}_{j\in \cJ_N}$, and $\cJ_N:=\{0,\dots, 2N\}^d$. Here, $\hN$ is of the form 
\[
\hN(\bm{a}) 
=
\hQ \circ \hL_L \circ \hL_{L-1} \circ \dots \circ \hL_1 \circ \hR(\bm{a}),
\]
where the lifting operator $\hR: \R^{d_a\times \cJ_N} \to \R^{d_v\times \cJ_N}$, $\bm{a} \mapsto \hR(\bm{a})$, the projection $\hQ: \R^{d_v\times \cJ_N} \to \R^{d_u \times \cJ_N}$, $\bm{v}\mapsto \hQ(\bm{v})$, are given by
\begin{alignat*}{2}
\hR(\bm{a}) &= \{Ra_j\}_{j\in \cJ_N}, && (R \in \R^{d_v\times d_a}),
\\
\hQ(\bm{v}) &= \{ Qv_j\}_{j\in \cJ_N},\quad  && (Q \in \R^{d_u\times d_v}), 
\end{alignat*}
and the non-linear layers $\hL_{\ell}$, for $\ell=1,\dots, N$, are of the form 
\begin{align}
\label{eq:hL}
\hL_\ell(\bm{v})_j &= 
\sigma\bigg(
W_\ell v_j + b_{\ell,j} + \cF_N^{-1}\Big(P_\ell(k) \cdot \cF_N(\bm{v})(k)\Big)_j
\bigg)
\end{align}
for $j\in \cJ_N$. Here, $W_\ell\in \R^{d_v\times d_v}$, $b_{\ell,j} = b_\ell(x_j) \in \R^{d_v \times \cJ_N}$ defines a pointwise affine mapping $W_\ell v_j + b_{\ell,j}$, the coefficients $P_\ell(k) \in \C^{d_v\times d_v}$ $(k\in \cK_N)$ satisfy the Hermitian conjugacy condition $P_\ell(-k) = P_\ell(k)^\dagger$ and define a (non-local) convolution operator via the discrete Fourier transform, and the non-linear activation function $\sigma: \R \to \R$ is extended componentwise to a function $\R^{d_v\times \cJ_N} \to \R^{d_v\times \cJ_N}$. Comparing $\cN$ with the corresponding discretization $\hN$, it is easy to see that 
\[
\hN(\{a(x_j)\}_{j\in \cJ_N})_j 
=
\cN(a)(x_j), 
\quad
\forall \, j\in \cJ_N.
\]
In particular, this implies that $\cN(a)(x)$ can in practice be computed for any $x\in \T^d$ via the Fourier interpolation of the grid values $\hN(\{a(x_j)\})_{j\in \cJ_N}$. In contrast to general FNOs, $\Psi$-FNOs therefore allow for efficient numerical implementation. Furthermore, the discrete (inverse) Fourier transforms in each hidden layer in \eqref{eq:hL} can be very efficiently computed using the fast Fourier transform (FFT).  

The above discussion also leads to a very natural definition of the size of a $\Psi$-FNO below:
\begin{definition}[Depth, width, lift and size]
The \define{depth} and \define{width} of a $\Psi$-FNO $\cN$ (cp. Definition \ref{def:pfno}), are defined by
\[
\depth(\cN) := L, 
\quad
\width(\cN) := d_v |\cJ_N| = d_v |\cK_N| = (2N+1)^d d_v.
\]
We refer to the dimension $d_v$, as the \define{lift} of $\cN$, i.e. we set 
\[
\lift(\cN) := d_v.
\]
The \define{size} of a $\Psi$-FNO $\cN$ is defined as the total number of degrees of freedom in a $\Psi$-FNO. A simple calculation shows that 
\[
\size(\cN)
=
\underbrace{d_u d_v}_{\size(\cQ)} 
+
L\underbrace{ \left(d_v^2 + d_v |\cJ_N| + d_v^2 |\cJ_N| \right)}_{\size(\cL_\ell)}
+
\underbrace{d_a d_v}_{\size(\cR)}.
\]
\end{definition}
The precise size of a $\Psi$-FNO will not be of any particular relevance for our asymptotic complexity estimates. Instead, we will usually content ourselves with the simple estimate
\[
\size(\cN)
\lesssim
\depth(\cN) \, \width(\cN)\, \lift(\cN),
\]
where we assume that $\max(d_a,d_u) \le d_v$; under this condition, the above estimate follows from the fact that $\size(\cN)\sim L d_v^2 |\cJ_N|$.

Given our discussion, it is natural to ask whether any FNO $\hN = \cQ \circ \cL_{L} \circ \cL_{L-1} \circ \dots \circ \cL_{1} \circ \cR$ can be approximated to arbitrary accuracy by an associated $\Psi$-FNO $\cN: L^2_N \to L^2_N$, 
\[
\cN 
=
\cQ \circ \cI_N \circ \cL_L \circ \cI_N  \circ \dots \circ \cL_1 \circ \cI_N \circ \cR,
\] 
for sufficiently large $N \in \N$? An affirmative answer can be given for a natural class of FNOs of \emph{finite width}, defined as follows.

\begin{definition}
A FNO $\hN: \cA(\T^d;\R^{d_a}) \to \cU(\T^d;\R^{d_u})$ is said to be \define{of finite width}, if $\hN$ is a composition $\hN = \cQ \circ \cL_{L} \circ \dots \circ \cL_{1} \circ \cR$, with layers $\cL_{\ell}$ of the form \eqref{eq:fno-layer}, and for which there exists a ``width'' $W\in \N$, such that the Fourier multiplier $P_\ell(k) \equiv 0$, for $|k|_\infty > W$. 
\end{definition}

We can now state the following theorem, which shows that $\Psi$-FNOs $\cN$ provide an arbitrarily close approximation of a given FNO $\hN$:

\begin{theorem} \label{thm:fno-pfno}
Assume that the activation function $\sigma \in C^{m}$ is globally Lipschitz continuous. 
Let $\hN: H^s(\T^d;\R^{d_a}) \to L^2(\T^d;\R^{d_u})$ be a FNO of finite width, with $s>d/2$. and assume that $m>s$. Then for any $\epsilon, B > 0$, there exists $N\in \N$ and a $\Psi$-FNO $\cN: L^2_N(\T^d;\R^{d_a}) \to L^2_N(\T^d;\R^{d_u})$, such that 
\[
\sup_{\Vert a \Vert_{H^s} \le B}
\Vert 
\hN(a) - \cN(a) 
\Vert_{L^2} \le \epsilon.
\]
\end{theorem}

For the proof, we refer to Appendix \ref{app:fno-pfno}. In particular, the last theorem implies an extension of the universal approximation theorem \ref{thm:universal} to \emph{$\Psi$-FNOs}, provided that the input functions have sufficient regularity for the pseudo-spectral projection \eqref{eq:IN} to be well-defined:

\begin{theorem}[Universal approximation for $\Psi$-FNOs] \label{thm:pfno-universal}
Let $s > d/2$, and let $s'\ge 0$. Let $\G: H^{s}(\T^d;\R^{d_a}) \to H^{s'}(\T^d;\R^{d_u})$ be a continuous operator. And let $K\subset H^{s}(\T^d;\R^{d_a})$ be a compact subset. Then for any $\epsilon > 0$, there exists $N\in \N$ and a $\Psi$-FNO $\cN: L^2_N(\T^d;\R^{d_a}) \to L^2_N(\T^d;\R^{d_u})$, such that 
\[
\sup_{a\in K}
\Vert 
\G(a) - \cN(a)
\Vert_{H^{s'}}
\le \epsilon.
\]
\end{theorem}

\begin{proof}
Similar to the proof of the universal approximation theorem for FNOs, we again note that the general case $s'\ge 0$ can be deduced from the statement of Theorem \ref{thm:pfno-universal} for the special case $s'=0$. This is the content of the following lemma, whose proof is provided in Appendix \ref{app:pfnout0}:
\begin{lemma} \label{lem:pfnout0}
Assume that Theorem \ref{thm:pfno-universal} holds for $s'=0$. Then it holds for arbitrary $s'\ge 0$.
\end{lemma}

The special case $s'=0$ follows immediately from Theorem \ref{thm:fno-pfno} and the observation that the FNO approximation constructed in the proof of the universal approximation theorem for FNOs, Theorem \ref{thm:universal}, has finite width. 
\end{proof}
\subsubsection{Structure and properties of $\Psi$-FNOs}
\label{sec:fnoprop}
We conclude this section by pointing out some aspects of the structure of $\Psi$-FNOs \eqref{eq:form-pfno} that will be relevant in the following. To start with, we can simplify $\Psi$-FNOs by viewing them in terms of two types of layers. which we will refer to as $\sigma$- and $\cF$-layers, respectively. A \define{$\sigma$-layer $\cL = \cL_\sigma$} of a $\Psi$-FNO is a \define{local, non-linear} layer of the form $\cL_\sigma(v)(x)=\cI_N\sigma\left(A \cI_N v(x) + b\right)$, or, in the numerical implementation (cp. \eqref{eq:hL})
\[
\cL_\sigma(v)_j = \sigma\left(A v_j + b_j\right)
,
\quad
\forall\, j\in \cJ_N,
\]
with $A \in \R^{d_v\times d_v}$, and $b_j\in \R^{\cJ_N \times d_v}$ defining an affine mapping. A \define{$\cF$-layer $\cL = \cL_{\cF}$} of a $\Psi$-FNO is a \define{non-local, linear} layer of the form $\cL_{\cF}(v)(x)=\cF^{-1}(P(k)\cdot \cF(\cI_N v)(k))(x)$, which in a practical implementation corresponds to
\[
\cL_{\cF}(v)_j
=
\cF^{-1}_N\Big(P(k)\cdot \cF_N(v)(k)\Big)_j
,
\quad
\forall \, j\in \cJ_N,
\]
where $P: \cK_N \to \C^{d_v\times d_v}$ is a collection of complex weights, with $P(-k) = P(k)^\dagger$ the Hermitian transpose of $P(k)$, and $\cF_N$ ($\cF_N^{-1}$) denotes the discrete (inverse) Fourier transform.

The main point of these definitions is that each $\Psi$-FNO can be decomposed into a finite number of $\sigma$-layers and $\cF$-layers, and that the converse is also true; i.e. any composition of $\sigma$-layers and $\cF$-layers can be represented by a $\Psi$-FNO. These statements are made precise in a series of technical Lemmas, which are stated and proved in Appendix \ref{app:fnoprop}.

\section{Approximation of PDEs by {$\Psi$}-FNOs}
\label{sec:3}
We have shown in the previous section that FNOs \eqref{eq:fno} as well as their computational realizations ($\Psi$-FNOs \eqref{eq:form-pfno}) are universal i.e., they approximate any continuous operator, defined in the setting \ref{setting}, to desired accuracy. However, as repeatedly discussed in the introduction, universality alone does not suffice to claim that FNOs can approximate operators efficiently. In particular, it could happen that the size of the FNO is  unfeasibly large to ensure a given accuracy of the approximation. That this is indeed the case is made precise in the following remark.
\begin{remark}
\label{rem:comp}
We observe from the proof of Theorem \ref{thm:universal} that the desired FNO, approximating the operator $\G$, is constructed as  $\cN_{\mathrm{IFT}} \circ \hN \circ \cN_{\mathrm{FT}}$, with $\cN_{\mathrm{FT}},\cN_{\mathrm{IFT}}$ approximating the Fourier and Inverse Fourier transforms, respectively, whereas $\hN: \R^{2\cK_N} \to \R^{2\cK_N}$ is a canonical finite-dimensional neural network approximation of the ``Fourier conjugate operator'' \eqref{eq:Fconj}: $\hG_N: \R^{2\cK_N} \to \R^{2\cK_N}$. We note that $N$ herein has to be chosen sufficiently large in order to yield the desired error tolerance of $\epsilon$. By Theorem \ref{thm:psest}, this depends on the smoothness of the input space, i.e., if the input $a\in K \subset H^s$, for some $s > 0$, then we need to choose $N$ such that $N^{-s} \sim \epsilon$. Further assuming that the mapping $\G$ is Lipschitz continuous, implies that the Fourier conjugate operator $\hG$ is also Lipschitz continuous as a mapping from $\R^{2\cK_N}$ to $\R^{2\cK_N}$. Hence, neural network approximation results, such as those of \cite{Yar1} for ReLU activations or \cite{DLM1} for tanh activations, yield that the width of the approximating neural network $\hN$ scales as  $\width(\hN) \gtrsim \epsilon^{-D}$, where $D$ is the dimension of the domain of $\hG_N$. In the present case, we have $D = |\cK_N| \sim N^d \sim \epsilon^{-d/s}$, yielding that 
\begin{equation}
    \label{eq:comp}
\width(\hN) \gtrsim \epsilon^{-\epsilon^{-d/s}}. 
\end{equation}
This scaling represents a \emph{super-exponential} growth in the size of the FNO $\cN$, with respect to the error $\epsilon$, incurred in approximating the underlying operator $\G$.
\end{remark}
Given the above remark, we infer that in the worst case, a FNO approximating a generic Lipschitz continuous operator $\G$, can require extremely large sizes to achieve the desired accuracy, making it unfeasible in practice. The same holds for $\Psi$-FNOs of the form \eqref{eq:form-pfno}. This super-exponential growth appears as a form of \emph{curse of dimensionality} i.e., exponential growth of complexity (measured here in the size of the FNO), with respect to the error. 

Hence, it is reasonable to ask how these extremely pessimistic complexity bounds on FNOs ($\Psi$-FNOs), can be reconciled to their robust numerical performance for approximating PDEs, as reported in \cite{fourierop2020}. The rest of the section investigates this fundamental question. 

The starting point of our explanation for the robustness of FNOs in approximating PDEs is the observation that operators which arise in the context of PDEs have a special structure and are not merely generic continuous operators mapping one infinite-dimensional function space to another. To see this, we point out that many time-dependent PDEs arising in physics can be written in the general abstract form,
\begin{align}\label{eq:intuition}
\partial_t u + \nabla \cdot F(u, \nabla u) = 0,
\end{align}
where for any $(t,x) \in [0,T]\times D \subset \R^d$, $u(t,x) \in \R^{d_u}$ is a vector of physical quantities, describing e.g. density, velocity or temperature of a fluid or other material at a given point $x\in D$ in the domain $D$ and at time $t\in [0,T]$. Equation \eqref{eq:intuition} describes the general form of a conservation law for the physical quantities $u$ with a flux function $F(u,\nabla u)$, which is typically \emph{non-linear}, and can e.g. represent advection or diffusion terms. The flux function $F(u,\nabla u)$ may also depend on $u$ in a \emph{non-local} manner. For example, for the incompressible Navier-Stokes equations in $\R^d$, 
\[
u(x,t) = (u_1(x,t),\dots, u_d(x,t)) \in \R^d,
\]
represent the fluid velocity at $(x,t)$, and the flux is defined by 
\[
F(u,\nabla u) = -u\otimes u - p + \nu \nabla u,
\]
where $p = p(u)$ depends on $u$ in a non-local manner:
\[
p = \cR : (u\otimes u), \quad \cR := (-\Delta)^{-1} (\nabla \otimes \nabla),
\]
where $\cR$ is a (non-local) operator closely related to the Riesz transform. 

A popular numerical method for time-dependent PDEs, of the form \eqref{eq:intuition}, particularly on periodic domains $D = \T^d$, is the \emph{pseudo-spectral method} \cite{Canuto2007}, wherein \eqref{eq:intuition} is discretized as,
\begin{align} \label{eq:intu0}
\partial_t u_N + \nabla \cdot \cI_N F(u_N, \nabla u_N) = 0,
\end{align}
where $u_N \in L^2_N$ is a trigonometric polynomial of degree $\le N$.

The resulting system of ODEs \eqref{eq:intu0} can be further discretized in time using a time-marching scheme. For simplicity, the forward Euler discretization with time step $\dt$ leads to,
\begin{align} \label{eq:intu}
u^{n+1}_N = u^n_N - \dt\nabla \cdot \cI_N F(u^n_N, \nabla u^n_N).
\end{align}
One might prove that the system \eqref{eq:intu} provides a convergent approximation for the underlying time-dependent PDE \eqref{eq:intuition} for many different choices of the flux $F$. 
In order to connect the approximation \eqref{eq:intu} with FNOs, we decompose the right hand side of \eqref{eq:intu} as, 
\begin{alignat*}{2}
u^n_N
&\,\overset{\cR}{\longmapsto}\,
\begin{Bmatrix}
u^n_N \\[3pt]
u^n_N
\end{Bmatrix}
\,\overset{\cF}{\longmapsto }\,
\begin{Bmatrix}
u^n_N \\[3pt]
\nabla u^n_N
\end{Bmatrix}
&&
\\
&
\,\overset{\sigma}{\mapsto }\,
\dots
\,\overset{\sigma}{\mapsto }\,
\begin{Bmatrix}
u^n_N \\[3pt]
F^\sigma_N(u^n_N,\nabla u^n_N)
\end{Bmatrix}
&&
\\
&\,\overset{\cF}{\longmapsto }\,
\begin{Bmatrix}
u^n_N \\[3pt]
\nabla \cdot F^\sigma_N(u^n_N,\nabla u^n_N)
\end{Bmatrix}
&&\,{\longmapsto }\;
u^n_N - \dt \nabla \cdot F^\sigma_N(u^n_N,\nabla u^n_N).
\end{alignat*}
Here,  $\cR$ is the lifting operator and $\sigma$, $\cF$ are the $\sigma$- and $\cF$-layers, respectively, of a $\Psi$-FNO, that are defined in section \ref{sec:fnoprop}. The above representation suggests that the Fourier $\cF$-layers of a $\Psi$-FNO allow us to take \emph{exact} derivatives, and a composition of $\sigma$-layers of a $\Psi$-FNO allows us to approximate continuous functions to any desired accuracy (via the standard universal approximation theorem for finite-dimensional neural networks); in particular, a composition of $\sigma$-layers can provide an approximation 
\[
(u,\nabla u) \mapsto F^\sigma_N(u, \nabla u) \approx \cI_N F(u,\nabla u).
\] 
Thus, by a suitable composition of $\sigma$- and $\cF$-layers, $\Psi$-FNOs can \emph{emulate} pseudo-spectral methods, providing a mechanism by which such neural operators can approximate solution operators for a large class of PDEs efficiently. 

We will make this intuition precise in the following. However, instead of considering a generic abstract form of PDEs, we focus on two PDEs, often encountered in physics, which serve as prototypes for a wide variety of PDEs. We start with an elliptic PDE below.
\subsection{Stationary Darcy Flow}
\label{sec:darcy}
We consider the elliptic PDE of the form,
\begin{align} \label{eq:darcy}
-\nabla \cdot (a\nabla u) = f.
\end{align}
Here, $u \in H^1(\T^d)$, can correspond to the steady-state pressure for a fluid, flowing according to the Darcy's law, in a porous medium with the positive coefficient $a \in L^\infty(\T^d)$, denoting the rock permeability. Another model for \eqref{eq:darcy} is that of a diffusion equation, with $u$ modeling the temperature and $a$ the thermal conductivity of the medium. 

For simplicity,  we assume periodic boundary conditions on $\T^d$, and we impose that $\int_{\T^d} f \, dx = \int_{\T^d} u \, dx = 0$. Employing a suitable rescaling, we will furthermore assume that $a$ can be written in the form:
\[
a = 1 + \tilde{a}, \quad \tilde{a}\in H^s(\T^d), \; s>d/2.
\]
We note that the assumption $s>d/2$ ensures that $\Vert \tilde{a} \Vert_{L^\infty} < \infty$, via the Sobolev embedding $H^s(\T^d) \embeds L^\infty(\T^d)$ (cp. Theorem \ref{thm:sobolev}). To ensure that \eqref{eq:darcy} is well-posed, we assume the following coercivity condition: There exists $\lambda > 0$, such that $1 + \tilde{a} \ge \lambda$. In fact, we shall assume the slightly stronger condition that 
\begin{align} \label{eq:lowerbd}
\Vert \tilde{a} \Vert_{L^\infty}
\le
C\Vert \tilde{a} \Vert_{H^s} 
\le 1-\lambda,
\end{align}
where $C$ is the norm of the embedding $H^s(\T^d)\embeds L^\infty(\T^d)$. The condition \eqref{eq:lowerbd}
clearly implies the $\lambda$-coercivity of \eqref{eq:darcy}.

The underlying operator $\G: L^\infty(\T^d) \to H^1(\T^d)$, maps a coefficient $a \in L^\infty(\T^d)$ to the solution $u \in H^1(\T^d)$ of \eqref{eq:darcy}. Our aim is to \emph{learn} this operator efficiently using a $\Psi$-FNO. To this end, we will follow the program discussed above and first present a pseudo-spectral method that approximates the Darcy flow PDE \eqref{eq:darcy} accurately. Then, this pseudo-spectral method will be emulated by a suitable $\Psi$-FNO. 
\subsubsection{
A Fourier-Galerkin approximation of \eqref{eq:darcy}
}
We fix $N\in \N$ and assume that the coefficient field and right-hand side $a = 1+\tilde{a},f \in H^{s}(\T^d)$, for some $s>d/2$, such that the pseudo-spectral projections $\cI_N \tilde{a}$ and $\cI_Nf$ are well-defined. We can now define a Fourier-Galerkin approximation of \eqref{eq:darcy} as the (unique) solution $u_N \in L^2_N(\T^d)$, with $\int u_N \, dx = 0$, and such that
\begin{align} \label{eq:FG}
- \dot{P}_N \nabla \cdot (a_N \nabla u_N) = f_N,
\end{align}
where we set $a_N := 1+\tilde{a}_N$ and 
\begin{align} \label{eq:defan}
\left\{
\begin{aligned}
\tilde{a}_N &:= \dot{P}_N\cI_{2N}\tilde{a},
\\
f_N &:= \dot{P}_N\cI_{2N}f.
\end{aligned}
\right.
\end{align}
See notation for $P_N, \dot{P}_N$ and $\cI_N$ in appendix \ref{app:notn}. We observe from \eqref{eq:defan} that $\tilde{a}_N$ and $f_N$ are obtained by first carrying out a pseudo-spectral projection of $\tilde{a}$, $f$ on the regular grid $\{x_j\}_{j\in \cJ_{2N}}$ (with $2N$ grid points in each direction), yielding a representation of the form
\[
\cI_{2N}\tilde{a} = \sum_{|k|\le 2N} \hat{a}_k e^{i\langle k, x\rangle},
\quad
\cI_{2N}f = \sum_{|k|\le 2N} \hat{f}_k e^{i\langle k, x\rangle},
\]
for coefficients $\hat{a}_k, \hat{f}_k \in \C$, and then projecting these expressions onto $\dot{L}^2_N(\T^d)$:
\[
\dot{P}_N \cI_{2N} \tilde{a}
=
\sum_{0 < |k|\le N} \hat{a}_k e^{i\langle k, x\rangle},
\quad
\dot{P}_N \cI_{2N} f
=
\sum_{0 < |k|\le N} \hat{f}_k e^{i\langle k, x\rangle}.
\]
The reason for this particular choice of the projection \eqref{eq:defan} lies in the fact that the mapping
\[
L^2_N(\T^d) \times L^2_N(\T^d)
\to 
L^2_N(\T^d), 
\quad 
(a_N, u_N) 
\mapsto
-P_N \nabla \cdot \left( a_N \nabla u_N \right),
\]
can be exactly represented by a pseudo-spectral method on $\{x_j\}_{j\in \cJ_{2N}}$, i.e., on a denser grid of $2N$ grid points in each direction. Indeed, it is easy to see that if $a_N$, $u_N\in L^2_N$, 
then \(\nabla u_N \in L^2_N \), and the values $\nabla u_N(x_j)$ can be computed exactly via the discrete Fourier transform from the values $u_N(x_j)$, $j\in \cJ_{2N}$. Since the product $a_N \nabla u_N\in L^2_{2N}$ is a trigonometric polynomial of degree $\le 2N$, the Fourier coefficients of $a_N\nabla u_N$ can also be recovered from knowledge of the point-values $a_N(x_j)\nabla u_N(x_j)$ at the grid points $x_j\in \cJ_{2N}$. Finally, since the divergence and projection $P_N$ can be evaluated exactly via discrete Fourier transforms, we conclude that the mapping
$(a_N, u_N) 
\mapsto
-P_N \nabla \cdot \left( a_N \nabla u_N \right)
$
can be computed based on knowledge of the grid values $a_N(x_j)$, and $u_N(x_j)$. The above procedure of computing a product of two trigonometric polynomials of degree $\le N$ \emph{exactly}, based on the point-values on a finer grid of size $2N+1$ in each direction is well-known in the context of pseudo-spectral methods, and is usually referred in the numerical analysis literature as \emph{de-aliasing} (cf. section 3.4.2 of the textbook \cite{Canuto2007}).

In order to computationally realize the Fourier-Galerkin method \eqref{eq:FG}, we are going to recast it as a fixed point problem $u_N = F_N(u_N;a_N,f_N)$, where $F_N(\slot;a_N,f_N): L^2_N(\T^d) \to L^2_N(\T^d)$ is defined by
\begin{align} \label{eq:Fcontract}
F_N(u_N;a_N,f_N) 
= 
\dot{P}_N(-\Delta)^{-1} \nabla \cdot \left( \tilde{a}_N \nabla u_N\right) + (-\Delta)^{-1} f_N.
\end{align}
In Lemma \ref{lem:contract} Appendix \ref{app:s3a}, we show that the map \eqref{eq:Fcontract} is a \emph{contraction} and can be efficiently realized by a Picard type fixed-point iteration scheme. This leads to the following algorithm for realizing the Fourier-Galerkin method \eqref{eq:FG} computationally, 
\begin{algorithm}[Approximate solution of Darcy equation] \label{alg:darcy}
\phantom{a} \\
\begin{tabular}{lp{.8\textwidth}}
\textbf{Input:} & $N\in \N$, $a \in H^s(\T^d)$, $f \in \dot{H}^{k-1}(\T^d)$, with $s>d/2+k>d$, $k\in \N$, $k\ge 1$.
\\
\textbf{Output:} & 
$u_N \in H^1_N(\T^d)$, such that $u_N \approx u$, where $u \in H^1(\T^d)$ solves \eqref{eq:darcy} with coefficient field $a$ and right-hand side $f$.
\end{tabular}

\begin{enumerate}
\item Compute pseudo-spectral projections $\tilde{a}_N, f_N\in L^2_N(\T^d)$, defined via the values on ``de-aliased'' grid $\{x_j\}_{j\in \cJ_{2N}}$ (cp. \eqref{eq:defan}):
\begin{align*}
\left\{
\begin{aligned}
\tilde{a}_N &:= \dot{P}_N \cI_{2N} (a-1), \\
f_N &:= \dot{P}_N \cI_{2N} f
\end{aligned} 
\right\}
\in L^2_N(\T^d).
\end{align*}
\item Set 
\[
K := 
\left\lceil
\frac{\log\left(\lambda^{-1} N^{-k}\right)}{\log\left(1-\lambda/2\right)}
\right\rceil.
\]
\item Set $u_N^0 := 0\in L^2_N(\T^d)$.
\item For $k=1,2,\dots, K$: compute
\[
u_N^k  
:= 
\dot{P}_N (-\Delta)^{-1} \nabla \cdot \left( \tilde{a}_N \nabla u_N^{k-1} \right)
+
(-\Delta)^{-1} f_N.
\]
\item Set $u_N := u^K_N$.
\end{enumerate}
\end{algorithm}
We have the following theorem on the convergence of the algorithm \ref{alg:darcy},
\begin{theorem} \label{thm:FG}
Let $u$ be the unique solution of \eqref{eq:darcy}. Let $k\in \N$ be given, with $k>d/2+1$, and assume that the coefficient $a\in H^s(\T^d)$ for $s> d/2 + k$ satisfies the coercivity condition \eqref{eq:lowerbd}. Assume furthermore that $f \in H^{k-1}(\T^d)$. Then there exists $N_0 = N_0(s,d,\Vert a \Vert_{H^s},\lambda)\in \N$, such that for any integer $N \ge N_0$, there exists a unique solution of the discretized elliptic equation \eqref{eq:FG}, and there exists a constant $C = C(s,d,\lambda, \Vert a \Vert_{H^s}, \Vert f \Vert_{H^{k-1}})>0$, such that 
\[
\Vert u - u_N \Vert_{H^1} \le C N^{-k},
\]
where $u_N$ is the output of Algorithm \ref{alg:darcy}.
\end{theorem}

The proof of Theorem \ref{thm:FG} relies only on standard techniques of numerical analysis and is presented in appendix \ref{app:FG}. 

\subsubsection{$\Psi$-FNO approximation of the Darcy equations }
We will emulate the pseudo-spectral fixed point algorithm \ref{alg:darcy} by a $\Psi$-FNO, allowing us to derive approximation bounds. We consider the following setting,
\begin{setting} \label{set:darcy}
Let $s\ge d/2+k+\delta$ for some $k\in \N$, $\delta\in (0,1)$ be a given Sobolev regularity parameter, and let $\lambda \in (0,1)$ be a (coercivity) constant. Denote by $\cA_\lambda^s(\T^d) \subset H^s(\T^d)$ the set of $\lambda$-coercive coefficients $a\in H^s(\T^d)$ of the form $a = 1 + \tilde{a}$, and such that
\[
\Vert a \Vert_{H^s} \le \lambda^{-1},
\quad
\Vert \tilde{a}\Vert_{L^\infty} \le C \Vert \tilde{a} \Vert_{H^{d/2+\delta}} \le 1-\lambda.
\]
Here $C>0$ denotes the norm of the embedding $H^{d/2+\delta}(\T^d) \embeds L^\infty(\T^d)$. We consider the operator $\G: \cA^s_\lambda(\T^d) \to \dot{H}^1(\T^d)$, $a \mapsto u$, where $u$ solves the Darcy equation 
\[
-\nabla \cdot \left(a\nabla u\right) = f,
\quad
\fint_{\T^d} u(x) \, dx = 0,
\]
on the periodic torus $\T^d$, with right-hand side $f \in \dot{H}^{k-1}$.
\end{setting}

With this setting in place, we can now state our main FNO approximation theorem for the solution operator $\G: \cA^s_\lambda(\T^d) \to H^1(\T^d)$ of the Darcy problem:

\begin{theorem} \label{thm:fno-darcy}
Assume Setting \ref{set:darcy}, and assume that the activation function $\sigma$ is non-linear and $\sigma \in C^3(\R)$. Then there exists a constant $C = C(s,\lambda,d)>0$, such that for any $N\in \N$, there exists a $\Psi$-FNO $\cN: \cA_\lambda^s(\T^d) \to H^1(\T^d)$, such that 
\[
\sup_{a\in \cA_{\lambda}^s} 
\Vert \cG(a) - \cN(a) \Vert_{H^1(\T^d)}
\le C N^{-k}
\]
and
\[
\width(\cN) \le C N^d
, 
\quad 
\depth(\cN) \le C\log(N)
,
\quad 
\lift(\cN) \le C.
\]
In particular, we have
\[
\size(\cN) \lesssim N^d \log(N).
\]
\end{theorem}
The proof of this theorem, presented in Appendix \ref{app:fno-darcy}, relies crucially on the efficient approximation of quadratic nonlinearities by shallow neural networks with smooth activation functions, see Lemma \ref{lem:darcy-quadratic} in Appendix \ref{app:s3a} for the precise statement and proof. 
\begin{remark} \label{rmk:darcy-orderest}
To achieve a FNO approximation error of order $\epsilon$ for the Darcy problem, Theorem \ref{thm:fno-darcy} shows that a $\Psi$-FNO $\cN$ with
\begin{equation}
    \label{eq:darcycomp}
\size(\cN) \lesssim \left( \frac{1}{\epsilon} \right)^{d/k} \log(\epsilon^{-1}),
\quad
\depth(\cN) \lesssim \log(\epsilon^{-1}),
\end{equation}
is sufficient. Furthermore, the lifting dimension $d_v$ can be kept uniformly bounded, independently of $\epsilon$. In particular, for $k>d$, the required total size of the $\Psi$-FNO scales sub log-linearly in the approximation accuracy $\epsilon$, indicating that $\Psi$-FNOs may provide a very efficient approximation to the solution operator of the Darcy problem, in this case.
\end{remark}
As a concrete example for illustrating the approximation of Darcy equations by FNOs, we consider the following example.
\begin{example}
A possible model for coefficients $a$ with a typical length scale $\ell > 0$, is to assume an expansion of $a = a(x;Y)$ in terms of random variables $Y = (Y_1,Y_2,\dots) \in [-1,1]^\N$ (not necessarily i.i.d.), of the following form, similar to the ansatz in \cite{SchwabZech2019}:
\[
a(x;Y) = 1 + \sum_{k\in \Z^d\setminus \{0\}} b_k Y_k e^{i\langle k, x\rangle},
\]
where we assume that the coefficients $b_k$ satisfy a decay condition of the form $|b_k| \le C_b\exp(-\ell|k|)$ for a constant $C_b > 0$. We also assume that $a(x;Y)$ satisfy the coercivity condition 
\[
\Vert a(\slot; Y) - 1\Vert_{L^\infty} 
\le C \Vert a(\slot; Y) - 1 \Vert_{H^{d/2+\delta}(\T^d)}
\le 1-\lambda,
\]
for some $\delta, \lambda \in (0,1)$, and that the source term $f\in C^\infty(\T^d)$. Under these assumptions, we have $a\in H^s$ for any $s > 0$, and hence it follows from Theorem \ref{thm:fno-darcy} that \emph{for any $r \in \N$}, there exists a constant $C = C(r,\ell, \Vert f \Vert_{C^k}, \delta)$, with the following property (cp. also Remark \ref{rmk:darcy-orderest}): For any $\epsilon > 0$, there exists a $\Psi$-FNO $\cN$, such that 
\[
\sup_{Y\in [-1,1]^\N} \Vert \G(a(\slot, Y)) - \cN(a(\slot,Y)) \Vert_{H^1} < \epsilon,
\]
and 
\begin{equation}
    \label{eq:darcycomp1}
\size(\cN) \le C\epsilon^{-1/r}, 
\quad
\depth(\cN) \le C\log(\epsilon^{-1}).
\end{equation}
\end{example}
The complexity bounds \eqref{eq:darcycomp} and \eqref{eq:darcycomp1} suggest that the size of a $\Psi$-FNO approximating the operator $\G$ stemming from the Darcy equations, scales only sub-linearly (or even better) in the desired accuracy $\epsilon$. This should be contrasted with the fact that the size of a $\Psi$-FNO for a generic Lipschitz operator grows super-exponentially in the desired error \eqref{eq:comp}. Thus, we are able to show that a $\Psi$-FNO can approximate this PDE solution operator much more efficiently than it might a generic infinite-dimensional operator. 
\subsection{Incompressible Navier-Stokes equations.}
\label{sec:NS}
The motion of a viscous, incompressible Newtonian fluid is modeled by the incompressible Navier-Stokes equations,
\begin{align} \label{eq:NS}
\left\{
\begin{aligned}
\partial_t u + u\cdot \nabla u + \nabla p = \nu \Delta u, \\
\div(u) = 0, \; u(t=0) = u_0,
\end{aligned}
\right.
\end{align}
Here, $u \in \R^d$ is the fluid velocity and $p\in \R$ is the pressure of the fluid, acting as a Lagrange multiplier to enforce the divergence-free constraint $\div(u) = 0$. The initial fluid velocity is denoted by $u_0$. For simplicity, we assume periodic boundary conditions in the domain $\T^d$. The viscosity is denoted by $\nu\geq 0$ and we would like to state that the subsequent analysis also applies for $\nu=0$, where \eqref{eq:NS} reduces to the incompressible Euler equations modeling an ideal fluid. For definiteness, we recall the following well-known theorem for the well-posedness of the Navier-Stokes equations \eqref{eq:NS},
\begin{theorem}[{see e.g. \cite[Thm 3.4]{MB2001}}] \label{thm:classicalsol}
Let $r>d/2+2$. For any $u_0 \in H^r$, there exists $T > 0$ and a unique classical solution of the Navier-Stokes equations \eqref{eq:NS}, such that $u \in C([0,T];H^r) \cap C^1([0,T];H^{r-2})$ with $u(t=0) = u_0$.
\end{theorem}
It is well known that in two space dimensions $d=2$, the time interval for existence of solutions $[0,T]$ can be extended to any finite $T$ as long as $u_0 \in H^r$, whereas the corresponding finite-time well-posedness result for three space-dimensions is an outstanding open problem \cite{MB2001}. 

We recall that if the initial data $u_0$ of \eqref{eq:NS} belongs to $\dot{L}^2(\T^d;\R^d)$, i.e. if 
\[
\int_{\T^d} u_0(x) \, dx = 0,
\]
then we also have that the corresponding solution $u(x,t)\in \dot{L}^2(\T^d;\R^d)$ (reflecting momentum conservation). Next, we introduce the \define{Leray-projection operator} $\P: {L}^2(\T^d;\R^d) \to \dot{L}^2(\T^d;\div)$, as the $L^2$-orthogonal projection onto the subspace $\dot{L}^2(\T^d;\div)\subset \dot{L}^2(\T^d;\R^d)$, consisting of divergence-free vector fields; i.e. we have $u\in \dot{L}^2(\T^d;\div)$ if, and only if, $u\in \dot{L}^2(\T^d;\R^d)$ and 
\[
\int_{\T^d} u(x) \cdot \nabla \phi(x) \, dx = 0, \quad \forall \, \phi \in C^\infty(\T^d). 
\]
In terms of Fourier series, the Leray projection $\P: \dot{L}^2(\T^d;\R^d) \to \dot{L}^2(\T^d;\div)$ is explicitly given by
\begin{align} \label{eq:leray}
\P\left(
\sum_{k\in \Z^d} \hat{u}_k e^{i\langle k, x\rangle}
\right)
=
\sum_{k\in \Z^d\setminus \{0\}} \left(
{1} - \frac{k\otimes k}{|k|^2}
\right) \hat{u}_k e^{i\langle k, x\rangle}.
\end{align}

In terms of the Leray projection $\P$, we can now equivalently write the incompressible Navier-Stokes equations \eqref{eq:NS} as the following equation on the Hilbert space  $\dot{L}^2(\T^d;\div)$ as, 
\begin{align} \label{eq:NSp}
\left\{
\begin{aligned}
\partial_t u = -\P\left( u\cdot \nabla u \right) + \nu \Delta u, \\
u(t=0) = u_0.
\end{aligned}
\right.
\end{align}

Given this background, our main objective in this section is to construct a $\Psi$-FNO that will approximate the operator $\G$ which maps the initial data $u_0$ to the solution $u(\slot,T)$ (at the final time $T$) of the incompressible Navier-Stokes equations \eqref{eq:NS}, \eqref{eq:NSp}. To this end, we will follow the general program outlined at the beginning of this section and introduce a suitable pseudo-spectral method for approximating the Navier-Stokes equations. Then, we construct a $\Psi$-FNO that can efficiently emulate this pseudo-spectral method. 
\subsubsection{A fully-discrete $\Psi$-spectral approximation of the Navier-Stokes equations \eqref{eq:NS}}
The form of the Leray-projected Navier-Stokes equations \eqref{eq:NSp} naturally suggests the following fully-discrete approximation of \eqref{eq:NS}:
\begin{align} \label{eq:scheme1}
\left\{
\begin{gathered}
\frac{u_N^{n+1} - u_N^n}{\dt} 
+ \P_N \left(
u^n_N \cdot \nabla u^{n+1}_N
\right)
=
\nu \Delta u_N^{n+1}, 
\\
u^0_N = \cI_Nu(t=0).
\end{gathered}
\right.
\end{align}
Here, we fix $N\in \N$ and introduce the space, $\dot{L}^2_N(\T^d;\div) := \dot{L}^2(\T^d;\div)\cap \dot{L}^2_N(\T^d;\R^d)$. We fix a time-step $\tau > 0$ and let $u^n_N \in \dot{L}^2_N(\T^d;\div)$, for all $n=0,\ldots,n_T$, with $n_T$ such that $\tau n_T = T$. Moreover, we use the following \emph{finite-dimensional} Leray-Fourier projection operator $\P_N: \dot{L}^2(\T^d;\R^d) \to \dot{L}^2_N(\T^d;\div)$ in analogy with \eqref{eq:leray}:
\begin{align} \label{eq:lerayn}
\P_N\left(
\sum_{k\in \Z^d} \hat{u}_k e^{i\langle k, x\rangle}
\right)
:=
\sum_{0<|k|_\infty\le N}
\left(
1 - \frac{k\otimes k}{|k|^2}
\right) \hat{u}_k e^{i\langle k, x\rangle},
\end{align}
to complete the description of the scheme \eqref{eq:scheme1}. 

We observe that the scheme \eqref{eq:scheme1} is \emph{implicit} i.e., at each time step $n$, one has to solve an operator equation to compute the velocity field $u^{n+1}_N$ at the next time step. Thus, one needs to show the solvability of this operator equation in order to ensure that the scheme \eqref{eq:scheme1} is \emph{well-defined}. Under the following CFL condition for choosing a small enough time step $\tau$, 
\begin{equation}
    \label{eq:cfl}
    \dt \Vert u^n_N \Vert_{L^\infty} N \le \frac12,
\end{equation}
we prove in Appendix \ref{app:s1} that the scheme \eqref{eq:scheme1} is well-defined. 

Next, in practice, one has to numerically approximate the solutions of the implicit equation \eqref{eq:scheme1} for evaluating the velocity field $u^{n+1}_N$, at the next time-step. We choose to do so by recasting the solution of the implicit equation \eqref{eq:scheme1} to finding a \emph{fixed point} for the mapping,
\begin{align} \label{eq:Frec}
w_N \mapsto F(w_N) := (1-\nu\dt\Delta)^{-1} u^n_N - \dt (1-\nu\dt\Delta)^{-1} \P_N(u^n_N\cdot \nabla w_N).
\end{align}
In Appendix \ref{app:s1} Lemma \ref{lem:iter1}, we show that a standard Picard-type iteration converges to a fixed point for the map \eqref{eq:Frec}. This suggests the following numerical algorithm for approximating strong solutions of the incompressible Navier-Stokes equations \eqref{eq:NS},

\begin{algorithm}[Pseudo-spectral approximation of the Navier-Stokes equations \eqref{eq:NS}] 
\label{alg:NS1}
\phantom{a} \\
\noindent
\begin{tabular}{lp{.8\textwidth}}
\textbf{Input:} & 
$U>0$, $N\in \N$, $T>0$, a time-step $\dt>0$, such that $n_T = T/\dt \in \N$, and $\dt U N^{d/2+1} \le \frac1{2e}$, initial data $u_N^0 \in L^2_N(\T^d;\div)$, such that $\Vert u_N^0 \Vert_{L^2} \le U$.
\\
\textbf{Output:} & 
$u_N^{n_T} \in L^2_N(\T^d;\div)$ an approximation of the solution $u_N^{n_T} \approx u(t=T)$ of \eqref{eq:NS} at time $t=T$.
\end{tabular}
\begin{enumerate}
\item Set 
\[
\kappa_0 := 
\left\lceil
\frac{\log\left(T^2 /\dt^2\right)}{\log(2)}
\right\rceil
\in \N.
\]
\item For $n=0,\dots, n_T-1$:
\begin{enumerate}
\item Set $w^{n,0}_N := 0$,
\item For $k=1,\dots, \kappa_0$: Compute
\[
\hspace{40pt}
w^{n,k}_N := (1-\nu \dt \Delta)^{-1} u^n_N - \dt (1-\nu\dt \Delta)^{-1} \P_N \left( u^n_N \cdot \nabla w^{n,k-1}_N \right),
\]
\item Set $u^{n+1}_N := w^{n,\kappa_0}_N$,
\end{enumerate}
\end{enumerate}
\end{algorithm}
The convergence of the algorithm \ref{alg:NS1}, together with a convergence rate, to the strong solution of the Navier-Stokes equations is summarized in the following theorem,
\begin{theorem} \label{thm:scheme1}
Let $U,T>0$. Consider the Navier-Stokes equations on $\T^d$, for $d\ge 2$. Assume that $r\ge d/2 +2$, and let $u \in C([0,T]; H^r) \cap C^1([0,T];H^{r-2})$ be a solution of the Navier-Stokes equations \eqref{eq:NS}, such that $\Vert u \Vert_{L^2}\le U$. Choose a time-step $\dt$, such that $\dt U N^{d/2+1}\le (2e)^{-1}$. There exists a constant 
\[
C = C(T,d,r,\Vert u \Vert_{C_t(H^r_x)}, \Vert u \Vert_{C^1_t(H^{r-2}_x)}) > 0,
\]
such that with $u^0_N := \cI_N u(0)$, and for the sequence $u^1_N, \dots, u^{n_T}_N \in L^2_N(\T^d;\div)$ generated by Algorithm \ref{alg:NS1}, we have
\[
\max_{n=0,\dots, n_T}
\Vert u^n_N - u(t^n) \Vert_{L^2}
\le
C \left(\dt + N^{-r}\right),
\]
where $n_T\dt = T$. In particular, choosing $\dt \sim N^{-r}$, we have
\[
\max_{n=0,\dots, n_T}
\Vert u^n_N - u(t^n) \Vert_{L^2}
\le
C N^{-r},
\]
with $n_T \sim N^r$ (and enlarging the constant $C>0$ by a constant factor).
\end{theorem}
The proof of this theorem relies on several techniques from numerical analysis and is presented in detail in Appendix \ref{app:s1pf}. 
\subsubsection{Approximation of algorithm \ref{alg:NS1} by $\Psi$-FNOs.}
Next, we are going to construct a $\Psi$-FNO of the form \eqref{eq:form-pfno}, which can efficiently emulate the pseudo-spectral algorithm \ref{alg:NS1}. To this end, we have the following result (proved in Appendix \ref{app:NS1}) on the efficient approximation of the non-linear term in the Navier-Stokes equations by FNOs,
\begin{lemma} \label{lem:fno-NSnonlin}
Assume that the activation function $\sigma\in C^3$ is three times continuously differentiable and non-linear. There exists a constant $C>0$, such that for any $N\in \N$, and for any $\epsilon, B > 0$, there exists a $\Psi$-FNO $\cN: L^2_{2N}(\T^d;\R^d)\times L^2_{2N}(\T^d;\R^d) \to L^2_{2N}(\T^d;\R^d)$, with 
\[
\depth(\cN), \; \lift(\cN) \le C, 
\quad
\width(\cN) \le C N^d,
\]
such that we have
\[
\left\Vert
\P_N\left(
u_N \cdot\nabla w_N
\right)
- 
\cN(u_N, w_N)
\right\Vert_{L^2_N}
\le 
\epsilon,
\]
for all trigonometric polynomials $u_N, w_N\in L^2_N(\T^d; \R^d)\subset L^2_{2N}(\T^d; \R^d)$ of degree $|k|_\infty \le N$, satisfying the bound $\Vert u_N \Vert_{L^2}, \Vert w_N \Vert_{L^2} \le B$.
\end{lemma}
Thus, from the preceding Lemma, we have that the nonlinearities in algorithm \ref{alg:NS1} can be efficiently approximated by $\Psi$-FNOs. This paves the way for the following theorem on the emulation of the pseudo-spectral algorithm \ref{alg:NS1} by $\Psi$-FNOs,
\begin{theorem} \label{thm:fno-NS1}
Let $U,T>0$ and viscosity $\nu \ge 0$. Consider the Navier-Stokes equations on $\T^d$, for $d\ge 2$. Assume that $r\ge d/2 +2$, and let $\cV \subset C([0,T]; H^r) \cap C^1([0,T];H^{r-2})$ be a set of solutions of the Navier-Stokes equations \eqref{eq:NS}, such that $\sup_{u\in \cV} \Vert u \Vert_{L^2}\le U$, and 
\[
\bar{U}:= \sup_{u\in \cV} \left\{
\Vert u \Vert_{C_t(H^r_x)} + \Vert u \Vert_{C^1_t(H^{r-2}_x)}
\right\}
< \infty.
\]
For $t\in [0,T]$, denote $\cV_t := \set{u(t)}{u\in \cV}$. Let $\G: \cV_0 \to \cV_T$ denote the solution operator of \eqref{eq:NS}, mapping initial data $u_0 = u(t=0)$, to the solution $u(T)$ at $t=T$ of the incompressible Navier-Stokes equations. There exists a constant 
\[
C = C(d,r,U,\bar{U},T) > 0,
\]
such that for $N\in \N$ there exists a $\Psi$-FNO $\cN: L^2_N(\T^d;\R^d) \to L^2_N(\T^d;\R^d)$, such that 
\[
\sup_{u\in \cV_0}
\Vert \G(u) - \cN(u) \Vert_{L^2}
\le
CN^{-r},
\]
and such that 
\[
\width(\cN) \le C N^d, 
\quad
\depth(\cN) \le C N^r \log(N),
\quad
\lift(\cN) \le C.
\]
\end{theorem}
The proof of this theorem is provided in Appendix \ref{app:NS1}. 
\begin{remark}
\label{rem:NScomp}
It is straightforward to observe from Theorem \ref{thm:fno-NS1} that the size of a $\Psi$-FNO to achieve a desired error tolerance of $\epsilon > 0$, scales (neglecting $\log$-terms) as
\begin{equation}
\label{eq:NScomp}
\size(\cN) \le C\epsilon^{-\left(1+\frac{d}{r}\right)},
\end{equation}
Given that we need $r \geq d/2+2$, we observe from \eqref{eq:NScomp} that the size of the $\Psi$-FNO, approximating the initial data to solution operator $\G$, for the Navier-Stokes equations \eqref{eq:NS}, scales at most \emph{sub-quadratically} with respect to the error tolerance $\epsilon$ for the physically relevant values $d=2,3$. This polynomial scaling should be compared with the super-exponential growth (see Remark \ref{rem:comp}) of the size of FNOs in approximating a generic Lipschitz-continuous operator. Thus, we are able to demonstrate that $\Psi$-FNOs can approximate the solutions of Navier-Stokes equations far more efficiently than what the universal approximation theorem \ref{thm:pfno-universal} suggests. 
\end{remark}
\begin{remark}
\label{rem:scheme2}
From the convergence theorem \ref{thm:scheme1}, we observe that the underlying scheme \eqref{eq:scheme1} is \emph{first-order} in time. This low accuracy of the scheme necessitates a large number of time steps and affects the overall complexity. We describe a second-order accurate time discretized version of the pseudo-spectral method for approximating the Navier-Stokes equations \eqref{eq:NS} in Appendix \ref{app:scheme2} and in complete analogy with Theorem \ref{thm:fno-NS1} (see Theorem \ref{thm:fno-NS2} in Appendix \ref{app:scheme2}), we can construct a $\Psi$-FNO to emulate this second-order in time pseudo-spectral scheme, resulting in a $\Psi$-FNO of
\begin{equation}
\label{eq:NScomp2}
\size(\cN) \le C\epsilon^{-\left(\frac{1}{2}+\frac{d}{r}\right)},
\end{equation}
to obtain a desired accuracy of $\epsilon$. Thus, we can obtain a more efficient approximation of the underlying operator than $\Psi$-FNO emulating the first-order time scheme \eqref{eq:scheme1}. In particular for $r \geq 2d$, we obtain that the size of a $\Psi$-FNO only grows sub-linearly in terms of the desired accuracy, making this FNO approximation comparable in complexity to the FNO approximation of the Darcy equation (see \eqref{eq:darcycomp}). 
\end{remark}
\section{Comparison of FNOs with DeepOnets}
\label{sec:4}
In this section, we will compare $\Psi$-FNOs with another operator learning framework, namely DeepOnets of \cite{ChenChen,deeponets}, defined as,
\begin{definition} \label{def:deeponet}
Fix $m,p\in \N$. A \define{DeepOnet} $\cN$ with output dimension $p$ and \define{sensor points} $x_1,\dots, x_m\in \T^d$ is a mapping $\cN: C(\T^d;\R^{d_a}) \to C(\T^d;\R^{d_v})$ of the form 
\[
\cN(a)(x) 
=
\sum_{k=1}^p \beta_k(a(x_1),\dots, a(x_p)) \tau_k(x),
\]
where $\beta: \R^{m\times d_a} \to \R^{p\times d_u}$, $\alpha \mapsto \beta(\alpha) = (\beta_1(\alpha),\dots, \beta_p(\alpha))$, and $\tau: \R^d \mapsto \R^{p}$, $x\mapsto \tau(x) = (\tau_1(x),\dots,\tau_p(x))$ are (ordinary) neural networks. We refer to $\beta$ and $\tau$ as the \define{branch} and \define{trunk} nets, respectively.
\end{definition}
In the following theorem (proved in Appendix \ref{app:E}), we show that a $\Psi$-FNO can naturally be viewed as a DeepOnet with a specific choice of the branch and trunk-nets; where the trunk net is fixed to represent a trigonometric basis, and the branch net is constrained by a specific choice of hidden layer architecture, which provides a more parsimonious representation compared to a canonical DeepOnet implementation based on dense layers.
\begin{theorem}[DeepOnet approximation of $\Psi$-FNOs]
\label{thm:deeponet}
Let $\cN: L^2_N(\T^d;\R^{d_a}) \to L^2_N(\T^d;\R^{d_u})$ be a $\Psi$-FNO, and fix $B>0$. For any $\epsilon > 0$, there exists $p\in \N$ and a DeepOnet $(\beta,\tau)$, with branch net $\beta$, trunk net $\tau$, and sensor points $\{x_j\}_{j\in \cJ_N}$, such that 
\[
\sup_{\Vert a\Vert_{L^\infty}\le B} 
\sup_{y \in \T^d}
\left|
\cN(a)(y) - \sum_{k=1}^p \beta_k(a) \tau_k(y)
\right|
\le \epsilon,
\]
where the first supremum is taken over all $a\in C(\T^d)$, such that $\Vert a \Vert_{L^\infty} \le B$.
Furthermore, we have 
\[
\width(\beta) = \width(\cN), 
\quad
\depth(\beta) = \depth(\cN),
\]
and the trunk net $\tau$ defines a mapping $\tau: \R^{d} \to \R^{\cK_N}$, which approximates an (arbitrary) orthonormal trigonometric basis $\{\fb_k\}_{k\in \cK_N}$ with $\Span\{\fb_k\}_{k\in \cK_N} = \Span\{e^{i\langle k, x\rangle}\}_{k\in \cK_N}$, such that 
\[
\max_{k\in \cK_N} \Vert \fb_k - \tau_k \Vert_{L^\infty} \le \epsilon / \bar{B},
\]
where
\[
\bar{B} := 
(2N+1)^d
\left(
\sup_{\Vert a \Vert_{L^\infty}\le B} \Vert \cN(a) \Vert_{L^2}
\right).
\]
\end{theorem}
\begin{remark}
We note that due to the particular architecture of $\Psi-$Fourier neural operators, the total size of a $\Psi$-FNO $\cN$ is upper bounded by 
\[
\size(\cN) \lesssim \depth(\cN) \width(\cN) \lift(\cN).
\]
 In contrast, since $\beta$ and $\tau$ in the DeepOnet approximation of Theorem \ref{thm:deeponet} are conventional deep neural networks, the total number of degrees of freedom (i.e. the total number of weights and biases) is given, in general, by 
\[
\size(\beta) 
\sim \depth(\beta) \width(\beta)^2
=
\depth(\cN) \width(\cN)^2,
\]
and $\size(\tau) \sim \depth(\tau) \width(\tau)^2$. So that in particular, comparing the sizes of the $\Psi$-FNO and the corresponding (fully connected) DeepOnet approximation, we have
\[
\frac{\size(\cN)}{\size(\beta,\tau)}
\lesssim
\frac{\lift(\cN)}{\width(\cN)}.
\]
Based on our explicit complexity estimates \eqref{eq:darcycomp} and \eqref{eq:NScomp}, we can expect that $\lift(\cN) \ll \width(\cN)$ for a wide range of problems which can be efficiently approximated by FNOs. In particular, this indicates that the $\Psi$-FNO provides a more parsimonious approximation of the underlying operator than the DeepOnet ``emulator'' constructed in Theorem \ref{thm:deeponet}, with $\size(\cN) \ll \size(\beta,\tau)$.
\end{remark}
\section{Summary and Discussion.}
\label{sec:5}
Many learning tasks, particularly, but not exclusively, in scientific computing, are naturally formulated as learning operators mapping one infinite-dimensional space to another. Neural operators have recently been proposed as a framework for operator learning. A particular form, the so-called Fourier Neural Operators (FNOs) \eqref{eq:fno}, have been shown to be efficient in approximating a wide variety of operators that arise in PDEs \cite{fourierop2020}. Our main aim in this paper was to analyze FNOs and $\Psi$-FNOs \eqref{eq:form-pfno}, which is a concrete computational realization of FNOs. To this end, we have presented the following results,
\begin{itemize}
    \item We showed in Theorem \ref{thm:universal} and Theorem \ref{thm:pfno-universal} that FNOs (resp. $\Psi$-FNOs) are \emph{universal} i.e., they can approximate any continuous operator to desired accuracy. Our proof relies heavily on the ability of FNOs to approximate the Fourier transform and its inverse, together with the neural network approximation of the finite-dimensional Fourier conjugate operator \eqref{eq:Fconj}. Thus, FNOs have the same universal approximation property as canonical neural networks for finite-dimensional functions and DeepOnets for operators \cite{LMK2021}. This universality result paves the way for the widespread use of FNOs in the context of operator learning.
    \item However as stated in remark \ref{rem:comp}, in the worst case, the size of a FNO can grow \emph{super-exponentially} in terms of the desired error for approximating a general Lipschitz continuous operator. This might inhibit the use of FNOs. On the other hand, we argue in the beginning of section \ref{sec:3} that $\Psi$-FNOs, which are a concrete computational realization of FNOs, can approximate the nonlinearities and differential operators that define PDEs, very efficiently. Hence, one can think of $\Psi$-FNOs as a new form of pseudo-spectral methods for PDEs, which in practice are adapted to, and optimized based on the given training data. Thus, one can expect that $\Psi$-FNOs can approximate PDEs efficiently. 
    \item We consider two widely used prototypical PDEs, namely the elliptic PDE \eqref{eq:darcy} that arises in a stationary Darcy flow and the well-known incompressible Navier-Stokes equations for fluid dynamics. For both these PDEs, we prove rigorously that there exists a $\Psi$-FNO which can approximate the underlying nonlinear operators efficiently, as we can show that the size of the $\Psi$-FNO only needs to grow polynomially in terms of the error. In fact, we show that the size grows sub-linearly in terms of the error. Thus, FNOs can approximate these widely used PDEs efficiently, corroborating the empirical results presented in \cite{fourierop2020}. 
\end{itemize}
Hence, our analysis provides very strong theoretical evidence that FNOs are an effective framework for operator learning. Moreover, we also compare FNOs to an alternative operator learning framework, that of DeepOnets \cite{deeponets} and show that $\Psi$-FNOs can be thought of a special case of DeepOnets with a trunk net approximating trigonometric functions and sensor points being equi-distributed Cartesian grid points. Given its special architecture, we argue that a $\Psi$-FNO can allow for a more parsimonious representation of operators than a DeepOnet, enabling a cheaper approximation of certain operators. 

The comparison with DeepOnets also brings out some obvious limitations of FNOs. In particular, FNOs are efficient on rectangular domains as the $\Psi$-FNO can be evaluated efficiently with FFT. Although one can use FNOs for operators defined on arbitrary domains using suitable extension and restriction operators (see Theorem \ref{thm:universal_extend}), it is unclear if these operators can be realized computationally in an efficient manner. Moreover, FNOs fix the trigonometric basis as the trunk net in the underlying DeepOnet (see Theorem \ref{thm:deeponet}). On the other hand, a general DeepOnet can learn trunk nets from the data during training, allowing the possibility of learning a more suitable representation from data. These considerations call for a more thorough computational comparison between DeepOnets and FNOs, and possibly other operator learning frameworks such as the one from \cite{bhattacharya2020model}.

Similarly, extending the analysis of this article to other neural operators can be readily envisaged. The use of FNOs for more general operators, particularly those arising in non-scientific computing settings, such as images, text and speech, also needs to be investigated.

\bibliographystyle{siam}
\bibliography{bibliography}

\appendix
\section{Glossary of mathematical notation}
\label{sec:glos}

\begin{center}

\begin{tabular}{l p{0.7\textwidth}  r}
\textbf{Symbol} & \textbf{Description} & \textbf{Page}
\\
\hline
$\sigma$ & activation function & \\
$\T^d$ & periodic torus, identified with $[0,2\pi]^d$ & \\
$d$ & spatial dimension of domain & p.~\pageref{setting} \\
$d_a$, $d_u$, $d_v$ & number of components of input, output and lifting & p.~\pageref{setting} \\
$\cA(D;\R^{d_a})$ & input function space & p.~\pageref{setting} \\
$\cU(D;\R^{d_u})$ & output function space & p.~\pageref{setting} \\
$\cF$, $\cF^{-1}$ & Fourier transform and inverse Fourier transform & p.~\pageref{eq:ft} \\
$\cF_N$, $\cF_N^{-1}$ & discrete Fourier transform and inverse & p.~\pageref{eq:dft} \\
$\{x_j\}_{j\in \cJ_N}$ & regular periodic grid, $x_j = 2\pi j/(2N+1)$ & p.~\pageref{eq:grid} \\
$\cJ_N$ & grid point indices, $\cJ_N = \{0,\dots, 2N\}^d$ & p.~\pageref{eq:index} \\
$\cK_N$ & Fourier wavenumbers $\cK_N = \set{k\in \Z^d}{|k|_\infty \le N}$ & \\
$\cR$ & lifting operator & p.~\pageref{eq:no-r} \\
$\cL_\ell$ & neural operator layer & p.~\pageref{eq:fno-layer} \\
$\cQ$ & projection operator & p.~\pageref{eq:no-q} \\
$\cF$-layer & linear, non-local layer; $v(x)\mapsto \cF^{-1}(P\cF v)(x)$ & p.~\pageref{sec:fnoprop} \\
$\sigma$-layer & non-linear, local layer; $v(x) \mapsto \sigma(Wv(x) + b(x))$ & p.~\pageref{sec:fnoprop} \\[5pt]
\multicolumn{3}{c}{\textbf{Banach spaces}} \\
\hline && \\
$L^2$ & Space of square-integrable functions &  \\
$\dot{L}^2$ & $\dot{L}^2 \subset L^2$ square-integrable functions with zero mean & p.~\pageref{eq:proj} \\
$L^2_N$ & $L^2_N\subset L^2$ trigonometric polynomials of degree $\le N$ & p.~\pageref{eq:pN} \\
$\dot{L}^2_N$ & $\dot{L}^2_N = \dot{L}^2\cap L^2_N$ trigonometric polynomials  with zero mean & p.~\pageref{eq:dotHs} \\
$H^s$ & Sobolev space of smoothness $s$, with norm $\Vert \slot \Vert_{H^s}$ & p.~\pageref{eq:Hs} \\
$\dot{H}^s$ & Sobolev space with zero mean, with norm $\Vert \slot \Vert_{\dot{H}^s}$ & p.~\pageref{eq:dotHs} \\[5pt]
\multicolumn{3}{c}{\textbf{Projection operators}} \\
\hline && \\
$P_N$ & $L^2$-orthogonal Fourier projection $P_N: L^2 \to L^2_N$ & p.~\pageref{eq:proj} \\
$\dot{P}_N$ & Fourier projection $\dot{P}_N: {L}^2 \to \dot{L}^2_N$ with zero mean & p.~\pageref{eq:dotproj} \\
$\cI_N$ & Pseudo-spectral Fourier projection, i.e. trigonometric interpolation on regular grid $\{x_j\}_{j\in\cJ_N}$  & p.~\pageref{eq:pproj} \\
$\mathbb{P}$ & Leray projection onto divergence-free vector fields & p.~\pageref{eq:leray} \\
$\mathbb{P}_N$ & Leray projection followed by projection $P_N$; $\mathbb{P}_N = P_N \circ \mathbb{P}$ & p.~\pageref{eq:lerayn}
\end{tabular}

\end{center}

\section{Notation and Technical Preliminaries.}
\label{app:notn}
In this section, we introduce frequently used notation in the article main text and recall some essential facts about Fourier analysis. 

In the main text, our focus is functions defined on the periodic torus $\T^d$, identified as $\T^d = [0,2\pi]^d$. Following standard practice, we  denote by $L^2(\T^d)$ the space of square-integrable functions. For any such function $v \in L^2(\T^d)$, we can define the \emph{Fourier Transform} as, 
\begin{equation}
    \label{eq:ft}
\cF(v)(k)
:=
\frac1{(2\pi)^d}
\int_{\T^d}
v(x) e^{-i\langle k, x\rangle} \, dx,
\quad
\forall \, k\in \Z^d.
\end{equation}
For any $k\in \Z^d$, the $k$-th Fourier coefficient of $v$ is denoted by $\hat{v}_k =\cF(v)(k)$. 

Given a set of Fourier coefficients $\{\hat{v}_k\}_{k \in \Z^d}$, the \emph{inverse Fourier Transform} is defined as, 
\begin{equation}
    \label{eq:ift}
    \cF^{-1}(\hat{v})(x)
:=
\sum_{k\in \Z^d} \hat{v}_k e^{i \langle k, x\rangle},
\quad
\forall \, x\in \T^d.
\end{equation}
Using the Fourier transform \eqref{eq:ft} and for $s \ge 0$, one can denote by $H^s(\T^d)$ the \define{Sobolev space} of functions $v\in L^2(\T^d)$, with Fourier coefficients $\{\hat{v}_{k}\}_{k\in \Z^d}$,  having a finite $H^s$-norm:
\begin{align} \label{eq:Hs}
\Vert
v
\Vert_{H^s}^2
:=
\frac{(2\pi)^{d}}{2} \sum_{k\in \Z^d} (1+|k|^{2s}) |\hat{v}_k|^2 < \infty.
\end{align}
Note that with this definition, we have from Parseval's identity, that $\Vert v \Vert_{H^0} = \Vert v \Vert_{L^2}$, so that $H^0(\T^d) = L^2(\T^d)$.

 We also introduce the corresponding \define{homogeneous Sobolev spaces} $\dot{H}^s(\T^d)$ (and $\dot{L}^2(\T^d) := \dot{H}^0(\T^d)$), consisting of functions $v(x)\in H^s(\T^d)$ and with zero mean $\fint_{\T^d} v(x) \, dx = \hat{v}_0 = 0$, and with norm
\begin{align} \label{eq:dotHs}
\Vert 
v
\Vert_{\dot{H}^s}
:=
\left(
(2\pi)^{d}
\sum_{k\in \Z^d\setminus \{0\}}
|k|^{2s} |\hat{v}_k|^2
\right)^{1/2}.
\end{align}

Given $N\in \N$, throughout this work, we will denote by $L^2_N(\T^d)$, the space of trigonometric polynomials $v_N: \T^d \to \R$, of the form
\begin{align} \label{eq:pN}
v_N(x)
=
\sum_{|k|_\infty\le N} c_k e^{i\langle x, k\rangle},
\end{align}
where the summation is over all $k=(k_1,\dots, k_d)\in \Z^d$ such that 
\[
|k|_{\infty}:= \max_{i=1,\dots, d} |k_i| \le N.
\]
The space $L^2_N(\T^d)$ is viewed as a normed vector space with norm $\Vert \slot \Vert_{L^2}$. Similarly, for $s\ge 0$, we denote by $H^s_N(\T^d)$ the normed vector space of trigonometric polynomials $v_N$ of degree $\le N$, with norm $\Vert \slot \Vert_{H^s}$.

We note that in order to ensure that $v_N(x) \in \R$ is real-valued for all $x\in \T^d$, the coefficients $c_k\in \C$ must satisfy the relations $c_{-k} = \overline{c_k}$ for all $|k|_\infty \le N$, and where $\overline{c_k}$ denotes the complex conjugate of $c_k$.

We denote by 
\begin{align}\label{eq:proj}
P_N: L^2(\T^d) \to L^2_N(\T^d),
\quad v \mapsto P_Nv,
\end{align}
the \define{$L^2$-orthogonal projection onto $L^2_N(\T^d)$}; or more explicitly,
\[
P_N
\left(
\sum_{k\in \Z^d} c_k e^{i\langle k, x\rangle}
\right)
=
\sum_{|k|_\infty \le N} c_k e^{i\langle k, x\rangle},
\quad
\forall \, (c_k)_{k\in \Z^d} \in \ell^2(\Z^d).
\]
In fact, the mapping $P_N$ defines a projection $H^s(\T^d) \to H^s_N(\T^d)$ for any $s\ge 0$. We have the following spectral approximation estimate: Let $s>0$ be given. There exists a constant $C = C(s,d)>0$, such that for any $v\in H^s(\T^d)$, we have
\begin{align}\label{eq:spectral}
\Vert v - P_N v \Vert_{H^\varsigma}
\le C N^{-(s-\varsigma)} \Vert v \Vert_{H^s}, 
\quad
\text{for any $\varsigma \in [0,s]$.}
\end{align}

We also define a natural projection 
\begin{align} \label{eq:dotproj}
\dot{P}_N: L^2(\T^d) \to \dot{L}^2_N(\T^d),
\end{align}
by removing the mean, i.e. $\dot{P}_Nv = P_N v - \fint_{\T^d} v(x) \, dx$, or equivalently:
\[
\dot{P}_N
\left(
\sum_{k\in \Z^d} c_k e^{i\langle k, x\rangle}
\right)
=
\sum_{0 < |k|_\infty \le N} c_k e^{i\langle k, x\rangle},
\quad
\forall \, (c_k)_{k\in \Z^d} \in \ell^2(\Z^d).
\]

Furthermore, we denote by by 
\begin{align} \label{eq:pproj}
\cI_N: C(\T^d) \mapsto L^2_N(\T^d),
\quad
u \mapsto \cI_Nu,
\end{align}
the \define{pseudo-spectral projection} onto $L^2_N(\T^d)$; we recall that the pseudo-spectral projection $\cI_Nv$ of a continuous function $v$ is defined as the unique trigonometric polynomial $\cI_Nv \in L^2_N(\T^d)$, such that 
\begin{align}
\label{eq:IN}
\cI_Nv(x_j) = v(x_j), 
\quad
\forall \, j \in \cJ_N,
\end{align}
where $\{x_j\}_{j\in \cJ_N}$ denotes the set of all regular grid points $x_j \in \Z^d$ of the form $x_j = 2\pi j/(2N+1) \in \T^d$, $j\in\Z^d$ (cp. equation \eqref{eq:grid}). 

We also recall the following embedding theorem for the Sobolev spaces $H^s(\T^d)$:
\begin{theorem} [Sobolev embedding] \label{thm:sobolev}
Let $d\in \N$. For any $s> d/2$, we have a compact embedding $H^s(\T^d) \embeds C(\T^d)$ into the space of continuous functions. In particular, there exists a constant $C = C(s,d)>0$, such that
\[
\Vert v \Vert_{L^\infty} \le C \Vert v \Vert_{H^s}, \quad \forall \, v\in H^s(\T^d).
\]
\end{theorem}
The Sobolev embedding theorem implies in particular that the pseudo-spectral projection $\cI_N$ is well-defined as an operator $\cI_N: H^s(\T^d) \to L^2_N(\T^d)$ for $s>d/2$. In the following theorem, we recall a well-known approximation error estimate for the pseudo-spectral projection $\cI_N$:

\begin{theorem}[Pseudo-spectral approximation estimate]
\label{thm:psest}
Let $d \in \N$. For any $s > d/2$ and $N\in \N$, the spectral interpolation operator $\cI_N: H^s(\T^d) \to L^2_N(\T^d)$ is well-defined. Furthermore, there exists a constant $C=C(s,d)>0$, such that the following approximation error estimate holds
\[
\Vert (1-\cI_N) v \Vert_{H^\varsigma}
\le 
C N^{-(s-\varsigma)} \Vert v \Vert_{H^s},
\quad
\forall \, v\in H^s(\T^d),
\]
 for any $\varsigma \in [0,s]$.	
\end{theorem}

Finally, for $N\in \N$, we fix a regular grid $\{x_j\}_{j\in \cJ_N}$ of values
\begin{align} \label{eq:grid}
x_j = \frac{2\pi j}{2N+1}, 
\end{align}
where the index $j\in \cJ_N$ belong to the index set 
\begin{align} \label{eq:index}
\cJ_N:=\{0,\dots, 2N\}^d.
\end{align}
Recall the set of Fourier wave numbers \eqref{eq:k} and we define the \define{discrete Fourier transform} $\cF_N: \R^{\cJ_N} \to \C^{\cK_N}$ by 
\begin{equation}
    \label{eq:dft}
\cF_N(v)(k)
:=
\frac{1}{(2N+1)^d}
\sum_{j\in \cJ_N}
v_j e^{-2\pi i \langle j, k\rangle/{N}},
\end{equation}
with inverse $\cF^{-1}_N: \C^{\cK_N} \to \R^{\cJ_N}$, 
\begin{equation}
    \label{eq:dift}
\cF^{-1}_N(\hat{v})(j)
:=
\sum_{k\in \cK_N}
\hat{v}_k e^{2\pi i \langle j, k \rangle / N}.
\end{equation}

\begin{lemma}[Periodic extension operator]
\label{lem:periodic_extension}
Let \(\Omega \subset \R^d\) be a bounded, Lipschitz domain. There exists a continuous, linear operator \(\mathcal{E} : W^{m,p}(\Omega) \to W^{m,p}(B)\) for any \(m \geq 0\) and \(1 \leq p \leq \infty\) where \(B \Subset \R^d\) is a hypercube with \(\Omega \subset B\)
such that, for any \(u \in W^{m,p}(\Omega)\),
\begin{enumerate}
	\item \(\mathcal{E}(u)|_\Omega = u\),
	\item \(\mathcal{E}(u)\) is periodic on \(B\).
\end{enumerate}
\end{lemma}
\begin{proof}
By \cite[Chapter 6, Theorem 5]{stein1970singular}, there exists a continuous, linear operator \(\tilde{\mathcal{E}} : W^{m,p}(\Omega) \to W^{m,p}(\R^d)\)
such that \(\tilde{\mathcal{E}}(u)|_\Omega = u\) for any \(u \in W^{m,p}(\Omega)\). Let \(\Phi \in C^\infty(\R^d)\) be a mapping
whose zeroth level set defines a curve \(\partial \Omega '\) such that \(\Omega \subset \Omega'\) and \(\text{dist}(\partial \Omega',  \partial \Omega) > 0\)  where \(\Omega'\) is a bounded domain with boundary \(\partial \Omega'\). Furthermore, suppose that the 
first level set of \(\Phi\) defines a curve \(\partial \Lambda\) such that \(\Omega' \subset \Lambda\)
and \(\text{dist}(\partial \Lambda,  \partial \Omega') > 0\)
where, again, \(\Lambda\) is a bounded domain with boundary \(\partial \Lambda\).
For example, we may take \(\Phi(x) = x_1^2 + \dots + x_d^2 - r\) for some \(r > 0\) large enough then \(\Omega'\) is \(d\)-ball of radius \(\sqrt{r}\) enclosing \(\Omega\), and \(\Lambda\) is a \(d\)-ball of radius \(\sqrt{1+r}\) clearly enclosing \(\Omega'\). We will now follow 
the approach of \cite{boyd2005fourier} and construct a windowing function \(\rho \in C^\infty_c (\R^d)\). In particular, we
define
\[\rho(x) = H \bigl ( 1 - 2 \Phi(x) \bigl ), \qquad \forall x \in \R^d\]
where 
\[H(t) = \frac{1}{2} \bigl (1 + S(t) \bigl ), \qquad \forall t \in \R\]
and 
\[
S(t) = \begin{cases}
-1, & t < -1 \\
\text{erf}\left( \frac{t}{\sqrt{1-t^2}} \right), & -1 \leq t \leq 1\\
1, & t > 1
\end{cases}, \qquad \forall t \in \R
\]
with \(\text{erf}\), the Gauss error function,
\[\text{erf}(t) = \frac{2}{\sqrt{\pi}} \int_0^t \text{e}^{-z^2} \: \text{d}z, \qquad \forall t \in \R.\]
Define \(B \Subset \R^d\) to be a large enough \(d\)-cube such that \(\Lambda \subset B\) and \(\text{dist}(\partial \Lambda, \partial B) >  0.\)
It is easy to check that \(\rho \in C^\infty_c(\R^d)\) and we have that \(\rho|_\Omega  = \rho|_{\Omega'} = 1\) and \(\rho|_{B \setminus \bar{\Lambda}} = 0\) with \(\rho\) and all of its 
derivatives vanishing on \(\partial \Lambda\). Furthermore,
\[\|\rho\|_{C^m(B)} < \infty\]
for any \(m \geq 0\). We can thus define \(\mathcal{E} : W^{m,p}(\Omega) \to W^{m,p}(B)\) by
\[\mathcal{E}(u) = \rho \tilde{\mathcal{E}}(u), \qquad \forall u \in W^{m,p}(\Omega).\]
By Leibniz's rule and the generalized triangle inequality,
there exists a constant \(C_1 = C_1(d,m,p) > 0\) (independent of \(p\) in the case \(p=\infty\)) such that
\[\|\mathcal{E}(u)\|_{W^{m,p}(B)} \leq C_1 \|\rho\|_{C^m(B)} \|\tilde{\mathcal{E}}(u)\|_{W^{m,p}(B)}.\]
Since \(\tilde{\mathcal{E}}\) is bounded, there is a constant \(C_2 > 0\) such that
\begin{align*}
\|\mathcal{E}(u)\|_{W^{m,p}(B)} &\leq C_1 \|\rho\|_{C^m(B)} \|\tilde{\mathcal{E}}(u)\|_{W^{m,p}(B)} \\ 
&\leq C_1 \|\rho\|_{C^m(B)} \|\tilde{\mathcal{E}}(u)\|_{W^{m,p}(\R^d)} \\ 
&\leq C_1 C_2 \|\rho\|_{C^m(B)} \|u\|_{W^{m,p}(\Omega)}
\end{align*}
hence \(\mathcal{E}\) is a continuous, linear operator. Since \(\rho|_\Omega = 1\), 
we immediately have \(\mathcal{E}(u)|_\Omega = \tilde{\mathcal{E}}(u)|_\Omega = u\). Since \(\text{ supp}(\partial^\alpha \rho) \subseteq \bar{\Lambda}\) for any multi-index \(0 \leq |\alpha| \leq m\), we conclude 
that \(\partial^{\alpha} \mathcal{E}(u) |_{B \setminus \bar{\Lambda}} = 0\)  hence \(\mathcal{E}(u)\) is periodic on \(B\) as desired.

\end{proof}

\section{Non-linear lifting and projection operators}

\begin{lemma} \label{lem:linear-to-nonlinear}
Assume that the activation function $\sigma \in C^2$ is non-linear and that $\Vert \sigma \Vert_{C^2(\R)} < \infty$. Let $\cN: L^2(D;\R^{d_a}) \to L^2(D;\R^{d_u})$ be a neural operator of the form $\cN = \cQ \circ \cL_{L}\circ \dots \circ \cL_{1} \circ \cR$, with Lipschitz continuous layers $\cL_{\ell}: L^2(D;\R^{d_v}) \to L^2(D;\R^{d_v})$, and with linear lifting/projection operators $\cR$, $\cQ$ of the form \eqref{eq:no-r} and \eqref{eq:no-q}, respectively. For any $K\subset L^2(D;\R^{d_v})$ compact and $\epsilon > 0$, there exist neural networks $\hat{R}: \R^{d_a}\times D \to \R^{d_v}$, $\hat{Q}: \R^{d_v}\times D \to \R^{d_u}$ with a single hidden layer and $\width(\hat{R}), \width(\hat{Q}) \le 2d_v$, such that the mapping $\hN: \cA(D;\R^{d_a}) \to \cU(D;\R^{d_u})$ defined by
\[
\hN(a) := \hQ \circ \cL_{L} \circ \dots \circ \cL_{1} \circ \hR(a),
\]
where $\cR(a)(x) := \hat{R}(a(x),x)$, $\cQ(v) := \hat{Q}(v(x),x)$, approximates $\cN$ to order $\epsilon$:
\[
\sup_{a\in K} 
\Vert 
\cN(a) - \hN(a)
\Vert_{L^2}
\le \epsilon.
\]
\end{lemma}

\begin{proof}
We recall that, by our definition of a neural operator, we have $\cR(a)(x) = Ra(x)$, $\cQ(v)(x) = Qv(x)$, where $R\in \R^{d_v\times d_a}$, $Q \in \R^{d_u\times d_v}$. By assumption on the non-linearity of $\sigma$, there exists $z_0 \in \R$, such that $\sigma'(z_0) \ne 0$. For any $h>0$, we now define
\begin{align*}
\left\{
\begin{aligned}
\hat{Q}_h(v,x) 
&:= 
\frac{\sigma(z_0 + h Qv) - \sigma(z_0 - hQv)}{2h \sigma'(z_0)}, 
\\
\hat{R}_h(a,x) 
&:= 
\frac{\sigma(z_0 + h Ra) - \sigma(z_0 - hRa)}{2h \sigma'(z_0)}.
\end{aligned}
\right.
\end{align*}
We remark that in the above expressions, the addition of $z_0 \in \R$ and $hQv\in \R^{d_u}$, $hRa\in \R^{d_v}$ is carried out componentwise, as is the evaluation of $\sigma$. 

Since the interior layers $\cL_1,\dots, \cL_{L}$ are fixed, we introduce the short-hand notation $\cL := \cL_{L}\circ \dots \circ \cL_{1}$. By assumption on the activation function and the layers, we have that $\cL: L^2(D;\R^{d_v}) \to L^2(D;\R^{d_v})$ is a Lipschitz continuous mapping. Our goal is to show that for any $K\subset L^2$ compact and $\epsilon >0$, there exists a sufficiently small $h>0$, such that
\[
\sup_{a\in K} 
\Vert
\cN(a) - \hN_h(a)
\Vert_{L^2}
\le\epsilon,
\]
where $\hN_h(a) := \hQ_h \circ \cL \circ \hR_h(a)$. To this end, we note that for any $a\in K$, we have
\begin{align*}
\Vert
\cN(a) - \hN_h(a)
\Vert_{L^2}
&=
\Vert
\cQ\circ \cL \circ \cR(a) - \hQ_h\circ \cL \circ \hR_h(a)
\Vert_{L^2}
\\
&\le
\Vert
\cQ\circ \cL \circ \cR(a) - \hQ_h\circ \cL \circ \cR(a)
\Vert_{L^2}
\\
&\qquad +
\Vert
\hQ_h\circ \cL \circ \cR(a) - \hQ_h\circ \cL \circ \hR_h(a)
\Vert_{L^2}.
\end{align*}

Introducing $\tilde{K} := \cL \circ \cR(K)$, we note that $\tilde{K} \subset L^2$ is compact, and 
\begin{align*}
\sup_{a\in K} \Vert \cN(a) - \hN_h(a) \Vert_{L^2}
&\le
\sup_{v\in \tilde{K}} \Vert \cQ(v) - \hQ_h(v) \Vert_{L^2}
\\
&\qquad +
\Lip(\hQ_h) \Lip(\cL) \sup_{a\in K} \Vert \cR(a) - \hR_h(a) \Vert_{L^2}.
\end{align*}
The proof of this lemma thus follows form the following three claims:

\textbf{Claim 1:}
There exists $C>0$, independent of $h$, such that $\Lip(\hQ_h) \le C$ for all $h>0$.

\textbf{Claim 2:}
For any $\epsilon > 0$ and $\tilde{K}\subset L^2$ compact, there exists $h>0$, such that 
\[
\sup_{v\in \tilde{K}}\Vert \hQ_h(v) - \cQ(v) \Vert_{L^2} \le \epsilon.
\]

\textbf{Claim 3:}
For any $\epsilon > 0$ and $K\subset L^2$ compact, there exists $h>0$, such that 
\[
\sup_{a\in K}\Vert \hR_h(a) - \cR(a) \Vert_{L^2} \le \epsilon.
\]

Clearly, the only difference between Claims 2 and 3 is notational, hence it suffices show Claims 1 and 2 to conclude the proof of the present lemma.

\emph{Proof of Claim 1:}
We note that for any $v(x),v'(x)\in \R^{d_v}$, we have
\begin{align}
|\hat{Q}_h(v(x),x) - \hat{Q}_h(v'(x),x)|
&\le
\frac{1}{2h|\sigma'(z_0)|} 2\Lip(\sigma) |hQ(v(x) - v'(x))|
\\
&\le
\frac{\Lip(\sigma) \Vert Q \Vert}{|\sigma'(z_0)|} |v(x)-v'(x)|,
\end{align}
where $|\slot|$ denotes the Euclidean norm and $\Vert Q \Vert$ denotes the operator norm of $Q$. Hence, it follows that for any $v,v'\in L^2(D;\R^{d_v})$, we have
\[
\Vert \hQ_h(v) - \hQ_h(v') \Vert_{L^2}
\le
\frac{\Lip(\sigma) \Vert Q \Vert}{|\sigma'(z_0)|} \Vert v - v'\Vert_{L^2},
\]
and thus,
\[
\Lip(\hQ_h) \le \frac{\Lip(\sigma) \Vert Q \Vert}{|\sigma'(z_0)|},
\]
is bounded independently of $h>0$.

\emph{Proof of Claim 2:}
It follows form Taylor expansion that $\hat{Q}_h(v(x),x)=Qv(x) + h R(h;Qv(x))$, where the remainder $R$ satisfies a uniform bound
\[
|R(h;Qv(x))|\le \frac{C\Vert \sigma \Vert_{C^2}}{|\sigma'(z_0)|} |Qv(x)|^2,
\quad
\forall h\in (0,1], \; v(x) \in \R^{d_v}.
\]
In particular, considering the $L^2$-norm of the left-hand side and assuming $x\mapsto v(x)$ to be \emph{bounded}, we conclude that there exists a constant $C = C(\sigma, \Vert Q \Vert)>0$, depending only on the activation function $\sigma$ and the (Euclidean) operator norm $\Vert Q\Vert$ of $Q: (\R^{d_v},|\slot|) \to (\R^{d_u},|\slot|)$, such that 
\begin{align} \label{eq:remainder}
\Vert \hQ_h(v) - \cQ(v) \Vert_{L^2}
=
\Vert R(h;Qv) \Vert_{L^2} \le C \Vert v\Vert_{L^\infty} \Vert v \Vert_{L^2}, \quad \forall \, v \in L^2\cap L^\infty.
\end{align} 
Next, for $M>0$ we introduce a cut-off operator $p_M: L^2\to L^2\cap L^\infty$, by \[
p_M(v(x)) :=
\begin{cases}
v(x), & (|v(x)|\le M), \\
0, & (|v(x)| > M).
\end{cases}
\]
Since $\tilde{K}\subset L^2$ is compact, it follows that
\[
\lim_{M\to \infty} \sup_{v\in \tilde{K}} \Vert v - p_M(v) \Vert_{L^2} = 0.
\]
We can thus choose $M>0$ sufficiently large, such that 
\[
\sup_{v\in \tilde{K}}
\Vert v - p_M(v) \Vert_{L^2}
\le
\frac{\epsilon}{2\left(
\sup_{h\in (0,1]} \Lip(\hQ_h) + \Lip(\cQ)
\right)},
\]
where we note that $\sup_{h\in (0,1]} \Lip(\hQ_h) <\infty$, by Claim 1.
Furthermore, and again by compactness of $\tilde{K}\subset L^2$, there exists a constant $\tilde{M} > 0$, such that 
\[
\sup_{v\in \tilde{K}} \Vert v \Vert_{L^2} \le \tilde{M}.
\]
With this $\tilde{M}$, our choice of $M$ and the estimate \eqref{eq:remainder} on the remainder $R$, it follows that for any $v\in \tilde{K}$:
\begin{align*}
\Vert \hQ_h(v) - \cQ(v) \Vert_{L^2}
&\le
\Vert \hQ_h(v) - \hQ_h(p_M(v)) \Vert_{L^2}
\\
&\qquad +\Vert \hQ_h(p_M(v)) - \cQ(p_M(v)) \Vert_{L^2}
\\
&\qquad
+\Vert \cQ(p_M(v)) - \cQ(v) \Vert_{L^2}
\\
&\le
\left(\Lip(\hQ_h) + \Lip(\cQ)\right) \Vert v - p_M(v) \Vert_{L^2}
\\
&\qquad +
Ch\Vert p_M(v) \Vert_{L^\infty} \Vert p_M(v) \Vert_{L^2}
\\
&\le
\epsilon/2
+
CM \tilde{M} h.
\end{align*}
Finally, choosing $h = \epsilon/(2CM\tilde{M})$ implies that
\[
\sup_{v\in \tilde{K}}
\Vert \hQ_h(v) - \cQ(v) \Vert_{L^2}
\le \epsilon.
\]
This concludes the proof.

\end{proof}

\section{Proofs and Technical details for section \ref{sec:2}}

\subsection{
Proof of Lemma \ref{lem:fut0}
}
\label{app:pfut0}

The proof of Lemma \ref{lem:fut0} will rely on the following technical lemma:

\begin{lemma}\label{lem:proj-layer}
Let $s'\ge 0$, $N\in \N$ be given. Let $K\subset H^{s'}$ be compact, and assume that $\sigma\in C^m$, where $m>s'$ is integer. Then for any $\epsilon > 0$, there exists a single-layer FNO $\cL: H^{s'} \to H^{s'}$, such that
\[
\sup_{v\in K} \Vert P_N v - \cL(v) \Vert_{H^{s'}} \le \epsilon.
\] 
\end{lemma}

\begin{proof}[Proof of Lemma \ref{lem:proj-layer}]
First, we note that the Fourier projection $P_N: H^{s'}\to H^{s'}$ is a continuous operator, and hence the image $P_N K\subset H^{s'}$ is compact. Furthermore, $P_N$ maps into a finite-dimensional subspace of $H^{s'}$. Due to the norm-equivalence on finite-dimensional spaces, there thus exists $C_0 = C_0(N,K) > 0$, such that 
\begin{align} \label{eq:C0bd}
\sup_{v\in K} \Vert P_N v\Vert_{L^\infty} \le C_0, \quad \sup_{v\in K}\Vert P_N v \Vert_{H^m} \le C_0. 
\end{align}
Let $x_0\in \R$ be such that $\sigma'(x_0) \ne 0$. We define, for $h>0$,
\begin{align}
\psi_h(x) := \frac{\sigma(x_0 + hx) - \sigma(x_0-hx)}{2h\sigma'(x_0)}.
\end{align}
One readily shows that $\psi_h \in C^m$, and that there exists a constant $C_1 = C_1(\sigma, C_0)>0$, such that
\begin{align}
\Vert \psi_h \Vert_{C^m([-C_0,C_0])} \le C_1, \quad \forall \, h\in (0,1].
\end{align}
Furthermore, by Taylor expansion, we have 
\begin{align} \label{eq:approxbd}
|\psi_h(x) - x| \le Ch, \quad \forall \, x\in [-C_0,C_0], \; \forall \, h\in (0,1].
\end{align}
By the composition rule for Sobolev functions, we have that $\psi_h\circ P_Na\in H^m$, for $P_Na\in H^m$, and there exists a constant $C_2 = C_2(C_1,C_0)>0$, such that 
\begin{align} \label{eq:Hmbd0}
\Vert \psi_h(P_Nv) \Vert_{H^m} \le C_2, 
\quad
\forall\, v\in K.
\end{align}
We finally observe that the mapping $v\mapsto \cL_h(v) := \psi_h(P_Nv)$ can be represented by a single-layer FNO, and by \eqref{eq:Hmbd0}, we have
\begin{align} \label{eq:Hmbd}
\Vert \cL_h(v) \Vert_{H^m} \le C_2, 
\quad
\forall\, v\in K.
\end{align}
From the interpolation inequality between Sobolev spaces, it follows that
\begin{align*}
\Vert \cL_h(v) - P_Nv \Vert_{H^{s'}}
&\le 
\Vert \cL_h(v) - P_Nv \Vert_{L^2}^\theta
\Vert \cL_h(v) - P_Nv \Vert_{H^{m}}^{1-\theta},
\end{align*}
where $\theta = 1-s'/m > 0$. By \eqref{eq:Hmbd} and \eqref{eq:C0bd}, the second factor can be bounded independently of $h>0$. By \eqref{eq:approxbd} and \eqref{eq:C0bd}, we have 
\[
\Vert \cL_h(v) - P_Nv \Vert_{L^2}
=
\Vert \psi_h(P_Nv) - P_Nv \Vert_{L^2}
\le
Ch,
\]
for a constant $C>0$, independent of $h$ and $v\in K$. We conclude that
\[
\Vert \cL_h(v) - P_Nv \Vert_{H^{s'}}
\le C h^\theta \to 0,
\]
as $h\to 0$, for some constant $C>0$, independent of $h$. This proves the claim.
\end{proof}

We can now prove:
\begin{proof}[Proof of Lemma \ref{lem:fut0}]
Let $\cG: H^s\to H^{s'}$ be a continuous operator, and let $K\subset H^s$ be compact. We assume that FNOs are universal approximators of operators $H^s\to L^2$, and we wish to show that for any $\epsilon > 0$, there exists a FNO approximation of $\cG: H^s \to H^{s'}$ to accuracy $\epsilon$. We first note that by the compactness of $\cG(K)\subset H^{s'}$, there exists $N\in \N$, such that 
\begin{align}\label{eq:GGNbd}
\sup_{a\in K}\Vert \cG(a) - P_N\cG(a) \Vert_{H^{s'}} \le \epsilon/2.
\end{align}
Fix $\delta > 0$ for the moment. A suitable choice of $\delta$ will be specified at the end of this proof. By assumption on the universal approximation of operators $H^s\to L^2$, there exists a FNO $\tN: H^s\to L^2$, continuous as an operator $H^s\to L^2$, such that 
\begin{align}\label{eq:tNbd}
\sup_{a\in K}\Vert P_N\cG(a) - \tN(a) \Vert_{L^2} \le \delta.
\end{align}
One difficulty in the present construction is that there is no guarantee that $\tN$ defines a mapping $H^s\to H^{s'}$, and indeed for $s'>s$ this is not generally the case. We circumvent this issue by composing with an additional FNO layer $\tL: L^2 \to H^{s'}$. By Lemma \ref{lem:proj-layer}, there exists a single-layer FNO $v\mapsto \tL(v)$, satisfying the identity 
\begin{align}\label{eq:tLid}
\tL(v) = \tL(P_Nv),
\end{align}
 for all $v$, and defining a continuous operator $H^{s'}\to H^{s'}$, such that 
\begin{align}\label{eq:tLbd}
\sup_{v\in K'} \Vert P_N v - \tL(v) \Vert_{H^{s'}} \le \delta,
\end{align}
where $K' := P_N\tN(K) \subset H^{s'}$ is a compact subset\footnote{We note that $\cN(K)\subset L^2$ is compact as the continuous image of $K$, and that $P_N: L^2\to H^{s'}$ defines a continuous mapping for fixed $N\in \N$.} 
of $H^{s'}$. Next, we define a new FNO by the composition $\cN := \tL\circ \tN: H^s \to H^{s'}$. $\cN$ is a continuous operator $H^s\to H^{s'}$, since it can be written as the composition
\[
H^s \overset{\tN}{\longrightarrow} L^2 \overset{P_N}{\longrightarrow} H^{s'} \overset{\tL}{\longrightarrow} H^{s'},
\]
of continuous operators. We observe that for any $a\in K$, we have the following bound:
\begin{align*}
\Vert P_N\cG(a) - \cN(a) \Vert_{H^{s'}}
&\le
\Vert P_N \cG(a) - P_N\tN(a) \Vert_{H^{s'}}
 + \Vert P_N \tN(a) - \cN(a) \Vert_{H^{s'}}
\\
&\le
CN^{s'}\Vert P_N \cG(a) - \tN(a) \Vert_{L^2}
 + \Vert P_N \tN(a) - \tL(P_N\tN(a)) \Vert_{H^{s'}},
\end{align*}	
having made use of the inequality $\Vert P_Nv \Vert_{H^{s'}} \le CN^{s'}\Vert P_N v\Vert_{L^2}$ for a constant $C = C(\T^d,s')>0$ independent of $N$, and the fact that $\cN(a) = \tL(\tN(a)) = \tL(P_N\tN(a))$ (cp. \eqref{eq:tLid}). Using \eqref{eq:tNbd}, we can estimate
\[
CN^{s'}\Vert P_N \cG(a) - \tN(a) \Vert_{L^2} \le CN^{s'} \delta.
\]
The bound \eqref{eq:tLbd} implies that $\Vert P_N \tN(a) - \tL(P_N\tN(a)) \Vert_{H^{s'}}=\Vert P_N v - \tL(v) \Vert_{H^{s'}} 
\explain{\le}{\eqref{eq:tLbd}} \delta$, with $v:= P_N\cN(a) \in K'$. We thus obtain
\begin{align}
\Vert P_N \G(a) - \cN(a) \Vert_{H^{s'}} \le (CN^{s'} + 1) \delta,
\end{align}
where $C = C(\T^d,s')>0$ is independent of $\delta$. Since $\delta>0$ was arbitrary, we can ensure that $(CN^{s'} + 1)\delta \le \epsilon/2$. From this estimate, and the bound \eqref{eq:GGNbd}, we conclude that there exists a FNO $\cN: H^{s}\to H^{s'}$, such that
\begin{align*}
\sup_{a\in K}\Vert \cG(a) - \cN(a) \Vert_{H^{s'}} 
&\le 
\sup_{a\in K}\Vert \cG(a) - P_N\cG(a) \Vert_{H^{s'}} \\
&\qquad + \sup_{a\in K}\Vert P_N\cG(a) - \cN(a) \Vert_{H^{s'}} 
\\
&\le \epsilon,
\end{align*}
This concludes our proof.
\end{proof}

\subsection{Proof of Lemma \ref{lem:fno-ft}}
\label{app:pfut1}

\begin{proof}
\textbf{Step 1:}
In this first step, for any $\epsilon > 0$, we will construct a FNO
\[
\cN_1: L^2(\T^d; \R) \to L^2(\T^d; \R^{2\cK_N}),
\]
such that 
\begin{align} \label{eq:N1}
\left\{
\begin{aligned}
\Vert \cN_1(v)_{1,k} - P_Nv(x)\, \cos(\langle k, x\rangle) \Vert_{L^\infty} < \epsilon, \\
\Vert \cN_1(v)_{2,k} - P_Nv(x)\, \sin(\langle k,x\rangle) \Vert_{L^\infty} < \epsilon,
\end{aligned}
\right.
\quad \forall \, {k\in\cK_N}.
\end{align}

To see how to construct such $\cN_1$, we first define a lifting 
\[
\cR_1: L^2(\T^d;\R) \to L^2(\T^d;\R^{4\cK_N}), \quad
v(x) \mapsto \hat{w}(x),
\]
where $\hat{w}(x) := \{ (v(x), 0, v(x), 0) \}_{k\in \cK_N} \in \R^{4\cK_N}$ for any $x\in \T^d$. In the following, we will identify $\R^{4\cK_N}\simeq \R^{d_v}$, with $d_v = 4|\cK_N|$. Next, we define the inner part of the first FNO layer (i.e. the matrix $W$, multiplier $P$, bias $b(x)$ in \eqref{eq:fno-layer}), such that $P(k) \equiv 1_{|k|\le N} \bm{1}_{d_v\times d_v}$ either vanishes (for $|k|>N$) or is the unit matrix (for $|k|\le N$), $W \equiv 0$ is zero and the bias function $b(x) := \{(0,b_{2,k}(x),0,b_{4,k}(x))\}_{k\in \cK_N}$, $b_{2,k}(x) = \cos(\langle k, x\rangle)$, $b_{4,k}(x) = \sin(\langle k, x\rangle)$, such that
\begin{align*}
\hL_1(\hat{v})(x)
&:=
W \hat{w}(x) + b(x) + \cF^{-1} (P \cF \hat{w})(x)
\\
&=
\begin{Bmatrix}
\big(\,
P_N\hat{w}_{1,k}(x),\,
\cos(\langle k, x\rangle),\,
P_N\hat{w}_{3,k}(x),\,
\sin(\langle k, x\rangle)
\,\big)
\end{Bmatrix}_{k\in \cK_N}.
\end{align*}
Next, we recall that by assumption, we have $\Vert v \Vert_{L^2} \le B$, and by construction, we have $\hat{w}_{1,k}(x) = \hat{w}_{3,k}(x) = v(x)$, implying that \[
\Vert P_N \hat{w}_{1,k} \Vert_{L^\infty}
=
\Vert P_N \hat{w}_{3,k} \Vert_{L^\infty}
=
\Vert P_N v \Vert_{L^\infty}
\le
C \Vert P_N v \Vert_{L^2}
\le
C B,
\]
where $C = C(N) \propto N^{d/2}$ is a constant depending on $N$. By the universal approximation theorem for (ordinary) neural networks, there exists a neural network 
\[
\hN: [-CB,CB]\times [-1,1]\times [-CB,CB]\times [-1,1] \to \R^2
\]
with activation function $\sigma$, such that $\max_{(a,b,c,d)} |\hN(a,b,c,d) - (ab,cd)| < \epsilon$, and where the maximum is taken over all $a,c \in [-CB,CB]$, $b,d\in [-1,1]$. But then, the point-wise mapping \begin{align*}
\cN_1 = \hN \circ \hL\circ \cR_1: 
\left\{
\begin{aligned}
&L^2(\T^d;\R) \to L^2(\T^d;\R^{2\cK_N}), 
\\
&v(x) \mapsto \tilde{w}(x) := \hL_1(\cR_1(v))(x) \mapsto \hN(\tilde{w}(x)),
\end{aligned}
\right.
\end{align*}
satisfies \eqref{eq:N1}, and $\cN_1$ can be represented by a FNO (cp. Remark \ref{rem:loc}).
\textbf{Step 2:} By the last step, we have $\Vert \cN_1(v)_{1,k} - P_Nv(x) \, \cos(\langle k, x\rangle)\Vert_{L^\infty} < \epsilon$ and $\Vert \cN_1(v)_{2,k} - P_Nv(x) \, \sin(\langle k, x\rangle)\Vert_{L^\infty} < \epsilon$ for all $v\in L^2(\T^d)$ with $\Vert v \Vert_{L^2}\le B$.  We next note that, since 
\begin{align*}
P_Nv(x) 
&= \sum_{|k|_\infty\le N} \hat{v}_k e^{i\langle k, x\rangle},
\end{align*}
and since $v(x)$ is real-valued, the Fourier coefficients $\hat{v}_k$ satisfy $\hat{v}_{-k} = \overline{\hat{v}_k}$, where $\overline{\hat{v}_k}$ denotes the complex conjugate of ${\hat{v}_k}$, and 
\[
\fint_{\T^d} P_Nv(x) \cos(\langle k, x\rangle) \, dx
=
\Re(\hat{v}_k),
\quad
\fint_{\T^d} P_Nv(x) \sin(\langle k, x\rangle) \, dx
=
-\Im(\hat{v}_k),
\]
In particular, this implies that the $0$-th Fourier modes of $P_Nv(x) \, \cos(\langle k, x\rangle)$, and $P_Nv(x) \, \sin(\langle k, x \rangle)$, respectively, are given by
\begin{align*}
\cF\Big[P_Nv \cos(\langle k, \slot\rangle)\Big](0)
=
\Re(\hat{v}_k),
\quad
\cF\Big[P_Nv \sin(\langle k, \slot\rangle)\Big](0)
=
-\Im(\hat{v}_k),
\end{align*}
and, as a consequence, we have with the Fourier multiplier 
\[
\delta_{0}(k') = 
\begin{cases}
1, & (k'=0) \\
0, & (k'\ne 0),
\end{cases}
\quad 
\forall \, k'\in \Z^d,
\]
 and with $w(x) := \cN_1(v)(x) \in L^2(\T^d;\R^{2\cK_N})$, written as $w(x) = \{(w_{1,k}(x), w_{2,k}(x))\}_{k\in \cK_N}$,  that 
 \begin{align*}
\left\Vert 
\cF^{-1}\Big(\delta_0(k')\cF(w_{1,k})(k')\Big)(x) - \Re(\hat{v}_k) 
\right\Vert_{L^\infty} < \epsilon, \\
\left\Vert 
\cF^{-1}\Big(\delta_0(k')\cF(w_{2,k})(k')\Big)(x) + \Im(\hat{v}_k) 
\right\Vert_{L^\infty} < \epsilon,
\end{align*}
for all $v\in L^2(\T^d)$ with $\Vert v \Vert_{L^2}\le B$. We use this observation to define a suitable FNO layer; with local matrix $W = 0$, bias $b(x) = 0$ and a Fourier multiplier matrix $P: \Z^d \to \C^{d_v\times d_v}$, $k' \mapsto P(k')$ (where $d_v = 2|\cK_N|$, and $\C^{d_v} \simeq \C^{2\cK_N}$), with entries 
\[
[P(k')]_{(\ell,k),(\tilde{\ell},\tilde{k})} := \delta_{0}(k') \left\{\delta_{\ell=1}(\ell) - \delta_{\ell=2}(\ell)\right\} \bm{1}_{d_v\times d_v}.
\]
With this definition of the Fourier multiplier $P$, we define $\hL_2: L^2(\T^d; \R^{2\cK_N}) \to L^2(\T^d; \R^{2\cK_N})$, by
\[
\hL_2(w) := Ww(x) + b(x) + \cF^{-1} \Big(P\cF w\Big)(x).
\]
Then, by construction, we have for any $v\in L^2(\T^d)$ with $\Vert v \Vert_{L^2} \le B$ that the output $w := \hL_2(\cN_1(v))$, $w(x) = \{(w_{1,k}(x),w_{2,k}(x)) \}_{k\in \cK_N}$, satisfies 
\[
\left\{
\begin{aligned}
\left\Vert
w_{1,k}(x) - \Re(\hat{v}_k)
\right\Vert_{L^\infty}
&< \epsilon, \\
\left\Vert
w_{2,k}(x) - \Im(\hat{v}_k)
\right\Vert_{L^\infty}
&< \epsilon.
\end{aligned}
\right.
\]
We also note that $\hL_2$ outputs only constant functions, by construction.
This is almost the desired result, except for the fact that the composition $\hL_2 \circ \cN_1$ does not define a FNO, since the linear layer $\hL_2$ is missing the non-linearity $\sigma$. This can be rectified by composition with a suitable (ordinary) $\sigma$-neural network $\tN$, which approximates the identity: Indeed, the last estimate implies in particular that $|w_{\ell,k}| \le |\hat{v}_k| + \epsilon \le B + \epsilon$ for all $\ell$, $k$. By the (ordinary) universal approximation theorem, there exists a neural network $\tN: \R^{2\cK_N} \to \R^{2\cK_N}$, such that \[
\Vert \tN(w) - w \Vert_{\ell^\infty} < \epsilon,
\quad 
\forall w\in \R^{2\cK_N}, \; \Vert w \Vert_{\ell^\infty} \le B + \epsilon.
\]
Then the composition $\cN_2 := \tN \circ \hL_2: L^2(\T^d;\R^{2\cK_N}) \to L^2(\T^d;\R^{2\cK_N})$, given by $w(x) \mapsto \tN\left(\hL_2(w)(x)\right)$, \emph{does} define a FNO.

\textbf{Step 3:} We finally observe that since both $\cN_1$ and $\cN_2$ can be represented by FNOs, then also their composition 
\[
\cN := \cN_2 \circ \cN_1: L^2(\T^d) \to L^2(\T^d;\R^{2\cK_N}),
\]
can be represented by a FNO with $\depth(\cN) = \depth(\cN_1) + \depth(\cN_2)$, $\width(\cN) = \max_{j=1,2} \width(\cN_j)$, $\lift(\cN) = \max_{j=1,2} \lift(\cN_j)$. Furthermore, the mapping $v\mapsto \cN(v)$ maps to constant functions, and for any $v\in L^2(\T^d)$, $\Vert v \Vert_{L^2} \le B$, we have
\[
\left\{
\begin{aligned}
\left\Vert
\Re(\hat{v}_k) - \cN(v)_{1,k}
\right\Vert_{L^\infty} 
< 2\epsilon, \\
\left\Vert
\Im(\hat{v}_k) - \cN(v)_{2,k}
\right\Vert_{L^\infty}
< 2\epsilon.
\end{aligned}
\right.
\]
Since $\epsilon > 0$ was arbitrary, the claim follows.
\end{proof}

\subsection{Proof of Lemma \ref{lem:fno-ift}}
\label{app:pfut2}
\begin{proof}
\textbf{Step 1:} We first construct a FNO $\cN_1: L^2(\T^d;\R^{2\cK_N}) \to L^2(\T^d;\R^{2\cK_N})$, such that 
\begin{gather} \label{eq:0N1}
\left\{
\begin{aligned}
\left\Vert
\cN_1(w)_{1,k}
-
P_N w_{1,k}(x) \cos(\langle k, x\rangle)
\right\Vert_{L^\infty} < \left(2|\cK_N||\T^d|^{1/2}\right)^{-1}\,\epsilon, 
\\
\left\Vert
\cN_1(w)_{2,k}
-
P_N w_{2,k}(x) \sin(\langle k, x\rangle)
\right\Vert_{L^\infty} < \left(2|\cK_N||\T^d|^{1/2}\right)^{-1}\,\epsilon, 
\end{aligned}
\right.
\end{gather}
for all $w(x) = (w_{1,k}(x), w_{2,k}(x))$, such that $\Vert w_{1,k}\Vert_{L^\infty}, \, \Vert w_{2,k}\Vert_{L^\infty} \le B$, for all $k\in \cK_N$. Here $|\T^d|=(2\pi)^d$ denotes the Lebesgue measure of $\T^d$ and $|\cK_N|=(2N+1)^d$ the number of elements of $\cK_N$. 

To construct such $\cN_1$, we first define a lifting $\cR: L^2(\T^d;\R^{2\cK_N}) \to L^2(\T^d;\R^{4\cK_N})$, $\{(w_{1,k}(x),w_{2,k}(x))\}_{k\in \cK_N} \mapsto \{(w_{1,k}(x),0,w_{2,k}(x),0)\}_{k\in \cK_N}$, followed by a linear layer $\cL: L^2(\T^d;\R^{4\cK_N}) \to L^2(\T^d;\R^{4\cK_N})$, which only introduces a bias (setting $W\equiv 0$, $b(x) = \{(0,\cos(\langle k,x\rangle),0,\sin(\langle k,x\rangle))\}_{k\in \cK_N}$, $P\equiv 0$), to yield 
\[
\cL\circ \cR(w) = \{(w_{1,k}(x),\cos(\langle k,x\rangle),w_{2,k}(x),\sin(\langle k,x\rangle))\}_{k\in \cK_N},
\]
for all $w \in L^2(\T^d;\R^{2\cK_N})$. There exists an ordinary neural network $\hN: [-B,B] \times [-1,1] \times [-B,B]\times [-1,1] \to \R^2$, such that $\max_{a,b,c,d} |\hN(a,b,c,d) - (ab,cd)| < (2|\cK_N||\T^d|^{1/2})^{-1} \epsilon$, where the maximum is taken over $a,c\in [-B,B]$ and $b,d \in [-1,1]$. Since, by assumption, we have $\Vert w_{1,k} \Vert_{L^\infty}, \Vert w_{2,k}\Vert_{L^\infty} \le B$ for the inputs of interest, we conclude that 
$\cN_1(w) := \hN \circ \cL \circ \cR(w)$ satisfies \eqref{eq:0N1}. Furthermore, by Lemma \ref{lem:sigma} and Lemma \ref{lem:composition}, $\cN_1$ is represented by a FNO.

\textbf{Step 2:} We define a projection $\cQ$ (cp. the definition of neural operators \eqref{eq:no-q}), mapping $w(x) \in L^2(\T^d);\R^{2\cK_N})$ to a scalar-valued function, by
\[
\cQ: L^2(\T^d;\R^{2\cK_N}) \to L^2(\T^d), \quad w(x) \mapsto \sum_{k\in \cK_N} \left(w_{1,k}(x) - w_{2,k}(x)\right).
\]
Then $\cN := \cQ \circ \cN_1: L^2(\T^d;\R^{2\cK_N}) \to L^2(\T^d)$ is an FNO, and for any $v\in L^2_N(\T^d)$, and $w = \{(\Re(\hat{v}_k),\Im(\hat{v}_k))\}_{k\in \cK_N}$ defined as in the statement of this lemma, we have
\begin{align*}
\Vert 
v - \cN(w)
\Vert_{L^2}
&=
\left\Vert 
v - \sum_{k\in \cK_N} \Big( 
\cN_1(w)_{1,k}
- \cN_1(w)_{2,k}
\Big)
\right\Vert_{L^2}
\\
&\le 
\sum_{k\in \cK_N} 
\left\Vert 
\cN_1(w)_{1,k} - \Re(\hat{v}_k) \cos(\langle k,x \rangle)
\right\Vert_{L^2}
\\
&\qquad +
\sum_{k\in \cK_N} 
\left\Vert 
\cN_1(w)_{2,k} - \Im(\hat{v}_k) \sin(\langle k,x \rangle) 
\right\Vert_{L^2}
\\
&\le 
\sum_{k\in \cK_N} 
\left\Vert 
\cN_1(w)_{1,k} - \Re(\hat{v}_k) \cos(\langle k,x \rangle)
\right\Vert_{L^2}
\\
&\qquad +
\sum_{k\in \cK_N} 
\left\Vert 
\cN_1(w)_{2,k} - \Im(\hat{v}_k) \sin(\langle k,x \rangle) 
\right\Vert_{L^2}
\\
&\le
|\cK_N| |\T^d|^{1/2} 
\max_{k\in \cK_N} \left\Vert \cN_1(w)_{1,k} 
- 
w_{1,k}\cos(\langle k,x\rangle)
\right\Vert_{L^\infty}
\\
&\qquad
|\cK_N| |\T^d|^{1/2} 
\max_{k\in \cK_N} \left\Vert \cN_1(w)_{2,k} 
- 
w_{2,k} \sin(\langle k,x\rangle)
\right\Vert_{L^\infty}
\\
&< \epsilon.
\end{align*}
\end{proof}

\subsection{Proof of the Universal Approximation Theorem \ref{thm:universal}}
\label{app:pfut}
As mentioned in the sketch of the proof of Theorem \ref{thm:universal}, we only need to consider the special case $s'=0$, due to Lemma \ref{lem:fut0}; i.e. given a continuous operator $\cG: H^s(\T^d) \to L^2(\T^d)$, $K\subset H^s(\T^d)$ compact, and $\epsilon>0$, we wish to construct a FNO $\cN: H^s(\T^d) \to L^2(\T^d)$, such that $\sup_{a\in K}\Vert \cG(a) - \cN(a) \Vert_{L^2} \le \epsilon$.

Throughout this proof, we set $d_a=d_u=1$ for notational convenience. The general case with $d_a, d_u > 1$ follows analogously. For $N\in \N$, let $P_N$ denote the orthogonal Fourier projection \eqref{eq:proj}. First, we note that since $K\subset H^{s}(\T^d)$ is compact, one can prove by elementary arguments that the set $\tilde{K}$ given by
\[ 
\tilde{K} 
:= 
K \cup {\bigcup_{N\in \N} P_N K},
\]
is compact. Since $\G$ is continuous, its restriction to $\tilde{K}$ is uniformly continuous, i.e. there exists a modulus of continuity $\omega: [0,\infty) \to [0,\infty)$, such that 
\[
\Vert \G(a) - \G(a') \Vert_{L^2} 
\le 
\omega
\left(
\Vert a - a' \Vert_{H^{s}}
\right), 
\quad \forall \, a, a' \in \tilde{K}.
\]

From the definition of the projection $\G_N$ \eqref{eq:defGN}, we have,
\begin{align}
\Vert \G(a) - \G_N(a) \Vert_{L^2}
&\le 
\Vert \G(a) - P_N\G(a) \Vert_{L^2}
+
\Vert P_N \G(a) - P_N\G(P_Na) \Vert_{L^2}
\notag\\
&\le 
\Vert \G(a) - P_N\G(a) \Vert_{L^2}
+
\Vert \G(a) - \G(P_Na) \Vert_{L^2}
\notag\\
&\le 
\sup_{v\in \G(\tilde{K})} \Vert (1-P_N) v \Vert_{L^2}
+
\omega\left(
\sup_{a\in \tilde{K}}
\Vert (1-P_N)a \Vert_{H^{\sigma}}
\right).
\label{eq:Gw}
\end{align}
Since $\tilde{K}$ is compact, also the image $\G(\tilde{K})$ is compact leading to,
\begin{align*}
\limsup_{N\to \infty} 
\sup_{u\in \G(\tilde{K})} \Vert (1-P_N) v \Vert_{L^2}
= 0 =
\limsup_{N\to \infty} 
\sup_{a\in \tilde{K}}
\Vert (1-P_N)a \Vert_{H^{s}}.
\end{align*}
In particular, there exists $N\in \N$, such that 
\begin{align} \label{eq:aGN}
\Vert \G(a) - \G_N(a) \Vert_{L^2} \le \epsilon, \quad
\forall \, a \in K \subset \tilde{K}.
\end{align}
In the remainder of this proof, we will construct an FNO approximation of $\G_N$. In fact, we note that $\G_N$ defines a continuous operator $\G_N:L^2(\T^d) \to L^2(\T^d)$, via $a \mapsto P_N\G(P_N a)$, and the compact set $K$ remains compact also as a subset of $L^2(\T^d)$. We will show that there exists a FNO $\cN: L^2(\T^d) \to L^2(\T^d)$, such that 
\[
\sup_{a\in K} \Vert \G_N(a) - \cN(a) \Vert_{L^2} < \epsilon.
\]
Then the restriction of $\cN$ to $H^s(\T^d) \subset L^2(\T^d)$ provides an approximation of $\G$, such that 
\[
\sup_{a\in K} \Vert \G(a) - \cN(a) \Vert_{L^2} < 2\epsilon.
\]
Since $\epsilon > 0$ was arbitrary, the claim follows from this.

As outlined in the sketch, our proof for the existence of a FNO approximating $\G_N$ relies on the decomposition \eqref{eq:decompG}, which in turn is defined via the Fourier conjugate operator $\hat{G}_N$ \eqref{eq:Fconj}. Our aim is to show that each operator in the decomposition \eqref{eq:decompG} can be approximated by FNOs to desired accuracy. To this end, let $\epsilon > 0$ be given, and choose $R_K, R_{\hat{K}}, R_{\hat{\G}} > 0$, such that 
\begin{align} \label{eq:subsets}
\left\{
\begin{gathered}
K \subset B_{R_K}(0) := \{\Vert u \Vert_{L^2}\le R_K\} \subset L^2(\T^d), 
\\
\cF_N \circ P_N \left( B_{R_K}(0) \right)
\subset 
\left[
-\frac{R_{\hat{K}}}{2}, \frac{R_{\hat{K}}}{2}
\right]^{2\cK_N}, 
\\
\hat{\G}_N\left(
[-R_{\hat{K}}, R_{\hat{K}}]^{2\cK_N}
\right)
\subset
\left[
-\frac{R_{\hat{\G}}}{2}, \frac{R_{\hat{\G}}}{2}
\right]^{2\cK_N}.
\end{gathered}
\right.
\end{align}
The reason for introducing $R_K, R_{\hat{K}}, R_{\hat{\G}}$ lies in the fact that in order to approximate $\G_N$ by a composition of FNOs, we need to ensure that each FNO in this composition maps its own domain into the domain of the next FNO. The FNO approximations of the individual steps in the composition $\G_N = \cF_N^{-1} \circ \hat{\G}_N \circ (\cF_N \circ P_N)$ are constructed below: \\
\textbf{FNO approximation of $\cF_N^{-1}$:} We start with our construction of a FNO approximation of the last step in the composition. To this end, we are going to interpret the mapping 
\[
\cF_N^{-1}: [-R,R]^{2\cK_N} \subset \R^{2\cK_N} \to L^2(\T^d),
\]
as a mapping 
\begin{align} \label{eq:ift0}
\cF_N^{-1}:
\left\{
\begin{aligned}
&L^2(\T^d; [-R,R]^{2\cK_N}) \to L^2(\T^d),
\\
&\{\Re(\hat{v}_k), \Im(\hat{v}_k)\}_{|k|\le N} \mapsto v(x),
\end{aligned}
\right.
\end{align}
where the input $\{\Re(\hat{v}_k), \Im(\hat{v}_k)\}_{|k|\le N} \in [-R,R]^{2\cK_N}$ is identified with a \emph{constant} function in $L^2(\T^d;[-R,R]^{2 \cK_N})$ (for non-constant input function $v(x)$, we apply $\cF_N^{-1}$ to the constant function $x \mapsto \fint_{\T^d} v(\xi) \, d\xi$ to define the mapping \eqref{eq:ift0} for general inputs). By Lemma \ref{lem:fno-ift}, the mapping \eqref{eq:ift0} can be approximated to any desired accuracy by an FNO $\cN_{\mathrm{IFT}}: L^2(\T^d; \R^{2\cK_N}) \to L^2(\T^d)$, such that 
\begin{align} \label{eq:Gift}
\Vert \cN_{\mathrm{IFT}}(\hat{v}) - \cF_N^{-1}(\hat{v}) \Vert_{L^2}
\le
\epsilon/3,
\end{align}
for all \emph{constant} input functions $\hat{v} \in L^2(\T^d;[-R,R]^{2\cK_N})$. \\
\textbf{FNO approximation of $\hat{\G}_N$:}
We view the Fourier conjugate operator $\hat{\G}_N$ \eqref{eq:Fconj} as a continuous mapping 
\[
\hat{\G}_N: [-R_{\hat{K}}, R_{\hat{K}}]^{2\cK_N} \subset \R^{2 \cK_N} \to \R^{2\cK_N}.
\]
Since $[-R_{\hat{K}}, R_{\hat{K}}]^{2\cK_N}$ is compact, there exists a finite-dimensional canonical neural network $\hat{\cN}: \R^{2\cK_N} \to \R^{2 \cK_N}$, such that 
\begin{align} \label{eq:GNhat}
\sup_{\hat{v}\in [-R_{\hat{K}}, R_{\hat{K}}]^{2\cK_N}} 
\Vert 
\hat{\G}_N(\hat{v}) - \hat{\cN}(\hat{v}) 
\Vert_{\ell^2}
\le \epsilon/3.
\end{align}
Furthermore, by \eqref{eq:subsets}, we have
\[
\hat{G}_N\left(
[-R_{\hat{K}}, R_{\hat{K}}]^{2\cK_N}
\right)
\subset
\left[
-\frac{R_{\hat{\G}}}{2}, \frac{R_{\hat{\G}}}{2}
\right]^{2\cK_N}.
\]
Therefore, by choosing a neural network approximation $\hat{\cN}$ with sufficiently high accuracy, we can also ensure that
\[
\hat{\cN}\left(
[-R_{\hat{K}}, R_{\hat{K}}]^{2\cK_N}
\right)
\subset
\left[
-{R_{\hat{\G}}}, {R_{\hat{\G}}}
\right]^{2\cK_N},
\]
in addition to \eqref{eq:GNhat}. Finally, we note that for $v\in L^2(\T^d;\R^{2\cK_N})$, the corresponding mapping
\[
\hat{\cN}: 
L^2(\T^d;\R^{2\cK_N}) \to L^2(\T^d;\R^{2\cK_N}), 
\quad
v(x) \mapsto \hat{\cN}(v(x)), 
\]
is in fact an FNO, with only \emph{local} layers of the form
\[
v_\ell(x) \mapsto \sigma\left(A_\ell v_\ell(x) + b_\ell \right), 
\quad
(A_\ell \in \R^{d_v\times d_v}, \; b_\ell \in \R^{d_v}),
\]
and where $d_v := 2|\cK_N|$, i.e. a FNO for which all $P_\ell \equiv 0$ (cp. Remark \ref{rem:loc}). We shall thus identify $\hN$ with this particular FNO, in the following. \\
\textbf{FNO approximation of $\cF_N\circ P_N$:}
Finally, we can similarly interpret
\[
\cF_N\circ P_N: B_{R_K}(0)\subset L^2(\T^d) \to \R^{2\cK_N},
\]
as a mapping 
\begin{align} \label{eq:ft0}
\cF_N\circ P_N:
\left\{
\begin{aligned}
&B_{R_K}(0) \to L^2(\T^d;\R^{2\cK_N}),
\\
&v \mapsto \{\Re(\hat{v}_k), \Im(\hat{v}_k)\}_{|k|\le N},
\end{aligned}
\right.
\end{align}
where the output $\{\Re(\hat{v}_k), \Im(\hat{v}_k)\}_{|k|\le N} \in \R^{2\cK_N}$ is a \emph{constant} function in  $L^2(\T^d;\R^{2 \cK_N})$.
By Lemma \ref{lem:fno-ft}, the mapping \eqref{eq:ft0} can be approximated to any desired accuracy by an FNO $\cN_{\mathrm{FT}}: B_{R_K}(0) \to L^2(\T^d;\R^{2\cK_N})$ (with constant output functions). In particular, denoting $\Lip(\hat{\cN})$ the Lipschitz constant of the FNO constructed in the previous step, we can ensure that 
\begin{align} \label{eq:Gft}
\Lip(\hat{\cN})\, 
\Vert \cF_N P_N v - \cN_{\mathrm{FT}}(v) \Vert_{\ell^2}
\le \epsilon/3, 
\quad \forall \, v\in B_{R_K}(0),
\end{align}
and furthermore, since by \eqref{eq:subsets}, we have
\[
\cF_N\circ P_N(B_{R_K}(0))
\subset 
\left[
-\frac{R_{\hat{K}}}{2},\frac{R_{\hat{K}}}{2}
\right]^{2\cK_N},
\]
we can in addition ensure that 
\[
{\cN_{\mathrm{FT}}}
\left(
B_{R_K}(0)
\right)
\subset 
\left[
-{R_{\hat{K}}}, {R_{\hat{K}}}
\right]^{2\cK_N}.
\]
\textbf{Error estimate for resulting FNO approximation:} We now define a FNO $\cN(a) := \cN_{\mathrm{IFT}} \circ \hat{\cN} \circ \cN_{\mathrm{FT}}(a)$, where the right-hand side terms have been constructed above. We note that 
\begin{align*}
\sup_{K} &\Vert \G_N - \cN \Vert_{L^2}
\\
&\le
\sup_{B_{R_K}(0)} 
\left\Vert 
\cF_N^{-1} \circ \hat{\G}_N \circ \cF_N \circ P_N
- 
\cN_{\mathrm{IFT}} \circ \hat{\cN} \circ \cN_{\mathrm{FT}} 
\right\Vert_{L^2}
\\
&\le 
\sup_{B_{R_K}(0)} 
\left\Vert 
\cF_N^{-1} \circ \hat{\G}_N \circ \cF_N \circ P_N
- 
\cF_N^{-1}  \circ \hat{\cN} \circ \cN_{\mathrm{FT}}
\right\Vert_{L^2}
\\
&\quad +
\sup_{ B_{R_K}(0)}
\left\Vert 
\cF_N^{-1}  \circ \hat{\cN} \circ \cN_{\mathrm{FT}}
- 
\cN_{\mathrm{IFT}} \circ \hat{\cN} \circ \cN_{\mathrm{FT}}
\right\Vert_{L^2}
\\
&\le 
\sup_{B_{R_K}(0)} 
\left\Vert 
\hat{\G}_N \circ \cF_N \circ P_N
- 
\hat{\cN} \circ \cN_{\mathrm{FT}}
\right\Vert_{L^2}
\\
&\quad +
\sup_{\hat{\cN}\left(\cN_{\mathrm{FT}}\left(B_{R_K}(0)\right)\right)}
\left\Vert 
\cF_N^{-1} 
- 
\cN_{\mathrm{IFT}}
\right\Vert_{L^2}
\\
&=: (I) + (II).
\end{align*}
For the second term $(II)$, we note that 
\[
\hat{\cN}\left(\cN_{\mathrm{FT}}\left(B_{R_K}(0)\right)\right)
\subset 
\hat{\cN}\left(
[-R_{\hat{K}}, R_{\hat{K}}]^{2\cK_N}
\right)
\subset
\left[
-{R_{\hat{\G}}}, {R_{\hat{\G}}}
\right]^{2\cK_N},
\]
and hence, by \eqref{eq:Gift}, we can bound
\begin{align*}
(II)
&\le
\sup_{\left[
-{R_{\hat{\G}}}, {R_{\hat{\G}}}
\right]^{2\cK_N}}
\left\Vert 
\cF_N^{-1} 
- 
\cN_{\mathrm{IFT}}
\right\Vert_{L^2}
\le
\epsilon/3.
\end{align*}
To estimate the first term $(I)$, we note that 
\begin{align*}
(I) &=\sup_{B_{R_K}(0)} 
\left\Vert 
\hat{\G}_N \circ \cF_N \circ P_N
- 
\hat{\cN} \circ \cN_{\mathrm{FT}}
\right\Vert_{L^2}
\\
&\le 
\sup_{B_{R_K}(0)} 
\left\Vert 
\hat{\G}_N \circ \cF_N \circ P_N
- 
\hat{\cN} \circ \cF_N \circ P_N
\right\Vert_{L^2}
\\
&\quad 
+ 
\sup_{B_{R_K}(0)} 
\left\Vert 
\hat{\cN} \circ \cF_N \circ P_N
- 
\hat{\cN} \circ \cN_{\mathrm{FT}}
\right\Vert_{L^2}
\\
&=: (Ia) + (Ib).
\end{align*}
To estimate $(Ia)$, we note that 
\[
\cF_N \left( P_N\left( B_{R_K}(0) \right) \right)
\subset 
\left[
-R_{\hat{K}}, R_{\hat{K}}
\right]^{2\cK_N}, 
\]
and hence 
\[
(Ia) \le
\sup_{\left[
-R_{\hat{K}}, R_{\hat{K}}
\right]^{2\cK_N}
}
\left\Vert 
\hat{\G}_N
- 
\hat{\cN}
\right\Vert_{L^2}
\le \epsilon/3,
\]
by \eqref{eq:GNhat}. Finally, to estimate $(Ib)$, we note that 
\[
(Ib)
\le
\Lip(\hat{\cN}) 
\;
\sup_{B_{R_K}(0)} 
\left\Vert 
\cF_N \circ P_N
- 
\cN_{\mathrm{FT}}
\right\Vert_{L^2}
\le 
\epsilon/3,
\]
by \eqref{eq:Gft}. Combining these estimates, we conclude that 
\[
\sup_{a\in K} \Vert \G_N(a) - \cN(a) \Vert_{L^2} \le \epsilon,
\]
This shows that the continuous operator $\G_N$ can be approximated by a FNO $\cN$ to any desired accuracy $\epsilon> 0$, and together with \eqref{eq:GN} concludes our proof of the universal approximation theorem \ref{thm:universal} for the special case $s'=0$. The general case with $s'\ge 0$ now follows from Lemma \ref{lem:fut0}.

\subsection{Proof of Theorem \ref{thm:fno-pfno}}
\label{app:fno-pfno}
\begin{proof}
The proof involves 4 steps. 

\step{1}{ We may wlog assume that all biases $b_\ell(x) \in C^\infty$. }

\noindent 
Indeed, it is easy to see that under the finite width assumption and for Lipschitz continuous $\sigma$, each layer $\cL_\ell$ defines a Lipschitz continuous mapping $L^2 \to L^2$. Replacing $b_\ell(x)$ by it's $\delta$-mollification $b^\delta_\ell(x)$ we obtain a new layer $\cL^\delta_\ell$, i.e.
\[
\cL^\delta_\ell(v) 
=
\sigma\left(
W_\ell v(x) + b^\delta_\ell(x) + \cF^{-1}\left( P_\ell \cF v\right)(x)
\right).
\]
Then we have 
\begin{align} \label{eq:moll}
\Vert \cL_\ell(v) - \cL^\delta_\ell(v) \Vert_{L^2}
\lesssim_{\Lip(\sigma),\ell}
\Vert b_\ell - b^\delta_\ell \Vert_{L^2}
\to 0,
\end{align}
and uniformly in $\delta > 0$, we have
\begin{align} \label{eq:Lip}
\Vert
\cL_\ell^\delta(v) - \cL_\ell^\delta(v')
\Vert_{L^2}
\lesssim_{\Lip(\sigma),\ell} 
\Vert v - v' \Vert_{L^2}.
\end{align}
In particular, this implies that $\Lip(\cL_\ell^\delta) \le C(\Lip(\sigma),\ell)$ is uniformly bounded in $\delta \ge 0$. Thus, there exists $\bar{M} \ge 1$, such that $\Lip(\cQ) \le \bar{M}$ and $\Lip(\cL_\ell^\delta)\le \bar{M}$, for all $\ell=1,\dots, L$ and $\delta > 0$. Properties \eqref{eq:moll} and \eqref{eq:Lip} are sufficient to show that with a sufficiently small choice of $\delta > 0$, we have
\begin{align} \label{eq:Ndeps}
\sup_{\Vert a \Vert_{H^s} \le B} \Vert \cN(a) - \cN^\delta(a) \Vert_{L^2} \le \epsilon,
\end{align}
where $\cN^\delta := \cQ\circ \cL^\delta_L \circ \dots \circ \cL^\delta_1 \circ\cR$: To see this, we introduce $\cN^\delta_{\ell} := \cQ \circ \cL^\delta_L \circ\dots \cL^\delta_{\ell+1} \circ \cL_\ell \circ \dots \circ \cL_1 \circ \cR$, and we note that for any $a\in L^2$, we have
\begin{align*}
\Vert
\cN^\delta_\ell(a) - \cN^\delta_{\ell-1}(a)
\Vert_{L^2}
&\le
\Lip(\cQ \circ \cL^\delta_{L}\circ \dots \circ \cL^\delta_{\ell+1})
\Vert 
\cL_\ell(a')
 - 
\cL^\delta_{\ell}(a') 
\Vert_{L^2},
\end{align*}
where $a' := \cL_{\ell-1} \circ \dots \circ \cL_1 \circ \cR(a)$. Using the uniform Lipschitz bound $\bar{M}$ and the estimate \eqref{eq:moll}, we obtain
\[
\Vert
\cN^\delta_\ell(a) - \cN^\delta_{\ell-1}(a)
\Vert_{L^2}
\le 
\bar{M}^L \Vert b_\ell - b_\ell^\delta \Vert_{L^2}.
\]
It now follows from the telescoping sum $\cN(a) - \cN^\delta(a) = \cN^\delta_L(a) - \cN^\delta_0(a) = \sum_{\ell=1}^L \left(\cN^\delta_{\ell}(a) - \cN^\delta_{\ell-1}(a)\right)$ that
\begin{align*}
\Vert 
\cN(a) - \cN^\delta(a)
\Vert_{L^2}
&\le
\sum_{\ell = 1}^{L}
\Vert \cN^\delta_{\ell}(a) - \cN^\delta_{\ell-1}(a)\Vert_{L^2}
\\
&\le
\bar{M}^L 
\sum_{\ell=1}^L \Vert b_\ell - b_\ell^\delta \Vert_{L^2}.
\end{align*}
Since the bound on the right-hand side is independent of $a$, and since $\lim_{\delta \to 0} \Vert b_\ell - b_\ell^\delta\Vert_{L^2} = 0$, choosing $\delta>0$ sufficiently small, we obtain \eqref{eq:Ndeps}.

\step{2}{
Using Step 1, we will assume wlog that all biases are smooth, in the following. Given $a\in H^s$, let $a_\ell := \cL_\ell \circ \dots \cL_1 \circ \cR(a)$. We claim that there exists a constant $B'>0$, depending only on $B$, the activation function $\sigma$ and the biases $b_\ell(x)$, such that $\Vert a_\ell \Vert_{H^s} \le B'$, provided that $\Vert a \Vert_{H^s}\le B$ and $\ell=1,\dots, L$.
}

Indeed, it is a well-known fact\footnote{This is well-known in the integer case $s\in \N$, but remains true also for fractional Sobolev spaces \cite{BM2001}.} that if $v\in H^s$, $s>d/2$, and if $\sigma \in C^{\lfloor s \rfloor + 1}$ (satisfied by assumption), then $\sigma \circ v \in H^s$. In fact, there exists a monotonically increasing function $C_\sigma$, such that $\Vert \sigma \circ v \Vert_{H^s} \le C_\sigma(\Vert v \Vert_{H^s})$. The claimed existence of $B'$ now follows immediately from this result on compositions with $\sigma$, and the observation that the linear part of each layer define a continuous (bounded, affine) mapping $H^s\to H^s$, due to the finite width assumption, and the assumption that $b_\ell \in C^\infty \subset H^s$.

\step{3}{
Let $\varsigma_1,\varsigma_2$ be given such that $s \ge \varsigma_1 > \varsigma_2 > d/2$. Each layer defines a H\"older continuous mapping $\cL_\ell: D_\ell \subset H^{\varsigma_1}\to H^{\varsigma_2}$, with H\"older exponent $\alpha = 1 - \varsigma_2/s$, and where $D_\ell$ denotes the image of the set $\{\Vert a \Vert_{H^s}\le B\}$ under the previous $\ell-1$ layers.
}

By Step 2, we have that $\Vert \cL_\ell(v) \Vert_{H^s} \le B'$ for all $v\in D_\ell$. Furthermore, as observed in Step 1, \eqref{eq:Lip}, $\cL_\ell: L^2 \to L^2$ is Lipschitz continuous. It follows that for any $v\in D_\ell$, we have, by the interpolation inequality, that
\begin{align*}
\Vert \cL_\ell(v) - \cL_\ell(v') \Vert_{H^{\varsigma_2}}
&\lesssim
\Vert \cL_\ell(v) - \cL_\ell(v') \Vert_{L^2}^{1-\varsigma_2/s}
\Vert \cL_\ell(v) - \cL_\ell(v') \Vert_{H^s}^{\varsigma_2/s}
\\
&\lesssim_{\varsigma_2,B'}
\Vert \cL_\ell(v) - \cL_\ell(v') \Vert_{L^2}^{\alpha}
\\
&\lesssim_{\varsigma_2,B',\sigma,\ell} 
\Vert v - v' \Vert_{L^2}^\alpha
\\
&\lesssim
\Vert v - v' \Vert_{H^{\varsigma_1}}^\alpha.
\end{align*}

\step{4}{
We use Steps 1-3 to conclude that for a sufficiently large choice of $N\in \N$, the $\Psi$-FNO $\hN: L^2_N \to L^2_N$, given by 
\begin{align*}
\hN(a) 
&:= \cQ \circ \cI_N \circ \cL_L \circ \cI_N \circ \cL_{L-1} \circ \cI_N \circ \dots \circ \cL_1 \circ \cI_N \circ \cR,
\end{align*}
satisfies
\[
\sup_{\Vert a\Vert_{H^s} \le B}
\Vert \cN(a) - \hN(a) \Vert_{L^2}
\le 
\epsilon.
\]
}

In the following, we will denote by $\hL_{\ell} := \cI_N \circ \cL_{\ell} \circ \cI_N$ the ``pseudo-spectral projection'' of the layer $\cL_\ell$, and we observe that
\[
\hN = \cQ \circ \hL_L \circ \dots \circ \hL_1 \circ \cR.
\]

Fix a sequence $s = \varsigma_0 > \varsigma_1 > \varsigma_2> \dots > \varsigma_L > d/2$. By Step 3, we may view the $\ell$-th layer $\cL_\ell$ as a H\"older continuous mapping $\cL_\ell: H^{\varsigma_{\ell-1}} \to H^{\varsigma_\ell}$. Furthermore, since $\varsigma_{\ell} > d/2$ for all $\ell$, we also have that $\cI_N: H^{\varsigma_{\ell}} \to H^{\varsigma_{\ell}}$ is a bounded linear operator, with operator norm that is uniformly bounded in $N$. In particular, this implies that there exists a constant $\bar{C} \ge 1$, independent of $N$, such that, for all $a,a'\in D_\ell$:
\begin{align*}
\Vert \hL_{\ell}(a) - \hL_{\ell}(a') \Vert_{H^{\varsigma_{\ell}}}
&=
\Vert \cI_N \cL_{\ell}(\cI_N a) - \cI_N \cL_{\ell}(\cI_N a') \Vert_{H^{\varsigma_{\ell}}}
\\
&\le
\Vert \cI_N \Vert \Vert \cL_{\ell} \Vert_{C^{1-\varsigma_\ell/s}}\Vert \cI_N \Vert^{1-\varsigma_\ell/s} \Vert a - a' \Vert_{H^{\varsigma_{\ell-1}}}^{1-\varsigma_\ell/s}
\\
&\le
\bar{C} \Vert a - a' \Vert_{H^{\varsigma_{\ell-1}}}^{1-\varsigma_\ell/s}.
\end{align*}
We now introduce $\hN_{\ell} := \cQ \circ \hL_L \circ \dots \circ \hL_{\ell+1} \circ \cL_{\ell} \circ \dots \cL_1 \circ \cR$, for any $\ell=0,\dots, L$. Then, $\cN(a) - \hN(a) = \sum_{\ell=1}^L \left(\hN_\ell(a) - \hN_{\ell-1}(a) \right)$, and using the H\"older regularity of the $\hL_{\ell}$, we obtain
\begin{align*}
\Vert \hN_{\ell}(a) - \hN_{\ell-1}(a) \Vert_{L^2}
&\le \Vert \hN_{\ell}(a) - \hN_{\ell-1}(a) \Vert_{H^{\varsigma_{L}}}
\\
&\le
\Lip(\cQ)
\bar{C}^L \Vert \cL_{\ell}(a') - \hL_{\ell}(a') \Vert_{H^{\varsigma_{\ell}}}^{\beta_\ell},
\end{align*}
where $a' := \cL_{\ell-1} \circ \dots \cL_{1} \circ \cR(a)$ belongs to $D_\ell \subset \set{a'\in H^s}{\Vert a' \Vert_{H^s} \le B'}$ (cp. Step 3), and $\beta_\ell := \prod_{k=\ell+1}^L \left(1 - \frac{\varsigma_{k}}{s} \right)$ is the H\"older exponent of the composition of the layers $\hL_{L} \circ \dots \circ \hL_{\ell+1}: H^{\varsigma_{\ell+1}} \to H^{\varsigma_{L}}$. To estimate the last difference, we recall that if $\Vert a \Vert_{H^s}\le B$, then $\Vert a' \Vert_{H^{s}} \le B'$, by Step 2. In particular, it follows from the uniform H\"older continuity of the layers $\cL_{\ell}$, established in Step 3 and the pseudo-spectral approximation estimate, that 
\begin{align*}
\Vert \cL_{\ell}(a') - \hL_{\ell}(a') \Vert_{H^{\varsigma_{\ell}}}
&=
\Vert \cL_{\ell}(a') - \cI_N \cL_{\ell}(\cI_N a') \Vert_{H^{\varsigma_{\ell}}}
\\
&\le
\Vert \cL_{\ell}(a') - \cI_N \cL_{\ell}(a') \Vert_{H^{\varsigma_{\ell}}}
+
\Vert \cI_N \Vert \Vert \cL_{\ell}(a') - \cL_{\ell}(\cI_N a') \Vert_{H^{\varsigma_{\ell}}}
\\
&\lesssim
\frac{B'}{N^{-(s-\varsigma_\ell)}}
+
\Vert a' - \cI_N a' \Vert_{H^{\varsigma_{\ell-1}}}^{1-\varsigma_{\ell}/s}
\\
&\lesssim
\frac{B'}{N^{-(s-\varsigma_\ell)}}
+
\left(
\frac{
B'
}{
N^{-(s-\varsigma_{\ell-1})}
}
\right)^{(1-\varsigma_\ell/s)}
,
\end{align*}
where the implied constant is independent of $N$. In particular, we have 
\[
\lim_{N\to \infty} \sup_{a'\in D_\ell} \Vert \cL_{\ell}(a') - \hL_{\ell}(a') \Vert_{H^{\varsigma_{\ell}}} 
=
\lim_{N\to \infty} \sup_{a'\in D_\ell} \Vert \cL_{\ell}(a') - \cL_{\ell}(\cI_N a') \Vert_{H^{\varsigma_{\ell}}} 
= 0,
\]
for all $\ell=1,\dots, L$.  Choosing $N$ sufficiently large, we can thus ensure that 
\begin{align*}
\sup_{\Vert a \Vert_{H^s}\le B}
\Vert \cN(a) - \hN(a) \Vert_{L^2}
&\le
\sum_{\ell=1}^L 
\sup_{\Vert a \Vert_{H^s}\le B}
\Vert \hN_\ell(a) - \hN_{\ell-1}(a) \Vert_{L^2}
\\
&\le
\Lip(\cQ) \bar{C}^L
\sum_{\ell=1}^L \sup_{a'\in D_\ell} \Vert \cL_\ell(a') - \hL_{\ell}(a') \Vert_{H^{\varsigma_{\ell}}}^{\beta_\ell}
\\
&\le
\epsilon.
\end{align*}
This concludes the proof.
\end{proof}

\subsection{
Proof of Lemma \ref{lem:pfnout0}
}
\label{app:pfnout0}

\begin{proof}
Let $\epsilon > 0$ be given. Let $\cG: H^s \to H^{s'}$ be continuous operator, and $K\subset H^s$ a compact subset. We wish to show that there exists $N_0\in \N$, such that for any $N\ge N_0$, there exists a $\Psi$-FNO $\hN: L^2_N \to L^2_N$ such that 
\begin{align}\label{eq:target}
\sup_{a\in K} \Vert \cG(a) - \hN(a) \Vert_{H^{s'}} \le \epsilon.
\end{align}
Choose $M\in \N$, such that 
\begin{align} \label{eq:p-G}
\sup_{a\in K} \Vert P_{M}\G(a) - \G(a) \Vert_{H^{s'}} \le \epsilon/2.
\end{align} 
Fix $\delta > 0$ for the moment. We will specify a suitable choice of $\delta = \delta(M,s',\T^d,\epsilon)>0$ at the end of this proof. We emphasize that $\delta$ will depend only on parameters already introduced at this point of the proof. By assumption of the validity of the universal approximation theorem for $\Psi$-FNOs for $s'=0$, we can choose $\tilde{N}_0 = \tilde{N}_0(\cG, K, \delta) \in \N$, depending on $\cG$, $K$ and $\delta$, such that for any $N\ge \tilde{N}_0$, there exists a $\Psi$-FNO $\cN: L^2_N \to L^2_N$, such that 
\[
\sup_{a\in K} \Vert \cG(a) - \cN(a) \Vert_{L^2} \le \delta.
\]
We define $N_0 := \max(\tilde{N}_0, M)$. By our choice of $\tilde{N}_0$ and $M$, the constant $N_0$ depends only on the underlying operator $\cG$, the compact set $K$ and the parameter $\delta > 0$. Let $N\ge N_0$. We claim that if $\delta = \delta(M,s',\T^d,\epsilon)>0$ is been chosen sufficiently small, then for any $N\ge N_0$, there exists a $\Psi$-FNO $\hN: L^2_N \to L^2_N$ satisfying \eqref{eq:target}.

To see this, let $N\ge N_0$, and let $\cN: L^2_N \to L^2_N$ be a $\Psi$-FNO such that 
$\sup_{a\in K} \Vert \G(a) - \cN(a) \Vert_{L^2} \le \delta$. Then we have
\begin{align} \label{eq:p-GN}
\sup_{a\in K} \Vert P_M\G(a) - P_M\cN(a) \Vert_{H^{s'}}
\le CN^{s'} \delta,
\end{align}
where $C = C(\T^d,s')>0$ is independent of $N$ and $\delta$. Let $m\in \N$ be the smallest natural number strictly bigger than $s'$ and $d/2$, i.e. $m>\max(s', d/2)$. Let 
\begin{align} \label{eq:Kpdef}
K' := \cN(K) \subset L^2_N,
\end{align}
 be the compact image of $K\subset H^s$ under the (continuous) mapping $\cN: H^s \to L^2_N$. We note that since $L^2_N$ is a finite-dimensional space, the $L^2$-norm is equivalent to the $H^m$-norm on $L^2_N$. This implies that $K'\subset H^m$ is also compact when considered as a subset of $H^m$, with respect to norm $\Vert \slot \Vert_{H^m}$. By Lemma \ref{lem:proj-layer}, there exists a single-layer FNO $\cL: H^m \to H^m$, such that 
\begin{align} \label{eq:Lfd}
\sup_{v\in K'} \Vert P_M v - \cL(v) \Vert_{H^{m}}
\le \delta.
\end{align}
We define a $\Psi$-FNO $\hN: L^2_N \to L^2_N$ by the following composition
\[
\hN := \cI_N \circ \cL \circ \cI_N \circ \cN
=
\cI_N \circ \cL \circ \cN,
\]
where $\cI_N$ is the pseudo-spectral projection operator. We note that, since $N\ge N_0\ge M$, we have $\cI_N P_M = P_M$, and recall that for $m>d/2$, the mapping $\cI_N: H^m \to H^m$ is continuous, with an operator norm that can be bounded independently of $N$.
Hence, we can estimate, for any $a\in K$,
\begin{align*}
\Vert P_M \cN(a) - \hN(a) \Vert_{H^{s'}}
&=
\Vert \cI_N \circ P_M \cN(a) - \cI_N \circ \cL \circ \cN(a) \Vert_{H^{s'}}
\\
&\explain{\le}{(m>s')}
\Vert \cI_N \circ P_M \cN(a) - \cI_N \circ \cL \circ \cN(a) \Vert_{H^{m}}
\\
&\le
C
\Vert P_M v - \cL(v) \Vert_{H^{m}},
\end{align*}
where $C = C(\T^d,m)>0$ is independent of $M$ and $N$, and where $v:= \cN(a) \in K'$, with $K' = \cN(K)$ defined in \eqref{eq:Kpdef}. From \eqref{eq:Lfd}, we conclude that 
\begin{align} \label{eq:p-NN}
\sup_{a\in K}
\Vert P_M \cN(a) - \hN(a) \Vert_{H^{s'}}
\le
C \delta,
\end{align}
where $C = C(\T^d,m)>0$ is independent of $M$, $N$ and $\delta$. In fact, since $m$ is defined as the smallest integer $>\max(s',d/2)$, the above constant only depends on $s'$ and the dimension $d$ of the domain, i.e. we have $C = C(\T^d,s')$. Combining \eqref{eq:p-G}, \eqref{eq:p-GN} and \eqref{eq:p-NN}, we find that for $N \ge N_0$ and for any $\delta >0$, there exists a $\Psi$-FNO $\hN: L^2_N \to L^2_N$, such that 
\[
\sup_{a \in K} \Vert \G(a) - \hN(a) \Vert_{H^{s'}}
\le
\epsilon/2 + 
C(1 + N^{s'}) \delta,
\]
where $C = C(\T^d,s') > 0$ is a constant independent of $\delta$. 
Thus, if we choose $\delta = \delta(M,s',\T^d,\epsilon)>0$ at the beginning of this proof sufficiently small to ensure that 
\[
C(1 + N^{s'}) \delta \le \epsilon/2,
\]
then we conclude that for any $N\ge N_0$, there exists a $\Psi$-FNO $\hN: L^2_N \to L^2_N$, such that 
\[
\sup_{a \in K} \Vert \G(a) - \hN(a) \Vert_{H^{s'}}
\le
\epsilon.
\]
This concludes the proof.
\end{proof}

\subsection{Technical Results on the structure of {$\Psi$}-FNOs}
\label{app:fnoprop}
Recall section \ref{sec:fnoprop} where we introduced $\sigma$ layers and $\cF$-layers and claimed that $\Psi$-FNOs can be decomposed in terms of these layers. We have the following series of lemmas, which make this observation precise.
\begin{lemma} \label{lem:sigma}
Let $\hN: \R^{d_a} \to \R^{d_u}$ be a (ordinary) neural network with activation function $\sigma$. For any $N\in \N$, the mapping
\[
\cN: L^2_N(\T^d;\R^{d_a}) \to L^2_N(\T^d;\R^{d_u}),
\quad
a(x) \mapsto \cN(a)(x) := \cI_N\hN(a(x)),
\]
can be represented by a $\Psi$-FNO, with 
\[
\width(\cN) = |\cJ_N| \width(\hN),
\quad
\depth(\cN) = \depth(\hN),
\quad
\lift(\cN) = \width(\hN).
\]
\end{lemma}

The proof is straight-forward and completely analogous to the same statement for FNOs in Remark \ref{rem:loc}.  We also note the following lemma, which allows us in practice to replace $\cF$-layers by proper FNO layers:
\begin{lemma}[Linear approximation lemma]
\label{lem:linear}
Assume that the activation function $\sigma \in C^2$ is twice continuously differentiable and non-constant. Let $\hN$ be a linear $\Psi$-FNO layer of the form
\[
\hL: L^2_N(\T^d;\R^{d_v}) \to L^2_N(\T^d;\R^{d_v}), \; \hL(v_N) := Wv_N(x_j) + b_j + \cF_N^{-1}(P\cF_Nv_N)_j.
\]
Then there exists a constant $C>0$, such that for any $\epsilon, \, B>0$, there exists a $\Psi$-FNO $\cN: L^2_N(\T^d;\R^{d_v}) \to L^2_N(\T^d;\R^{d_v})$, such that 
\[
\sup_{\Vert v_N \Vert_{L^2} \le B} 
\Vert \cN(v_N) - \hL(v_N) \Vert_{L^2} \le \epsilon,
\]
and
\[
\width(\cN) \le CN^d,
\quad
\depth(\cN) \le C, 
\quad
\lift(\cN) \le C.
\]
\end{lemma}

\begin{proof}
This follows from the observation that, by assumption, there exists $x_0\in \R$, such that $\sigma'(x_0) \ne 0$. By Taylor expansion, we have
\[
\frac{\sigma(x_0 + hy) - \sigma(x_0-hy)}{h \sigma'(x_0)}
=
y + R(h,y),
\]
where $|R(h,y)|\le C\bar{B}^2 h$ for any $y\in [-\bar{B},\bar{B}]$ and $0<h\le 1$. Replacing $y$ by the output of the linear layer $\hL$, and choosing $\bar{B}$ sufficiently large, we obtain with 
\[
\cN_h(v_n) := \frac{\sigma(x_0 + h \hL(v_N)) - \sigma(x_0-h \hL(v_N))}{h \sigma'(x_0)},
\]
that
\begin{align*}
\Vert 
\hL(v_N) - \cN_h(v_N)
\Vert_{L^2}
=
O(h),
\end{align*}
uniformly for all $v_N \in L^2_N(\T^d;\R^{d_v})$, such that $\Vert v_N \Vert_{L^2}\le B$, and for $0<h\le 1$. In particular, choosing $h = h(\epsilon)>0$ sufficiently small, we can ensure that
\[
\Vert 
\hL(v_N) - \cN_h(v_N)
\Vert_{L^2}
\le
\epsilon,
\]
for all $v_n \in L^2_N(\T^d;\R^{d_v})$, such that $\Vert v_N \Vert_{L^2}\le B$. The proof is concluded by observing that the mapping $v_N \mapsto \cN_h(v_N)$ defines a $\Psi$-FNO, and that the size (width, depth, lift) is uniformly bounded in $h$.
\end{proof}

According to Lemma \ref{lem:sigma}, any composition of $\sigma$-layers can be identified with a $\Psi$-FNO. Lemma \ref{lem:linear} shows that we can approximate $\cF$-layers to arbitrary accuracy with a $\Psi$-FNO of uniformly bounded size. In the following lemma, we record the simple fact that a composition of $\Psi$-FNOs is again representable by a $\Psi$-FNO:

\begin{lemma} [Composition Lemma] \label{lem:composition}
For $p\in \N$, let $\cN_1,\dots, \cN_p$ be $\Psi$-FNOs, defined with respect to the same grid $\{x_j\}_{j\in \cJ_N}$, and defining operators $\cN_k: L^2_N(\T^d; \R^{d_{k-1}}) \to L^2_N(\T^d;\R^{d_{k}})$ for $d_0,\dots, d_p \in \N$, such that the composition $\cN_p\circ \dots \cN_1: L^2_N(\T^d;\R^{d_0}) \to L^2_N(\T^d;\R^{d_p})$ is well-defined. Then there exists a $\Psi$-FNO $\cN: L^2_N(\T^d;\R^{d_0}) \to L^2_N(\T^d;\R^{d_p})$, such that $\cN = \cN_p\circ \dots \circ \cN_1$, and such that
$\depth(\cN) \le \sum_{k=1}^p \depth(\cN_k)$, $\width(\cN) \le \max_{k=1,\dots, p} \width(\cN_k)$ and  $\lift(\cN) \le \max_{k=1,\dots, p} \lift(\cN_k)$.
\end{lemma}

The proof of the previous lemma is straight-forward, and only requires a padding of each layer with zeros to achieve a uniform lifting dimension $d_v$ across all layers. The following lemma will be central for our approximation and complexity estimates:

\begin{lemma} [Replacement Lemma] \label{lem:replacement}
Assume that $\sigma$ is Lipschitz continuous. For $p\in \N$, let $\hN_1,\dots, \hN_p$ be continuous operators $\hN_k: L^2_N(\T^d; \R^{d_{k-1}}) \to L^2_N(\T^d;\R^{d_{k}})$ for $d_0,\dots, d_p \in \N$, such that the composition $\hN := \hN_p\circ \dots \circ\hN_1: L^2_N(\T^d;\R^{d_0}) \to L^2_N(\T^d;\R^{d_p})$ is well-defined. Assume that there exist constants $D_k, W_k, L_k>0$, $k=1,\dots, p$, such that for any $\epsilon, \, M > 0$, there exists a $\Psi$-FNO $\cN_k: L^2_N(\T^d;\R^{d_{k-1}}) \to L^2_N(\T^d;\R^{d_{k}})$, such that 
\[
\sup_{\Vert u \Vert_{L^2}\le M} \Vert \hN_k(u) - \cN_k(u) \Vert_{L^2_N} \le \epsilon,
\]
and $\depth(\cN_k) \le D_k$, $\width(\cN_k) \le W_k$, $\lift(\cN_k) \le L_k$. Then for any $\epsilon>0$ and $M>0$, there exists a $\Psi$-FNO $\cN: L^2_N(\T^d;\R^{d_0}) \to L^2_N(\T^d;\R^{d_p})$, such that 
\[
\sup_{\Vert u \Vert_{L^2}\le M} \Vert \hN(u) - \cN(u) \Vert_{L^2_N} \le \epsilon,
\]
and $\depth(\cN) \le \sum_{k=1}^p D_k$, $\width(\cN) \le \max_{k=1,\dots,p} W_k$, $\lift(\cN) \le \max_{k=1,\dots, p} L_k$.
\end{lemma}

\begin{proof}
We prove the statement by induction on $p\in \N$. We start the induction at $p=2$: Let $\epsilon > 0$ and $M>0$ be given, and let $\hN_1$, $\hN_2$ satisfy the hypotheses of this lemma. Since $B_M(0) \subset L^2_N(\T^d;\R^{d_0})$, the closed ball of radius $M$, is compact in the finite-dimensional space $L^2_N(\T^d;\R^{d_0})$, and since $\hN_1$ is continuous, it follows that the image $\hN_1(B_M(0))$ is also compact; in particular, there exists $M_1> 0$, such that $\hN_1(B_M(0)) \subset B_{M_1}(0)$. By assumption, there exists a $\Psi$-FNO $\cN_2: L^2_N(\T^d;\R^{d_1}) \to L^2_N(\T^d;\R^{d_2})$, such that 
\[
\sup_{\Vert u \Vert_{L^2} \le 2M_1} \Vert \hN_2(u) - \cN_2(u) \Vert_{L^2_N} \le \epsilon/2.
\]
We note that (for Lipschitz continuous $\sigma$) $\hN_2$ is Lipschitz continuous on $B_{2M_1}(0)\subset L^2_N$. Let $\Lip(\hN_2)$ denote the corresponding Lipschitz constant. By assumption on $\hN_1$, there exists a $\Psi$-FNO $\cN_1: L^2_N(\T^d;\R^{d_0}) \to L^2_N(\T^d;\R^{d_1})$, such that 
\[
\sup_{\Vert u \Vert_{L^2} \le M} \Vert \hN_1(u) - \cN_1(u) \Vert_{L^2_N} \le \min\left(\frac{\epsilon}{2\Lip(\hN_2)}, M_1\right).
\]
Note that this estimate implies in particular that 
\[
\Vert \cN_1(u) \Vert_{L^2} \le \Vert\hN_1(u) -  \cN_1(u) \Vert_{L^2} + \Vert \hN_1(u) \Vert_{L^2} \le 2M_1.
\]
Thus, we can estimate
\begin{align*}
\sup_{\Vert u \Vert_{L^2} \le M} \Vert \hN_2\circ \hN_1(u) &- \cN_2 \circ \cN_1(u) \Vert_{L^2}
\\
&\le
\sup_{\Vert u \Vert_{L^2} \le M} \Vert \hN_2\circ \hN_1(u) - \hN_2 \circ \cN_1(u) \Vert_{L^2}
\\
&\qquad +
\sup_{\Vert u \Vert_{L^2} \le M} \Vert \hN_2\circ \cN_1(u) - \cN_2 \circ \cN_1(u) \Vert_{L^2}
\\
&\le
\Lip(\hN_2) \,
\sup_{\Vert u \Vert_{L^2} \le M} \Vert \hN_1(u) - \cN_1(u) \Vert_{L^2}
\\
&\qquad +
\sup_{\Vert v\Vert\le 2M_1} \Vert \hN_2(v) - \cN_2(v) \Vert_{L^2}
\\
&< \epsilon.
\end{align*}
The estimate on the depth, width and lift is immediate (cp. Lemma \ref{lem:composition}). This proves the base case $p=2$ of the induction. 

Let now $p>2$, and assume that the claim holds for the composition of $p-1$ operators. Given $\hN_1, \dots, \hN_p$ as in the statement of the lemma, we can apply the induction hypothesis for the composition of $p-1$ operators, to see that the two continuous operators $\tN_1:=\hN_1$ and $\tN_2 := (\hN_p\circ \dots \circ \hN_2)$ fulfill all assumptions of the lemma (with $p=2$), and with depths, widths and lifts of the approximating FNOs given by $\tilde{D}_1 = D_1$, $\tilde{W}_1 = W_1$, $\tilde{L}_1 = L_1$, and $\tilde{D}_2 = \sum_{k=2}^p D_k$, $\tilde{W}_2 = \max_{k=2,\dots, p} W_k$, $\tilde{L}_2 = \max_{k=2,\dots, p} L_k$. The proof of the induction step now follows from the base case (with $p=2$) already considered above.
\end{proof}
\section{Technical Results and Proofs for Section \ref{sec:darcy}}
\label{app:s3a}
In this section, we collect some technical results and proofs for the material in section \ref{sec:darcy} of the main text. We start with the following Lemma on the contraction property of the map \eqref{eq:Fcontract}.
\begin{lemma} \label{lem:contract}
Let $k\in \N$, $k>d/2+1$, and let $s>d/2+k$ be given. Assume that the coefficient $a\in H^s(\T^d)$ satisfies the $\lambda$-coercivity condition \eqref{eq:lowerbd} for some $\lambda\in (0,1)$, and that the right-hand side of the stationary Darcy equation \eqref{eq:darcy} belongs to the Sobolev space $f\in H^{k-1}$. Then there exists $N_0 = N_0(s,d,\Vert a \Vert_{H^s}, \lambda) \in \N$, such that for any $N\ge N_0$, we have $a_N \ge \lambda/2$ and the mapping $F_N$ defined by \eqref{eq:Fcontract} is a contraction, with
\[
\Lip\left( F_N: \dot{H}^1 \to \dot{H}^1 \right) 
\le
1 - \frac{\lambda}{2}.
\]
\end{lemma}
\begin{proof}
As $u_N \mapsto F_N(u_N)$ is an affine mapping, we can express the Lipschitz constant in terms of the following supremum
\[
\Lip\left(F_N: \dot{H}^1 \to \dot{H}^1 \right) 
=
\sup_{u_N\in \dot{H}^1\setminus\{0\}} 
\frac{
\left\Vert 
\dot{P}_N (-\Delta)^{-1} \nabla \cdot (\tilde{a}_N \nabla u_N) 
\right\Vert_{\dot{H}^1}
}
{
\Vert u_N \Vert_{\dot{H}^1}
}.
\]
We now observe that
\begin{align*}
\left\Vert 
\dot{P}_N (-\Delta)^{-1} \nabla \cdot (\tilde{a}_N \nabla u_N) 
\right\Vert_{\dot{H}^1}
&\le
\left\Vert 
(-\Delta)^{-1} \nabla \cdot (\tilde{a}_N \nabla u_N) 
\right\Vert_{\dot{H}^1}
\\
&=
\left\Vert 
\tilde{a}_N \nabla u_N
\right\Vert_{L^2}
\\
&\le
\Vert \tilde{a}_N \Vert_{L^\infty}
\Vert u_N\Vert_{\dot{H}^1},
\end{align*}
for all $u_N \in \dot{H}^1_N(\T^d)$. To finish the proof, we note that 
\begin{align*}
\Vert \tilde{a}_N \Vert_{L^\infty}
&\le
\Vert \tilde{a} \Vert_{L^\infty}
+
\Vert \tilde{a} - \tilde{a}_N \Vert_{L^\infty},
\end{align*}
where the first term is bounded by $1-\lambda$ by the coercivity assumption. By our definition of $\tilde{a}_N$ (cp. equation \eqref{eq:defan}), we have
\begin{align*}
\Vert \tilde{a} - \tilde{a}_N \Vert_{L^\infty}
&\le
\left\Vert \tilde{a} - \dot{P}_N\tilde{a} \right\Vert_{L^\infty}
+
\left\Vert \dot{P}_N\tilde{a} - \dot{P}_N\cI_{2N}\tilde{a} \right\Vert_{L^\infty}.
\end{align*}
But by the spectral/pseudo-spectral approximation estimates, we now have 
\begin{align*}
\left\Vert \tilde{a} - \dot{P}_N\tilde{a} \right\Vert_{L^\infty}
&\lesssim_{d,\delta} \left\Vert \tilde{a} - \dot{P}_N\tilde{a} \right\Vert_{H^{d/2+\delta}}
\\
&\lesssim_{s,d,\delta} N^{-(s-d/2-\delta)} \Vert a \Vert_{H^s}, 
\end{align*}
and
\begin{align*}
\left\Vert \dot{P}_N\tilde{a} - \dot{P}_N\cI_{2N}\tilde{a} \right\Vert_{L^\infty}
&\lesssim_{d,\delta} \left\Vert \dot{P}_N\tilde{a} - \dot{P}_N\cI_{2N}\tilde{a} \right\Vert_{H^{d/2+\delta}}
\\
&\le \left\Vert \tilde{a} - \cI_{2N}\tilde{a} \right\Vert_{H^{d/2+\delta}}
\\
&\lesssim_{s,d,\delta} N^{-(s-d/2-\delta)} \Vert a \Vert_{H^s}.
\end{align*}
To be definite, let us choose $\delta := (s-d/2)/2$. In particular, it follows that there exists a constant $C = C(s,d)$, and $N_0 = N_0(s,d,\Vert a \Vert_{H^s},\lambda)\in \N$, such that 
\[
\Vert \tilde{a} - \tilde{a}_N \Vert_{L^\infty} \le \frac{\lambda}{2},
\]
for $N \ge N_0$, and hence $a_N = 1+\tilde{a}_N \ge 1 + \tilde{a} - \lambda/2 \ge \lambda/2$, and
\[
\Lip(F_N) 
\le \Vert \tilde{a}_N \Vert_{L^\infty} 
\le \Vert \tilde{a} \Vert_{L^\infty} + \Vert \tilde{a} - \tilde{a}_N \Vert_{L^\infty}  
\le 1-\frac{\lambda}{2},
\]
for all $N \ge N_0$. The claim follows.
\end{proof}

Next, we provide the detailed proof of Theorem \ref{thm:FG} which guarantees convergence of the algorithm \ref{alg:darcy} to the solutions of the Darcy equation \eqref{eq:darcy}.
\subsection{Proof of Theorem \ref{thm:FG}} 
\label{app:FG}

\begin{proof}
Choose $N_0\in \N$ as in Lemma \ref{lem:contract}. 

\step{1}{
Well-posedness and error estimate of the FG approximation \eqref{eq:FG}.
}

We consider the bilinear form $B: \dot{H}^1_N \times \dot{H}^1_N \to \R$, given by 
\[
B(u_N, w_N) 
:=
\int_{\T^d} a_N(x) \nabla u_N(x) \cdot \nabla w_N(x) \, dx.
\]
We note that $u_N$ solves \eqref{eq:FG}, if and only if, 
\begin{align}\label{eq:bil}
B(u_N, w_N) = \langle f_N, w_N \rangle_{L^2},
\end{align}
for all $w_N \in \dot{H}^1_N$. By Lemma \ref{lem:contract}, we have $a_N \ge \lambda/2$ for all $N\ge N_0$, and hence
\[
B(w_N, w_N) \ge \lambda/2 \Vert w_N \Vert_{\dot{H}^1}^2. 
\]
Thus, $B$ is a coercive bilinear form. The existence and uniqueness of a solution $u_N^\ast \in \dot{H}^1_N$ of \eqref{eq:bil} follows from this for any right-hand side $f_N$. Furthermore, we have the estimate
\begin{align*}
\Vert u_N^\ast \Vert_{H^1}^2
\le
2\Vert u_N^\ast \Vert_{\dot{H}^1}^2
\le
\frac{4}{\lambda} B(u_N^\ast,u_N^\ast)
=
\frac{4}{\lambda}\langle u_N^\ast, f_N \rangle_{L^2}
\le
\frac{4}{\lambda}\Vert u_N^\ast \Vert_{{H}^1} \Vert f_N \Vert_{{H}^{-1}},
\end{align*}
and thus, $\Vert u_N^\ast \Vert_{{H}^1} \le 4\lambda^{-1} \Vert f_N \Vert_{{H}^{-1}}$.

It is also straightforward to show that by inductive differentiation of \eqref{eq:darcy}, one can also obtain higher-order estimates for $k\in \N$, $k\ge 1$ for the (non-discretized) elliptic equation \eqref{eq:darcy}, of the form:
\begin{align} \label{eq:reg}
\Vert u \Vert_{H^{k+1}} \le C \Vert f \Vert_{H^{k-1}},
\end{align}
where $C = C(k,d,\Vert a \Vert_{C^k})>0$.

To prove the claimed error estimate, we note that if $u$ solves \eqref{eq:darcy}, then $P_N u \in \dot{H}^1_N$ solves
\begin{align*}
-P_N \nabla \cdot (a_N \nabla P_N u) 
&= P_N f + P_N \nabla \cdot ((a-a_N) \nabla P_N u) 
\\
&\qquad + P_N \nabla \cdot (a \nabla (1-P_N) u).
\end{align*}
It thus follows that $w_N = P_Nu - u_N^\ast$ solves
\begin{align*}
-P_N \nabla \cdot (a_N \nabla w_N) 
&= (P_Nf -f_N) + P_N \nabla \cdot ((a-a_N) \nabla P_N u) 
\\
&\qquad + P_N \nabla \cdot (a \nabla (1-P_N) u)
\\
&=: (I) + (II) + (III).
\end{align*}
The stability estimate then implies that $\Vert w_N\Vert_{{H}^{1}} \le 4\lambda^{-1} \Vert (I) + (II) + (III) \Vert_{H^{-1}}$, can be bounded in terms of the $H^{-1}$-norm of the right-hand side. We can now estimate
\begin{align*}
\Vert (I)\Vert_{H^{-1}} 
&=
\Vert (P_Nf - f_N) \Vert_{H^{-1}}
\\
&\le 
\Vert (1-P_N)f \Vert_{H^{-1}}
+
\Vert f - f_N \Vert_{H^{-1}}
\lesssim_{d,s}
N^{-k} \Vert f \Vert_{H^{k-1}}.
\end{align*}
The last inequality follows from the fact that $f \in H^{k-1}(\T^d)$, for $k-1>d/2$, and the (pseudo-)spectral approximation estimate.
For the second term, we fix a small $\delta > 0$, and obtain
\begin{align*}
\Vert (II)\Vert_{H^{-1}} &\le
\Vert 
(a-a_N)
\nabla P_N u 
\Vert_{L^2}
\\
&\lesssim_{d}
\Vert 
a-a_N
\Vert_{L^\infty}
\Vert
u
\Vert_{H^1}
\\
&\lesssim_{d,\delta,\lambda}
\Vert 
a-a_N
\Vert_{H^{d/2+\delta}}
\Vert
f
\Vert_{H^{-1}},
\end{align*}
where we have used the embedding $H^{d/2+\delta} \embeds L^\infty$ in the last step. Assuming that $\delta>0$ is chosen sufficiently small, so that $s\ge d/2+\delta$, we can then estimate
\[
\Vert (II) \Vert_{H^{-1}}
\lesssim_{s,d}
\lambda^{-1}
N^{-(s-d/2-\delta)}
\Vert a \Vert_{H^s}\Vert f \Vert_{H^{-1}}.
\]
Finally, the third term can be bounded as follows:
\begin{align*}
\Vert (III)\Vert_{H^{-1}}
&\lesssim_{d}
\Vert a \nabla (1-P_N) u \Vert_{L^2}
\\
&\lesssim_{d}
\Vert a \Vert_{L^\infty} \Vert (1-P_N) u \Vert_{H^1}
\\
&\lesssim_{s,d}
N^{-k} \Vert a \Vert_{H^s}  \Vert u \Vert_{H^{k+1}}
\\
&\lesssim_{s,d,\lambda,k,\Vert a \Vert_{C^k}}
N^{-k} \Vert a \Vert_{H^s} \Vert f \Vert_{H^{k-1}}.
\end{align*}
In the last step, we used the higher regularity estimate \eqref{eq:reg}. We note that for this step we assume that $a \in H^s$, $s>d/2+k$, so that we have an embedding $H^s \embeds C^k$. Under this condition, we can further estimate $(II)$ by
\begin{align*}
\Vert (II) \Vert_{H^{-1}}
&\lesssim_{s,d}
N^{-(s-d/2-\delta)} \lambda^{-1} \Vert a \Vert_{H^s} \Vert f \Vert_{H^{-1}}
\\
&\lesssim_{s,d,\lambda, \Vert a \Vert_{H^s}, \Vert f \Vert_{H^{-1}}} N^{-k}.
\end{align*}
Combining the above estimates for $(I)$, $(II)$ and $(III)$, we conclude that for any $s\in \R$ and $k\in \N$ satisfying $s>d/2+k$, we have
\[
\Vert w_N \Vert_{H^1}
=
\Vert 
u_N^\ast - P_N u
\Vert_{H^{1}}
\le C N^{-k},
\]
provided that $a\in H^s$, $f \in H^{k-1}$, where $C = C\left(s,d,\lambda,k,\Vert a \Vert_{H^s}, \Vert f \Vert_{H^{k-1}}\right)>0$ is independent of $N$. 

The proof is now finished by observing that 
\[
\Vert u - u_N^\ast \Vert_{H^1}^2
=
\Vert (1-P_N) u \Vert_{H^1}^2
+
\Vert P_N u - u_N^\ast \Vert_{H^1}^2,
\]
and that, by the higher-regularity estimate \eqref{eq:reg}, we have
\[ 
\Vert (1-P_N) u \Vert_{H^1}
\lesssim_{k,d}
N^{-k} \Vert u \Vert_{H^{k+1}}
\lesssim_{s,k,d,\lambda,\Vert a\Vert_{H^s}, \Vert f \Vert_{H^{k-1}}}
N^{-k}.
\]
Thus, it follows that there exists a constant $C = C(s,k,d,\lambda,\Vert a\Vert_{H^s}, \Vert f \Vert_{H^{k-1}})>0$, such that 
\begin{align} \label{eq:FGexact}
\Vert u - u_N^\ast \Vert_{H^1}
\le
C N^{-k},
\end{align}
where $u^\ast_N$ is the solution of the Fourier-Galerkin discretization \eqref{eq:FG}.

\step{2}{
Picard iteration estimate.
}

The output of Algorithm \ref{alg:darcy} is obtained by Picard iteration, i.e. by $K$-fold application of the mapping $F_N: \dot{H}^1_N \to \dot{H}^1_N$, yielding a recursively defined sequence $u^0_N := 0$, and $u^k_N = F_N(u^{k-1}_N)$, for $k=1,\dots, K$. Since $F_N$ is a contraction with $\Lip(F_N) \le 1-\lambda/2$ for $N\ge N_0$, by Lemma \ref{lem:contract}, and since $u^\ast_N$ is the unique fixed point of $F_N$, this implies that
\[
\Vert 
u^K_N - u^\ast_N
\Vert_{\dot{H}^1}
\le
(1-\lambda/2)^K \Vert u^\ast_N \Vert_{\dot{H}^1}.
\]
Using the definition of $H^1$, $\dot{H}^1$, and the regularity estimate \eqref{eq:reg}, we can further estimate
\[
\Vert 
u^K_N - u^\ast_N
\Vert_{{H}^1}
\le
2
\Vert 
u^K_N - u^\ast_N
\Vert_{\dot{H}^1}
\le
8\lambda^{-1}(1-\lambda/2)^K \Vert f_N \Vert_{H^{-1}}.
\]
By definition of $K$, we have
\[
K \ge \frac{\log(\lambda N^{-k})}{\log(1-\lambda/2)},
\]
and hence $(1-\lambda/2)^K \le \lambda N^{-k}$. This yields
\[
\Vert 
u^K_N - u^\ast_N
\Vert_{{H}^1}
\le
8 N^{-k}\Vert f_N \Vert_{H^{-1}}.
\]

Combining Steps 1 and 2, implies that there exists a constant $C>0$, depending on $s$, $k$, $d$, $\lambda$, $\Vert a \Vert_{H^s}$ and $\Vert f \Vert_{H^{k-1}})>0$, such that 
\[
\Vert u - u_N \Vert_{H^1} \le C N^{-k},
\]
for all $N\ge N_0$, where $u_N := u^K_N$ is the output of Algorithm \ref{alg:darcy}. This is the claimed estimate.
\end{proof}
\subsection{Neural network approximation of quadratic non-linearities}
\label{sec:quad}
Our next aim is to prove Theorem \ref{thm:fno-darcy}. As stated in the main text, the proof relies crucially on the following Lemmas, which show that neural networks can efficiently approximate certain quadratic non-linearities. We start with the following result,
\begin{lemma} \label{lem:quadratic}
Let $\sigma \in C^3$ be a activation function. Let $d\in \N$. There exists a constant $C = C(d) > 0$, such that for any $\epsilon >0$ and $B>0$, there exists 
\begin{enumerate}
\item a neural network $\hN_1: \R^2 \to \R$, such that 
\[
\sup_{|a|,|b|\le B} |\hN_1(a,b) - ab| \le \epsilon,
\]
and $\width(\hN_1) \le C$, $\depth(\hN_1) \le C$,
\item a neural network $\hN_2: \R \times \R^d \to \R^d$, such that 
\[
\sup_{|a|,\Vert v\Vert_{\ell^2}\le B} |\hN_2(a,v) - av| \le \epsilon,
\]
and $\width(\hN_2) \le C$, $\depth(\hN_2) \le C$,
\item a neural network $\hN_3: \R^d \times \R^{d\times d}\to \R^d$, such that 
\[
\sup_{\Vert v\Vert_{\ell^2},\Vert U \Vert_{\ell^2\to \ell^2}\le B} |\hN_3(v,U) - v\cdot U| \le \epsilon,
\]
and $\width(\hN_3) \le C$, $\depth(\hN_3) \le C$.
\end{enumerate}
\end{lemma}

\begin{proof}
Points (2) and (3) easily follow from (1), by parallelizing multiple networks $\hN_1$. To see the first claim (1), we note that the quadratic function $y \mapsto y^2$ can be approximated for $y\in [-B,B]$, $B>0$, to arbitrary precision by finite differences
\[
y^2 = 
\underbrace{
\frac{\sigma(x+hy) - 2\sigma(x) + \sigma(x-hy)}{h^2 \sigma^{(2)}(x)}
}_{\sq_h(y)}
+ R(h;y),
\]
where we assume that $x$ is chosen so that the second derivative $\sigma^{(2)}(x) \ne 0$, and $|R(h,y)| \le Ch$ for all $y\in [-B,B]$, and $C = C(B)$. Finally, following \cite{Yar1}, we observe that the product $ab$ of two numbers $a$, $b$ can be expressed in the form
\[
ab = \frac12 \left( (a+b)^2 - a^2 - b^2\right)
= \frac12 \left( \sq_h(a+b) - \sq_h(a) - \sq_h(b) \right) + \tilde{R}(h;a,b),
\]
where $\tilde{R}$ is related to $R$, and there exists a constant $C = C(B)$, such that $|\tilde{R}(h;a,b)|\le Ch$, for all $a,b \in [-B,B]$. Since $\sq_h$ is a neural network of finite width and depth (independent of $h$), and since the last expression
\[
\hN_h(a,b) :=  \frac12 \left( \sq_h(a+b) - \sq_h(a) - \sq_h(b) \right),
\]
is simply a linear combination of $\sq_h$, we conclude that $\hN_h$ is a neural network of uniformly bounded width and depth (uniform in $h$), and for sufficiently small $h>0$, we have
\[
\sup_{a,b\in [-B,B]}
|\hN_h(a,b) - ab | \le \epsilon.
\]
This concludes the proof.
\end{proof}
Using the above lemma, one can prove the following result,
\begin{lemma} \label{lem:darcy-quadratic}
Assume that the activation function $\sigma\in C^3$ is three times continuously differentiable and non-linear. There exists a constant $C>0$, such that for any $N\in \N$, and for any $\epsilon, B > 0$, there exists a $\Psi$-FNO $\cN: L^2_{2N}(\T^d;\R)\times L^2_{2N}(\T^d;\R) \to L^2_{2N}(\T^d;\R)$, with 
\[
\depth(\cN), \; \lift(\cN) \le C, 
\quad
\width(\cN) \le C N^d,
\]
such that we have
\[
\left\Vert
P_N\left(
a_N \nabla u_N
\right)
- 
\cN(a_N, u_N)
\right\Vert_{L^2_N}
\le 
\epsilon,
\]
for all trigonometric polynomials $a_N, u_N\in L^2_N(\T^d; \R)\subset L^2_{2N}(\T^d; \R)$ of degree $|k|_\infty \le N$, satisfying the bound $\Vert a_N \Vert_{L^2}, \Vert u_N \Vert_{L^2} \le B$.
\end{lemma}
\begin{proof}
First, we observe that there exists a linear FNO layer $\cL: L^2_{2N} \to L^2_{2N}$, with a suitable choice of the Fourier multiplier matrix $P$, such that 
\[
\cL(u_N) 
= \cF^{-1}_{2N} (P \cF(u_N)) 
= \cF^{-1}_{2N} \left(\sum_{|k|_\infty\le 2N} ik \hat{u}_k e^{i\langle k, x\rangle} \right)
= \nabla u_N,
\]
is satisfied \emph{exactly} for all $u_N\in L^2_{2N}(\T^d;\R)$. We also note that if $\Vert u_N \Vert_{L^2} \le B$, then $\Vert \nabla u_N \Vert_{L^2}\le N B$. By the fact that all norms are equivalent on the finite-dimensional space $L^2_{2N}(\T^d;\R^{d\times d})$, there exists a constant $C'>0$ (depending on $N$), such that 
\[
\sup_{x\in \T^d} \Vert \nabla u_N(x) \Vert_{\ell^2} 
\le
C' \Vert \nabla u_N \Vert_{L^2} \le C'NB,
\]
whenever $\Vert u_N \Vert_{L^2}\le B$. We similarly see that by norm equivalence on $L^2_{2N}(\T^d;\R)$, there exists a constant $C''>0$ (depending on $N$), such that we also have
\[
\sup_{x\in \T^d} | a_N(x) | \le C''\Vert a_N \Vert_{L^2} \le C''B,
\]
for any $\Vert a_N \Vert_{L^2}\le B$. Let $\bar{B} := \max(C'NB, C''B)$. By Lemma \ref{lem:quadratic}, there exists an ordinary neural network $\hN: \R\times \R^d \to \R^d$ with $\width(\hN), \depth(\hN) \le C(d)$ (with $C=C(d)$ \emph{independent} of $N$), such that 
\[
\sup_{|a|, \Vert v \Vert_{\ell^2}\le \bar{B}}
\Vert 
\hN(a,v) - av
\Vert_{\ell^2}
\le \epsilon.
\]
By Lemmas \ref{lem:sigma}, \ref{lem:linear} and the composition lemma \ref{lem:composition}, the composition 
\[
(a_N,u_N) \mapsto (a_N, \cL(u_N))=(a_N, \nabla u_N) \mapsto \hN(a_N,\nabla u_N),
\]
can be represented by a $\Psi$-FNO $\tN: L^2_{2N}(\T^d;\R^2) \to L^2_{2N}(\T^d;\R^d)$, and by construction, we have 
\begin{align} \label{eq:tN}
\sup_{x\in \T^d}
\Vert
a_N(x) \nabla u_N(x) - \tN(a_N,u_N)(x)
\Vert_{\ell^2}
\le 
\epsilon,
\end{align}
for all $a_N$, $u_N \in L^2_{2N}$, with $ \Vert a_N \Vert_{L^2}$, $\Vert u_N \Vert_{L^2} \le B$. Furthermore, since $\cL$ is a linear layer, and $\hN$ is an ordinary neural network with $\width(\hN), \depth(\hN)\le C= C(d)$, we in fact conclude that for some new constant $C = C(d)>0$, we have
\[
\width(\tN) \le C N^d, 
\quad
\depth(\tN) \le C, 
\quad
\lift(\tN) \le C.
\]
Finally, we note that the projection $P_N: L^2_{2N} \to L^2_{2N}$ onto Fourier modes with wavenumbers $|k|_\infty \le N$ can again be represented exactly by a linear $\Psi$-FNO layer $\hL$, and by Lemmas \ref{lem:linear}, \ref{lem:composition}, there exists a $\Psi$-FNO $\cN: L^2_{2N}(\T^d;\R^2) \to L^2_{2N}(\T^d;\R^d)$, such that $\Vert \cN(a_N,u_N) - \hL \circ \tN(a_N,u_N) \Vert_{L^2} \le \epsilon$, for all $\Vert a_N \Vert_{L^2}$, $\Vert u_N \Vert_{L^2} \le B$, and such that 
\[
\width(\cN) \le C N^d, 
\quad
\depth(\cN) \le C, 
\quad
\lift(\cN) \le C,
\]
where $C = C(d)>0$ depends only on $d$. Combining this with \eqref{eq:tN}, we conclude that $\cN$ satisfies
\begin{align*}
\Vert P_N(a_N\nabla u_N) - \cN(a_N,u_N) \Vert_{L^2} 
&\le
\Vert \hL \circ \tN(a_N,u_N) - \cN(a_N,u_N) \Vert_{L^2} 
\\
&\qquad
+\Vert P_N(a_N\nabla u_N)  - \hL \circ \tN(a_N,u_N) \Vert_{L^2} 
\\
&=
\Vert \hL \circ \tN(a_N,u_N) - \cN(a_N,u_N) \Vert_{L^2} 
\\
&\qquad
+\Vert P_N(a_N\nabla u_N)  - P_N \tN(a_N,u_N) \Vert_{L^2} 
\\
&\le 
\epsilon + \Vert a_N\nabla u_N - \tN(a_N,u_N) \Vert_{L^2} 
\\
&\le 
\epsilon + |\T^d|^{1/2} \Vert a_N\nabla u_N - \tN(a_N,u_N) \Vert_{L^\infty} 
\\
&\le
\left( 1 + (2\pi)^{d/2} \right) \epsilon,
\end{align*}
for all $\Vert a_N \Vert_{L^2}$, $\Vert u_N \Vert_{L^2} \le B$. Since $\epsilon > 0$ was arbitrary, the claim follows.
\end{proof}
\subsection{Proof of Theorem \ref{thm:fno-darcy}}
\label{app:fno-darcy}
The stage is now set for the proof of Theorem \ref{thm:fno-darcy} in the following,
\begin{proof}[Proof of Theorem \ref{thm:fno-darcy}]
Since the claim is an asymptotic statement, it suffices to consider $N\ge N_0$, where $N_0$ is the constant of Lemma \ref{lem:contract}. Indeed, the exceptional cases $N< N_0$ can be handled by suitably enlarging the constant $C$. We will thus assume that $N\ge N_0$, and $f\in \dot{H}^{k-1}$ are given. We fix $F_N := (-\Delta)f_N \in L^2_{N}(\T^d)$ for the rest of this proof, where $f_N:= \dot{P}_N \cI_{2N}f$ is defined as in Algorithm \ref{alg:darcy}. 

We define two operators $\hN_1$, $\hN_2$ as follows: We let
\[
\hN_1: L^2_{2N}(\T^d;\R^2) \to L^2_{2N}(\T^d;\R\times \R^{d}),
\quad
\hN_1(a,u) := (a,P_N(a\nabla u)),
\]
and define $\hN_2: L^2_{2N}(\T^d;\R\times \R^d) \to L^2_{2N}(\T^d;\R^2)$ by
\[
\hN_2(a,U) := \left(a,\dot{P}_N (-\Delta)^{-1} \nabla \cdot U + F_N\right).
\]
In terms of $\hN_1$, $\hN_2$, Algorithm \ref{alg:darcy}, which defines a mapping $a \mapsto u_N = \hN(a)$, can be written in the form 
\[
\hN(a) 
=
\hQ \circ \underbrace{\hN_2 \circ \hN_1 \circ \dots \circ \hN_2 \circ \hN_1}_{\text{$K$-fold composition}} \circ \hR(a),
\]
where $\hR(a) := (a,0)$, $\hQ(a,u) := u$ and where $K \lesssim\log(N)$. By the composition lemma \ref{lem:composition}, to prove the claim of this theorem, it therefore suffices to show the following\\

\noindent
\textbf{Claim:} For any $B>0$, there exists $C>0$, such that for any $\epsilon > 0$, there exist $\Psi$-FNOs $\cN_1$ and $\cN_2$, with $\width(\cN_1)$, $\width(\cN_2) \le CN^d$, $\depth(\cN_1)$, $\depth(\cN_2)$, $\lift(\cN_1)$, $\lift(\cN_2) \le C$, and such that 
\[
\Vert \hN_1(u) - \cN_1(u) \Vert_{L^2}, \; \Vert \hN_2(u) - \cN_2(u) \Vert_{L^2} \le \epsilon,
\]
for all $u\in L^2_{2N}(\T^d)$ with bounded norm $\Vert u \Vert_{L^2}\le B$.\\

\noindent
For $\hN_1$, the claim follows from Lemma \ref{lem:darcy-quadratic}. For $\hN_2$, we note that $\hN_2$ can be represented exactly by a linear FNO layer with $W = \begin{pmatrix} 1 & 0 \\ 0 & 0 \end{pmatrix}$, bias $b_j = (0,F_{N}(x_j))$, and with Fourier multiplier matrix $P(k) = \begin{pmatrix} 0 & 0 \\ 0 & \tilde{P}(k) \end{pmatrix}$, where $\tilde{P}(k) := 1_{[|k|_\infty \le N]} \frac{ik^T}{|k|^2}$, so that
\[
\hN_2(v)_j = \hN_2(v)(x_j)
=
Wv_j + b_j + \cF_N^{-1}( P\cF_N v)_j, 
\]
for $v_j = (a(x_j), U(x_j))$. The claim for $\hN_2$ thus follows from the linear approximation lemma \ref{lem:linear}.
\end{proof}
\section{Technical Results and proofs from Section {\ref{sec:NS}}}
\subsection{Properties of the Pseudo-spectral scheme {\eqref{eq:scheme1}}}
\label{app:s1}
Our first aim is to show that the implicit operator equation that defines the scheme \eqref{eq:scheme1}. To this end, we have following lemmas,
\begin{lemma} \label{lem:advect1}
If $v\in L^2_N(\T^d;\R^d)$ and $\dt \Vert v \Vert_{L^\infty} N  \le \frac12$, then we have \[
\Vert \dt \P_N(v \cdot \nabla w) \Vert_{L^2}
\le
\frac12 \Vert w \Vert_{L^2},
\]
for all $w\in L^2_N(\T^d;\R^d)$. In particular, this estimate holds provided that 
\[
\dt \Vert v \Vert_{L^2} N^{d/2+1}\le \frac12.
\]
\end{lemma}

\begin{proof}
We first note that for any $v\in L^2_N$, we have
\[
\dt \Vert v\Vert_{L^\infty} N
\le
\dt \Vert v\Vert_{L^2} N^{d/2+1}
\le \frac12.
\]
The claim is now an immediate consequence of the fact that 
\begin{align*}
\Vert \dt \P_N(v \cdot \nabla w) \Vert_{L^2}
&\le
\Vert \dt v \cdot \nabla w \Vert_{L^2}
\\
&\le
\dt \Vert v \Vert_{L^\infty} \Vert \nabla w \Vert_{L^2}
\\
&\le
\dt \Vert v \Vert_{L^\infty}N \Vert w \Vert_{L^2},
\end{align*}
for any $w\in L^2_N(\T^d;\R^d)$. 
\end{proof}

We can now state the following lemma on the well-posedness:

\begin{lemma} \label{lem:welldef1}
Let $U>0$. If $\Vert u^n_N\Vert_{L^2}\le U$ for $n=0,\dots, n_T$, and if the CFL condition 
\begin{align}\label{eq:CFL1}
\dt U N^{d/2+1} \le \frac12,
\end{align}
is satisfied, then the recursion \eqref{eq:scheme1} is well-defined.
\end{lemma}

\begin{proof}
The recursion \eqref{eq:scheme1} can be written in the form 
\begin{align} \label{eq:recL}
\cT_n
u^{n+1}_N
=
u^n,
\end{align}
where the operator $\cT_n: \dot{L}^2_N(\T^d;\div)\to \dot{L}^2_N(\T^d;\div)$ is given by
\[
\cT_n w_N
:=
w_N + \dt \P_N \left(u^n_N  \cdot \nabla w_N\right) - \nu \dt \Delta w_N.
\]
The claimed well-posedness of the recursion follows from the fact that, under the CFL assumption \eqref{eq:CFL1}, the operator $\cT_n$ is invertible: Indeed, by Lemma \ref{lem:advect1}, this implies that for any $w_N \in L^2_N(\T^d;\div)$, we have
\[
\Vert \dt \P_N\left(u^n_N \cdot \nabla w_N\right) \Vert_{L^2} 
\le
\frac12 \Vert w_N \Vert_{L^2}.
\]
But then, we have for any $w_N \in L^2_N(\T^d;\div)$, that 
\begin{align*}
\Vert \cT_n w_N \Vert_{L^2}
&\ge 
\Vert (1-\dt \nu \Delta) w_N \Vert_{L^2} 
- 
\Vert \dt \P_N\left(u^n_N \cdot \nabla w_N\right) \Vert_{L^2_N} 
\\
&\ge
\Vert w_N \Vert_{L^2} - \frac{1}{2} \Vert w_N \Vert_{L^2}
\\
&=
\frac{1}{2} \Vert w_N \Vert_{L^2},
\end{align*}
where, in the first step, we have used the fact that $-\dt \nu \Delta$ is a non-negative operator. In particular, the estimate $\Vert \cT_n w_N\Vert_{L^2} \ge \frac12 \Vert w_N\Vert_{L^2}$ implies that $\cT_n: L^2_N(\T^d;\div) \to L^2_N(\T^d;\div)$ is injective. Since $L^2_N(\T^d;\div)$ is finite-dimensional, we conclude that $\cT_n$ is actually invertible, and hence the recursion \eqref{eq:scheme1} is well-defined, i.e. it can be solved for $u^{n+1}_N$, given $u^n_N$.
\end{proof}

By Lemma \ref{lem:welldef1}, to prove the well-posedness of the scheme \eqref{eq:scheme1}, it remains to be shown that with a suitable choice of the time-step $\dt$, we have a uniform $L^2$-energy bound of the form $\Vert u^n_N \Vert_{L^2} \le U$, for some $U>0$. This is the content of the following lemma:

\begin{lemma} \label{lem:energyest1}
Assume that $\Vert u^0_N \Vert_{L^2} \le U$, then the first-order scheme \eqref{eq:scheme1} is well-defined for any time-step satisfying the CFL condition $\dt U N^{d/2+1} \le \frac12$, and we have $\Vert u^{n}_N \Vert_{L^2} \le \Vert u^0_N \Vert_{L^2}\le U$ for all $n=0,\dots, n_T$.
\end{lemma}

\begin{proof}
We show inductively that if $\Vert u^n_N \Vert_{L^2} \le U$, then $\Vert u^{n+1}_N \Vert_{L^2} \le \Vert u^n_N \Vert_{L^2} \le U$. The well-posedness then follows from Lemma \ref{lem:welldef1}. To see that $\Vert u^{n+1}_N \Vert_{L^2} \le \Vert u^n_N \Vert_{L^2}$, we integrate \eqref{eq:scheme1} against $u^{n+1}_N$, to find
\begin{align*}
\Vert u^{n+1}_N \Vert_{L^2}^2
&=
\langle u^{n+1}_N, u^n_N \rangle
-
\dt \left\langle u^{n+1}_N, \P_N\left(u^n_N \cdot \nabla u^{n+1}_N\right) \right\rangle
-
\nu \dt \Vert \nabla u^{n+1}_N \Vert_{L^2}^2
\\
&=
\langle u^{n+1}_N, u^n_N \rangle
-
\dt \underbrace{\left\langle u^{n+1}_N, u^n_N \cdot \nabla u^{n+1}_N \right\rangle}_{=0}
-
\nu \dt \Vert \nabla u^{n+1}_N \Vert_{L^2}^2
\\
&\le
\left|
\langle u^{n+1}_N, u^n_N \rangle
\right|
\le
\Vert u^{n+1}_N \Vert_{L^2} \Vert u^n_N \Vert_{L^2}.
\end{align*}
This proves the claim.
\end{proof}
Next, we have the following Lemma on the convergence of the iterations in algorithm \ref{alg:NS1},
\begin{lemma} \label{lem:iter1}
Given $\Vert u^n_N \Vert_{L^2} \le U$, assume that the CFL condition $\dt U N^{d/2+1}\le \frac12$ is satisfied. Define a recursive sequence $w^{n,k}_N \in \dot{L}^2_N(\T^d;\div)$, $k\in \N$, by $w^{n,0}_N := 0$, and $w^{n,k+1}_N := F(w^{n,k}_N)$, where $F$ is defined by \eqref{eq:Frec}. Then, we have $\Vert u^{n+1}_N - w^{n,k}_N \Vert_{L^2} \le 2^{-k} \Vert u^n_N \Vert_{L^2}$, for $k \in \N$, and $\Vert w^{n,k}_N \Vert_{L^2} \le (1+2^{-k}) \Vert u^n_N \Vert_{L^2}$.
\end{lemma}

\begin{proof}
By Lemma \ref{lem:advect1}, and the fact that $(1-\nu \dt \Delta)^{-1}$ is a contraction, it follows that $\Lip(F) \le \frac12$. By Picard iteration, and recalling that $u^{n+1}_N$ is the unique fixed point of the recursion $w^{n,k}_N$ ($k\in \N$), it immediately follows that 
\[
\Vert u^{n+1}_N - w^{n,k}_N \Vert_{L^2}
\le
\frac1{2^k} \Vert u^{n+1}_N - w^{n,0}_N \Vert_{L^2}
\le
\frac{\Vert u^n_N \Vert_{L^2}}{2^k}
.
\]
The last step is a consequence of the a priori $L^2$-bound $\Vert u^{n+1}_N \Vert_{L^2}\le \Vert u^n_N \Vert_{L^2}$ proven in Lemma \ref{lem:energyest1}. In particular, this estimate implies that 
\[
\Vert w^{n,k}_N \Vert_{L^2} 
\le  
\Vert u^{n+1}_N \Vert_{L^2} + \Vert w^{n,k}_N - u^{n+1}_N \Vert_{L^2} 
\le
(1+2^{-k}) \Vert u^n_N \Vert_{L^2}.
\]
\end{proof}
\begin{remark} \label{rem:bound1}
Note that as a consequence of Lemma \ref{lem:iter1}, we recursively find that
\begin{align*}
\Vert u^{n+1}_N \Vert_{L^2}
&=
\Vert w^{n,\kappa_0}_N \Vert_{L^2}
\le 
(1+2^{-\kappa_0}) \Vert u^n_N \Vert_{L^2}
\le
\left(1+\frac{\dt^2}{T^2}\right) \Vert u^n_N \Vert_{L^2}
\\
&\le
\dots
\le
\left(1+\frac{\dt^2}{T^2}\right)^n \Vert u^0_N \Vert_{L^2}
\le
\exp\left(\frac{\dt}{T}\right) U
\le
e U,
\end{align*}
for $n=0,\dots, n_T$. In particular, this ensures that the CFL condition \eqref{eq:CFL1} is satisfied for all $u^n_N$, generated by Algorithm \ref{alg:NS1}.
\end{remark}

Finally, we provide the proof of the convergence Theorem \ref{thm:scheme1}. 
\subsection{Proof of Theorem \ref{thm:scheme1}} \label{app:s1pf}
In this appendix, we provide a detailed proof of the convergence estimate of Theorem \ref{thm:scheme1}, for the first-order scheme defined by Algorithm \ref{alg:NS1}. 

To this end, let $u(t)$ be an exact solution of \eqref{eq:NS}, satisfying the assumptions of Theorem \ref{thm:scheme1}, and let $u^0_N, \dots, u^{n_T}_N$ denote the sequence generated by Algorithm \ref{alg:NS1}. For any $n=0,\dots, n_T-1$, we denote by $u^{n+1,\ast}_N$ the solution of a single time-step with the semi-implicit scheme \eqref{eq:scheme1}, starting from $u^{n}_N$, i.e. satisfying
\begin{align} \label{eq:scheme-ast1}
\frac{u^{n+1,\ast}_N-u^n_N}{\dt}
+
\P_N\left(
u^n_N \cdot \nabla u^{n+1,\ast}_N
\right)
=
\nu \Delta u^{n+1,\ast}_N.
\end{align}
We recall that by Lemma \ref{lem:iter1} and Remark \ref{rem:bound1}, we have a uniform bound $\Vert u^n_N \Vert_{L^2}\le eU$. Since by assumption, the time-step $\dt e U N^{d/2+1}\le \frac1{2}$ satisfies the relevant CFL condition, it follows that a unique solution $u^{n+1,\ast}_N$ exists for all $n$. We also recall that by Lemma \ref{lem:iter1}, and by our choice of the number of iteration steps $\kappa$ in Algorithm \ref{alg:NS1}, we have
\[
\Vert u^{n+1,\ast}_N - u^{n+1}_N \Vert_{L^2}
\le
2^{-k} eU
\le
C \dt^2,
\]
where $C>0$ depends only on $U$ and the final time $T$. 

Our first goal is to derive an estimate on the magnification of the approximation error due to a single timestep $u^n_N \mapsto u^{n+1}_N$. Let $u(t)$ be the exact solution of \eqref{eq:NS}. Observing that 
\begin{align} \label{eq:single-dt}
\begin{aligned}
\Vert u^{n+1}_N - u(t^{n+1}) \Vert_{L^2}
&\le
\Vert u^{n+1}_N - u^{n+1,\ast}_N \Vert_{L^2} + \Vert u^{n+1,\ast}_N - u(t^{n+1}) \Vert_{L^2}
\\
&\le
C \dt^2 + \Vert u^{n+1,\ast}_N - u(t^{n+1}) \Vert_{L^2},
\end{aligned}
\end{align}
we only need to consider the error introduced by a single time-step $u^n_N \mapsto u^{n+1,\ast}_N$ of the semi-implicit scheme \eqref{eq:scheme1}. To this end, we can write 
\[
\frac{u(t^{n+1}) - u(t^n)}{\dt}
+ \P_N\left( u(t^{n})\cdot \nabla u(t^{n+1})\right) = \nu \Delta u(t^{n+1}) + \mathcal{E}^n,
\]
where $\mathcal{E}^n$ collects all error terms:
\[
\mathcal{E}^n 
=
\mathcal{E}^n_{\tau} 
+
\mathcal{E}^n_{NL}
+
\mathcal{E}^n_{P}
+
\mathcal{E}^n_{\nu},
\]
where
\begin{align*}
\mathcal{E}^n_{\tau}
&=
\frac{u(t^{n+1})-u(t^n)}{\dt}
-
\partial_t u(t^n),
\\
\mathcal{E}^n_{NL}
&=
\P \left( u(t^n) \cdot \nabla \left(u(t^{n+1}) - u(t^n)\right)\right), 
\\
\mathcal{E}^n_{P}
&=
(1-\P_N)\left( u(t^n)\cdot \nabla u(t^n)\right),
\\
\mathcal{E}^n_{\nu}
&=
-
\nu \Delta \left( u(t^{n+1}) - u(t^n) \right).
\end{align*}

Subtracting \eqref{eq:scheme-ast1}, and introducing the short-hand notation $e^n := u(t^n) - u_N^n$, $e^{n+1} := u(t^{n+1}) - u^{n,\ast}_N$, we find
\[
\frac{e^{n+1} - e^{n}}{\dt}
=
-\P_N\left(
e^n \cdot \nabla u(t^{n+1}) 
\right)
-
\P_N \left(
u^n_N \cdot \nabla e^{n+1}
\right)
+
\nu \Delta e^{n+1}
+
\mathcal{E}^n.
\]
Next, integrate against $e^{n+1}_N := P_N e^{n+1}$ to find
\begin{align*}
\frac1{2\dt}
&\left(
\Vert e_N^{n+1} \Vert_{L^2}^2 
+ 
\Vert e_N^{n+1} - e_N^n \Vert_{L^2}^2
- 
\Vert e_N^n \Vert_{L^2}^2
\right)
\\
&\quad \le
\Vert e^n \Vert_{L^2} \Vert \nabla u \Vert_{L^\infty_{t,x}} \Vert e_N^{n+1} \Vert_{L^2}
-
\langle \P_N(u^n_N \cdot \nabla e^{n+1}), e^{n+1} \rangle
-
\nu \Vert \nabla e^{n+1} \Vert_{L^2}^2 
+
\langle \cE^n, e^{n+1} \rangle.
\end{align*}
We note that 
\begin{align}
\label{eq:11}
\begin{aligned} 
\langle \P_N(u^n_N \cdot \nabla e^{n+1}), e^{n+1} \rangle
&=
\langle P_N(u^n_N \cdot \nabla e^{n+1}), e^{n+1} \rangle
\\
&=
\langle u^n_N \cdot \nabla e^{n+1}, e^{n+1} \rangle
-
\langle (1-P_N) u^n_N \cdot \nabla e^{n+1}, e^{n+1} \rangle
\\
&=
0 -
\langle u^n_N \cdot \nabla e^{n+1}, (1-P_N) e^{n+1} \rangle
\\
&=
-
\langle u^n_N \cdot \nabla e^{n+1}, (1-P_N) u(t^{n+1}) \rangle
\\
&=
\langle e^{n+1},  u^n_N \cdot \nabla (1-P_N) u(t^{n+1}) \rangle
\\
&=
-\langle e^{n+1},  e^n \cdot \nabla (1-P_N) u(t^{n+1}) \rangle
\\
&\qquad
+ \langle e^{n+1},  u(t^n) \cdot \nabla (1-P_N) u(t^{n+1}) \rangle
\end{aligned}
\end{align}
We now note that we can rewrite the last term as follows:
\begin{align*}
\langle e^{n+1},  u(t^n) \cdot \nabla (1-P_N) u(t^{n+1}) \rangle
&=
\left\langle
(1-P_N)\left(e^{n+1} \otimes  u(t^n)\right),
\nabla u(t^{n+1})
\right\rangle
\end{align*}
Using the fact that 
\[
(1-P_N)\left(e^{n+1} \otimes  u(t^n)\right)
=
(1-P_N)\left( (1-P_{N/2}) e^{n+1} \otimes  (1-P_{N/2}) u(t^n)\right),
\]
it then follows that we have
\begin{align*}
\langle P_N(u^n_N \cdot \nabla e^{n+1}), e^{n+1} \rangle
&= -\langle e^{n+1},  e^n \cdot \nabla (1-P_N) u(t^{n+1}) \rangle
\\
&\qquad
\left\langle
(1-P_N)\left(e^{n+1} \otimes  u(t^n)\right),
\nabla u(t^{n+1})
\right\rangle
\\
&\le
\Vert e^{n+1} \Vert_{L^2} \Vert e^n \Vert_{L^2} \Vert (1-P_N) \nabla u(t^{n+1}) \Vert_{L^\infty}
\\
&\qquad
+ \Vert e^{n+1} \Vert_{L^2} \Vert (1-P_{N/2}) u(t^n) \Vert_{L^2} \Vert (1-P_N)\nabla u(t^{n+1}) \Vert_{L^\infty}.
\end{align*}
For $r>d/2+1$, we have a continuous embedding $H^r \embeds W^{1,\infty}$, and an inequality of the form
\[
\Vert (1-P_{N/2}) u(t^n) \Vert_{L^2}
\lesssim_{r,d}
\frac{1}{N^{r}} \Vert u(t^n) \Vert_{H^r}.
\]
Hence we can estimate 
\begin{align*}
\langle P_N(u^n_N \cdot \nabla e^{n+1}), e^{n+1} \rangle
&\lesssim_{r,d}
\Vert e^{n+1} \Vert_{L^2} \Vert e^n \Vert_{L^2} \Vert u(t^n) \Vert_{H^r}
\\
&\qquad
+ \Vert e^{n+1} \Vert_{L^2} \frac{\Vert u(t^n) \Vert_{H^r}^2}{N^{2r}} .
\end{align*}

Estimating the products on the right-hand side using the inequality $ab \le \epsilon a^2 + \frac1{4\epsilon} b^2$ with suitable $\epsilon > 0$, it follows that there exists a constant $C>0$ (independent of $\nu>0$, $N$ and $n$), such that
\begin{gather} \label{eq:est0}
\begin{aligned} 
\Vert e^{n+1} \Vert^2_{L^2}
&\le
\left(
1
+
C \dt \Vert u \Vert_{C_t(H^{r}_x)}^2
\right)
\Vert e^n \Vert^2 _{L^2}
-
\nu \Vert \nabla e^{n+1} \Vert_{L^2}^2
\\
&\qquad + 
C\dt N^{-2r} \Vert u \Vert_{C_t(H^r_x)}^4
+
C \dt |\langle \cE^n, e^{n+1} \rangle|.
\end{aligned}
\end{gather}

\subsubsection{Time-differencing error}
We note that 
\begin{align*}
\langle \cE^n_{\dt}, e^{n+1}\rangle
&=
\frac1\dt \int_{t^n}^{t^{n+1}}
\int_{t^n}^t
\left
\langle
\partial_t^2 u(s)
,
e^{n+1}
\right
\rangle
\, ds
\, dt
\\
&=
\frac1\dt \int_{t^n}^{t^{n+1}}
\int_{t^n}^t
\left
\langle
-P
\left\{
\partial_t u(s)\cdot \nabla u(s) + u(s)\cdot \nabla \partial_t u(s)
\right\}
,
e^{n+1}
\right
\rangle
\, ds
\, dt
\\
&\qquad + 
\frac1\dt \int_{t^n}^{t^{n+1}}
\int_{t^n}^t
\left
\langle
\nu \Delta u(s)
,
e^{n+1}
\right
\rangle
\, ds
\, dt
\\
&=: (I)_\dt + (II)_\dt.
\end{align*}
The first term can be bounded from above by
\begin{align*}
(I)_\dt &\lesssim
\dt \left(
\Vert u \Vert_{C^1_t(L^2_x)} \Vert u \Vert_{C_t(W^{1,\infty}_x)} + \Vert u \Vert_{C_t(L^\infty_x)} \Vert u \Vert_{C^1_t(H^1_x)}
\right) \Vert e^{n+1} \Vert_{L^2}
\\
&\lesssim_{r,d}
\dt \Vert u \Vert_{C^1_t(H^{r-2})} \Vert u \Vert_{C_t(H^r)} \Vert e^{n+1}\Vert_{L^2},
\end{align*}
provided that $r > d/2+2 \ge 3$ (the last bound is automatic for $d\ge 2$). For the second term, we derive the bound
\[
(II)_\dt\lesssim
\dt \nu \Vert u \Vert_{C_t(H^1_x)} \Vert \nabla e^{n+1} \Vert_{L^2}.
\]
Thus, for any $\epsilon> 0$ (to be specified later), we have for some constant $C = C(r,d)> 0$:
\begin{gather} \label{eq:err-dt}
\begin{aligned}
|\langle \cE^n_{\dt} , e^{n+1} \rangle |
&\le
\epsilon\left( \Vert e^{n+1} \Vert_{L^2} + \nu \Vert \nabla e^{n+1} \Vert_{L^2} \right)
\\
&\qquad 
+
\frac{C \dt^2 }{ \epsilon } \left( 
\Vert u \Vert_{C^1_t(H^{r-2})}^2 \Vert u \Vert_{C_t(H^r)}^2
+
\nu\Vert u \Vert_{C_t(H^1_x)}^2
\right).
\end{aligned}
\end{gather}

\subsubsection{Non-linear time-differencing error}

For the error associated with $\cE^{n}_{NL}$, we simply estimate (for $\epsilon > 0$ to be determined later)
\[
|\langle \cE^n_{NL}, e^{n+1} \rangle |
\le
\epsilon \Vert e^{n+1} \Vert_{L^2}^2
+
\epsilon^{-1} \Vert \cE^n_{NL} \Vert_{L^2}^2,
\]
and we observe that 
\[
\Vert \cE^n_{NL} \Vert_{L^2} 
\le
\dt
\Vert u \Vert_{C_t(L^\infty_x)}
\Vert u \Vert_{C^1_t(H^1_x)}
\le
\dt
\Vert u \Vert_{C_t(H^r_x)}
\Vert u \Vert_{C^1_t(H^{r-2}_x)},
\]
assuming that $r > d/2+2 \ge 3$. This yields
\begin{align} \label{eq:err-NL}
|\langle \cE^n_{NL}, e^{n+1} \rangle |
\le
\epsilon \Vert e^{n+1} \Vert_{L^2}^2
+
\frac{\dt^2}{\epsilon}
\Vert u \Vert_{C_t(H^r_x)}^2
\Vert u \Vert_{C^1_t(H^{r-2}_x)}^2
\end{align}

\subsubsection{Nonlinear projection error}

Again, we estimate the error $\cE^{n}_{P}$ using the simple estimate 

\[
|\langle \cE^n_{P}, e^{n+1} \rangle |
\le
\epsilon \Vert e^{n+1} \Vert_{L^2}^2
+
\epsilon^{-1} \Vert \cE^n_{P} \Vert_{L^2}^2,
\]
with $\epsilon > 0$ to be specified later.
We furthermore note that 
\[
(1-P_N) \left\{u(t^n) \cdot \nabla u(t^{n+1})\right\}
=
(1-P_N) \left\{ (1-P_{N/2})u(t^n) \cdot (1-P_{N/2})\nabla u(t^{n+1})\right\},
\]
which implies that
\begin{align*}
\Vert \cE^n_{P} \Vert_{L^2}
&\le
\Vert (1-P_{N/2})u(t^n) \Vert_{L^2}
\Vert (1-P_{N/2}) \nabla u(t^{n+1}) \Vert_{L^\infty}
\lesssim_{r,d}
N^{-r} \Vert u \Vert_{C_t(H^r)}^2.
\end{align*}
In the last step, we have used the fact that $\Vert (1-P_{N/2}) v \Vert_{L^2} \lesssim_{r,d} N^{-r} \Vert v \Vert_{H^r}$ for $r \ge 0$, and that by Sobolev embedding
\begin{align*}
\Vert (1-P_{N/2}) \nabla u(t^{n+1}) \Vert_{L^\infty}
&\lesssim_d
\Vert (1-P_{N/2}) u(t^{n+1}) \Vert_{H^{r}}
\le
\Vert u \Vert_{C_t(H^r_x)},
\end{align*}
for any $r > d/2+1$. Thus, there exists $C = C(d,r)>0$, such that
\begin{gather} \label{eq:err-P}
|\langle \cE^n_{P}, e^{n+1} \rangle |
\le
\epsilon \Vert e^{n+1} \Vert_{L^2}^2
+
\frac{C}{\epsilon N^{2r}} \Vert u \Vert_{C_t(H^r_x)}^2.
\end{gather}

\subsubsection{Viscosity error}

Finally, we note that 
\begin{align}
|\langle \cE^n_{\nu}, e^{n+1}\rangle|
&\le
\nu \Vert \nabla (u(t^{n+1}) - u(t^n))\Vert_{L^2} \Vert \nabla e^{n+1}\Vert_{L^2}
\notag
\\
&\le
\nu \epsilon \Vert \nabla e^{n+1}\Vert_{L^2}^2 + \frac{\nu \dt^2}{\epsilon} \Vert u \Vert_{C^1_t(H^1_x)}^2
\notag
\\
&\le
\nu \epsilon \Vert \nabla e^{n+1}\Vert_{L^2}^2 + \frac{\nu \dt^2}{\epsilon} \Vert u \Vert_{C^1_t(H^{r-2}_x)}^2, 
\label{eq:err-nu}
\end{align}
for any $r>d/2+2\ge 3$.

\subsubsection{The final stability estimate}
Choosing $\epsilon = 1/4$, it follows from \eqref{eq:err-dt}, \eqref{eq:err-NL}, \eqref{eq:err-P} and \eqref{eq:err-nu}, that the total error term $\langle \cE^n, e^{n+1}\rangle$ can be estimated by
\begin{align} \label{eq:err}
|\langle \cE^n, e^{n+1}\rangle|
\le
\Vert e^{n+1} \Vert^2_{L^2}
+
\nu \Vert \nabla e^{n+1} \Vert_{L^2}^2
+
C^\ast \left(\dt^2 + N^{-2r}\right).
\end{align}
for some constant $C^\ast > 0$, depending only on $r > d/2+2$, the spatial dimension $d$ and the norms $\Vert u \Vert_{C_t(H^r)}$, $\Vert u \Vert_{C^1_t(H^r)}$ of the exact solution $u$.

Substitution of the error estimate \eqref{eq:err} in \eqref{eq:est0} finally yields
\begin{align*} 
(1-\dt) \Vert e^{n+1} \Vert^2 _{L^2}
\le
\left(
1 + C^\ast \dt
\right)
\Vert e^n \Vert^2 _{L^2}
+
\dt C^\ast \left( \dt^2 + N^{-2r}\right),
\end{align*}
where the constant $C^\ast$ depends only on $r\ge d/2+2$, the dimension $d$ and the norms $\Vert u \Vert_{C_t(H^r_x)}$ and $\Vert u \Vert_{C^1_t(H^{r-2}_x)}$. Assuming that $\dt \le 1/2$, dividing by $(1-\dt)$, and noting that
\[
\frac{1+C^\ast \dt}{1-\dt}
=
1 + \dt \left(C^\ast + \frac{(1+C^\ast \dt)}{1-\dt}\right)
\le
1 + \dt 2\left(C^\ast + 1\right),
\]
and
\[
\frac{C^\ast}{1-\dt} \le 2C^\ast,
\]
we can clearly absorb the additional factor of $(1-\dt)^{-1}$ by increasing the constant $C^\ast$, if necessary. 

From this, we conclude that for a time-step $\dt \le 1/2$ satisfying the CFL condition \eqref{eq:CFL1}, there exists a constant $C^\ast = C^\ast(r,d,\Vert u \Vert_{C_t(H^r_x)}, \Vert u \Vert_{C^1_t(H^{r-2}_x)}) > 0$, such that
\begin{align} \label{eq:E}
\begin{aligned}
\Vert u^{n+1,\ast}_N - u(t^{n+1}) \Vert^2_{L^2}
&\le
\left(
1 + C^\ast\dt
\right)
\Vert u^{n}_N - u(t^{n}) \Vert^2_{L^2}
 +
\dt C^\ast \left( \dt^2 + N^{-2r}\right).
\end{aligned}
\end{align}
In fact, recalling also that $\Vert u^{n+1}_N - u^{n+1,\ast}_N \Vert_{L^2} \le C \dt^2$ by \eqref{eq:single-dt}, we find that an inequality of the form \eqref{eq:E} remains true with $u^{n+1,\ast}_N$ replaced by $u^{n+1}_N$. Indeed, we have 
\begin{align} 
\begin{aligned}
\Vert u^{n+1}_N - u(t^{n+1}) \Vert^2_{L^2}
&\le
\Vert u^{n+1,\ast}_N - u(t^{n+1}) \Vert^2_{L^2}
+ 2C\dt^2 \Vert u^{n+1,\ast}_N - u(t^{n+1}) \Vert_{L^2}
+
C^2 \dt^4
\\
&\le
\left(1 + \epsilon\right)\Vert u^{n+1,\ast}_N - u(t^{n+1}) \Vert^2_{L^2}
+
C^2(1+4\epsilon^{-1}) \dt^4
\\
&\explain{=}{(\epsilon := \dt)}
\left(1 + \dt\right)\Vert u^{n+1,\ast}_N - u(t^{n+1}) \Vert^2_{L^2}
+
C^2(1+4\dt^{-1}) \dt^4
\\
&\le
\left(
1 + C^\ast\dt
\right)
\Vert u^{n}_N - u(t^{n}) \Vert^2_{L^2}
 +
\dt C^\ast \left( \dt^2 + N^{-2r}\right),
\end{aligned}
\end{align}
where the last estimate follows from \eqref{eq:E}, and $C^\ast$ has been suitably enlarged (but still only depends on $r,d,\Vert u \Vert_{C_t(H^r_x)}, \Vert u \Vert_{C^1_t(H^{r-2}_x)}$). In particular, denoting $E^n := \Vert u^n_N - u(t^n)\Vert_{L^2}^2$, we have obtained
\[
E^{n+1} \le (1+C^\ast \dt) E^{n} + \dt C^\ast \left( \dt^2 + N^{-2r}\right),
\]
and from Gronwall's inequality it now follows that
\[
E^n \le e^{C^\ast T}\left[E^0 + C^\ast T \left( \dt^2 + N^{-2r} \right)\right].
\]
We note that $E^0 = \Vert (1-\dot{P}_N\cI_{2N}) u(t=0)\Vert^2_{L^2} \le N^{-2r} \Vert u \Vert_{H^r}^2$. And hence, we finally find, for $n=0,\dots, n_T$, that
\begin{align} \label{eq:stability}
\Vert u^n_N - u(t^n) \Vert_{L^2}
= \sqrt{E^n} \le C\left( \dt + N^{-r}\right),
\end{align}
where $C>0$ depends only on $T$, $r$, $d$, $\Vert u \Vert_{C_t(H^r_x)}$ and $\Vert u \Vert_{C^1_t(H^{r-2}_x)}$.
\subsection{Proof of Theorem \ref{thm:fno-NS1}}
\label{app:NS1}
In this appendix, we provide a proof for Theorem \ref{thm:fno-NS1}. At this proof relies on Lemma \ref{lem:fno-NSnonlin}, we prove this lemma below. 
\begin{proof}[Proof of Lemma \ref{lem:fno-NSnonlin}]
The proof of Lemma \ref{lem:fno-NSnonlin} is almost identical to the proof of Lemma \ref{lem:darcy-quadratic}; where in the present case, we replace $a_N \to u_N$ and refer to Lemma \ref{lem:quadratic} point (3), rather than point (2). The only main difference being that in the last step of the proof, the projection $P_N$ is now replaced by the Leray projection $\P_N$. However, also for $\P_N$, we observe that $\P_N: L^2_{2N}(\T^d;\R^d) \to L^2_{2N}(\T^d;\R^d)$ can again be represented \emph{exactly} by a linear $\Psi$-FNO layer. Thus, replacing the linear $\Psi$-FNO layer $\hL$ which represents $P_N$ in the proof of Lemma \ref{lem:darcy-quadratic} by a layer representing $\P_N$, an almost identical argument also applies in this case, and yields for any $\epsilon, \, B > 0$, a $\Psi$-FNO $\cN: L^2_{2N}(\T^d;\R^d) \to L^2_{2N}(\T^d;\R^d)$, such that
\[
\Vert \P_N(u_N \cdot \nabla u_N) - \cN(u_N)\Vert_{L^2}\le \epsilon,
\]
for all $\Vert u_N \Vert_{L^2} \le B$, with
\[
\width(\cN) \le C N^d, 
\quad
\depth(\cN), \, \lift(\cN) \le C,
\]
where $C = C(d)$.
\end{proof}
\begin{proof}[Proof of Theorem \ref{thm:fno-NS1}]
Given $N\in \N$, choose $\dt \sim N^{-r}$, such that the CFL condition $\dt N^{d/2+1} U \le \frac12$ is satisfied. This is possible, since $r \ge d/2+2$, by assumption. It then follows that $n_T \sim N^r$, and we note that Algorithm \ref{alg:NS1} can be written as the composition of $O(n_T \log(n_T)) = O(N^r \log(N))$ mappings of the form 
\[
\hN_1(u_N^n,w_N^{n,k})
:=
\begin{bmatrix}
u_N^n \\
(1-\nu\dt \Delta)^{-1} u_N^n
-
\dt (1-\nu \dt \Delta)^{-1} \P_N\left(u^n_N \cdot \nabla w^{n,k}_N\right)
\end{bmatrix},
\]
where $\Vert u_N^n \Vert_{L^2}, \Vert w^{n,k}_N \Vert_{L^2} \le 2U$ for all $k,n$.
Applying the replacement lemma, Lemma \ref{lem:replacement}, the claim will thus follow if we can show that there exists a constant $C>0$ independent of $N$, such that for any $\epsilon > 0$, there exists a $\Psi$-FNO $\cN_1: L^2_{2N}(\T^d;\div) \to L^2_{2N}(\T^d;\div)$, such that $\Vert \cN_1(u_N,w_N) - \hN_1(u_N,w_N) \Vert_{L^2} \le \epsilon$, and 
\[
\width(\cN_1) \le C N^d \sim C \epsilon^{-d/r}, 
\quad
\depth(\cN_1) \le C,
\quad
\lift(\cN_1) \le C.
\]
This is immediate for the approximation of the first component of $\hN_1$. For the second component, we note that we can write it as a composition:
\[
\begin{bmatrix}
u_N^n \\
w_N^n
\end{bmatrix}
\mapsto 
\begin{bmatrix}
u_N^n \\
\P_N\left(u^n_N \cdot \nabla w^{n,k}_N\right)
\end{bmatrix}
\mapsto 
(1-\nu\dt \Delta)^{-1} 
\left\{
u_N^n
-
\dt \P_N\left(u^n_N \cdot \nabla w^{n,k}_N\right)
\right\}.
\]
To finish the proof, we note that the first mapping can be approximated to arbitrary accuracy $\epsilon>0$ with a $\Psi$-FNO of width $\lesssim N^d$, and uniformly bounded depth and lift, by Lemma \ref{lem:fno-NSnonlin}. The second mapping can be represented \emph{exactly} by a linear FNO layer. Thus, $\cN_1$ can be obtained as the composition of a $\Psi$-FNO approximating the quadratic non-linearity, and a linear FNO layer, implying the claimed complexity estimate.
\end{proof}
\subsection{A second-order in time accurate pseudo-spectral method for approximating the Navier-Stokes equation {\eqref{eq:NS}} and its emulation by {$\Psi$}-FNOs.}
\label{app:scheme2}
Our aim is to describe a second-order accurate (in time) version of the pseudo-spectral scheme \eqref{eq:scheme1}. To this end, we propose the following scheme,
\begin{gather} \label{eq:scheme2}
\begin{gathered}
\frac{u^{n+1}_N - u^n_N}{\dt}
+
\P_N\left(
\left[
\frac32 u^n - \frac12 u^{n-1} 
\right]
\cdot
\nabla \frac12\left[
 u^{n+1} + u^{n}
\right]
\right)
=
\nu \Delta \frac 12 
\left[
u^{n+1} + u^n
\right],
\end{gathered}
\end{gather}
In contrast to the first-order method \eqref{eq:scheme1} of the last section, to start the scheme \eqref{eq:scheme1}, we now require two starting values $u^0_N \approx u(t_0)$ and $u^1_N \approx u(t_1)$. Given initial data $u_0 \in \dot{H}^r(\T^d;\div)$, $r\ge d/2$, we propose to define $u^0_N := \cI_N u_0$, where $\cI_N$ denotes the pseudo-spectral projection, and to generate $u^1_N$ by the first-order accurate Algorithm \ref{alg:NS1} applied on the time-interval $[0,\dt]$, with reduced time-step of size $\sim \dt^2$.

As in the case of \eqref{eq:scheme1}, one needs to solve an implicit operator equation in the time update for \eqref{eq:scheme2}. Analogously, we will use a fixed point iteration to approximate this implicit equation, resulting in the following algorithm,
\begin{algorithm}[Second-order in time approximation of \eqref{eq:NS}] 
\label{alg:NS2}
\phantom{a} \\
\noindent
\begin{tabular}{lp{.8\textwidth}}
\textbf{Input:} & 
$U>0$, $N\in \N$, $T>0$, $\nu \ge 0$, a time-step $\dt>0$, such that $n_T = T/\dt \in \N$, and $\dt U N^{d/2+1} \le \frac1{2e}$, initial data $u_N^0 \in L^2_N(\T^d;\div)$, such that $\Vert u_N^0 \Vert_{L^2} \le U$.
\\
\textbf{Output:} & 
$u_N^{n_T} \in L^2_N(\T^d;\div)$ an approximation of the solution $S_T(u_N^0)$ of \eqref{eq:NS} at time $t=T$.
\end{tabular}
\begin{enumerate}
\item Set 
\[
\kappa := 
\left\lceil
\frac{\log\left(T^3 /\dt^3\right)}{\log(2)}
\right\rceil
\in \N.
\]
\item Compute $u^1_N \approx u(t_1)$ by applying Algorithm \ref{alg:NS1} on the time-interval $[0,\dt]$, with $n_T$ steps and with time-step $\dt' = \dt/n_T$.
\item For $n=1,\dots, n_T-1$:
\begin{enumerate}
\item Set $w^{n,0}_N := 0$,
\item For $k=1,\dots, \kappa_0$: Given the values $u^{n-1}_N, u^n_N$ from the previous steps, compute
\[
\hspace{40pt}
w^{n,k}_N := F_2^n(w^{n,k-1}_N),
\]
\item Set $u^{n+1}_N := w^{n,\kappa_0}_N$,
\end{enumerate}
\end{enumerate}
\end{algorithm}
Next, we have the following convergence theorem for the algorithm \ref{alg:NS2},
\begin{theorem} \label{thm:scheme2}
Let $U,T>0$. Consider the Navier-Stokes equations on $\T^d$, for $d\ge 2$. Assume that $r\ge d +2$, and let $u \in C([0,T]; H^r) \cap C^1([0,T];H^{r-2}) \cap C^2([0,T];H^1)$ be a solution of the Navier-Stokes equations \eqref{eq:NS}, such that $\Vert u \Vert_{L^2}\le U$. Choose a time-step $\dt$, such that $\dt U N^{d/2+1}\le (2e)^{-1}$. There exists a constant 
\[
C = C(T,d,r,\Vert u \Vert_{C_t(H^r_x)}, \Vert u \Vert_{C^1_t(H^{r-2}_x)}, \Vert u \Vert_{C^2_t(H^{1}_x)}) > 0,
\]
such that with $u^0_N := \cI_N u(0)$, if $\Vert u(t_1) - u^1_N \Vert_{L^2} \le \delta$, and for the sequence $u^2_N, \dots, u^{n_T}_N \in L^2_N(\T^d;\div)$ generated by Algorithm \ref{alg:NS1}, we have
\[
\max_{n=0,\dots, n_T}
\Vert u^n_N - u(t^n) \Vert_{L^2}
\le
C \left(\delta + \dt^2 + N^{-r}\right),
\]
where $n_T\dt = T$. In particular, choosing a suitable time-step $\dt \sim N^{-r/2}$, and assuming that $\delta \le C N^{-r}$, we have
\[
\max_{n=0,\dots, n_T}
\Vert u^n_N - u(t^n) \Vert_{L^2}
\le
3C N^{-r},
\]
with $n_T \sim N^{r/2}$.
\end{theorem}
The proof uses very similar techniques as the proof of Theorem \ref{thm:scheme1} and we omit it here. 

Finally and in complete analogy with the proof of Theorem \ref{thm:fno-NS1}, one can prove the following Theorem for the approximation of scheme \eqref{eq:scheme2} with a $\Psi$-FNO,
\begin{theorem} \label{thm:fno-NS2}
Let $U,T>0$ and viscosity $\nu \ge 0$. Consider the Navier-Stokes equations on $\T^d$, for $d\ge 2$. Assume that $r\ge d+2$, and let $\cV \subset C([0,T]; H^r) \cap C^1([0,T];H^{r-2}) \cap C^2([0,T];H^1)$ be a set of solutions of the Navier-Stokes equations \eqref{eq:NS}, such that $\sup_{u\in \cV} \Vert u \Vert_{L^2}\le U$, and 
\[
\bar{U}:= \sup_{u\in \cV} \left\{
\Vert u \Vert_{C_t(H^r_x)} + \Vert u \Vert_{C^1_t(H^{r-2}_x)} + \Vert u \Vert_{C^2_t(H^{1}_x)}
\right\}.
\]
For $t\in [0,T]$, denote $\cV_t := \set{u(t)}{u\in \cV}$. Let $\G: \cV_0 \to \cV_T$ denote the solution operator of \eqref{eq:NS}, mapping initial data $u_0 = u(t=0)$, to the solution $u(T)$ at $t=T$ of the incompressible Navier-Stokes equations. Then there exists a constant 
\[
C = C(d,r,U,\bar{U},T) > 0,
\]
such that for any $N\in \N$ there exists a $\Psi$-FNO $\cN: L^2_{2N}(\T^d;\R^d) \to L^2_{2N}(\T^d;\R^d)$, such that 
\[
\sup_{u\in \cV_0}
\Vert \G(u) - \cN(u) \Vert_{L^2}
\le
C N^{-r},
\]
and such that 
\[
\width(\cN) \le C N^{d}, 
\quad
\depth(\cN) \le C N^{r/2} \log(N),
\quad
\lift(\cN) \le C.
\]
\end{theorem}

\section{Proof of Theorem \ref{thm:deeponet}}
\label{app:E}
\begin{proof}
First, we note that each layer 
\[
v(x_j) \mapsto 
\sigma\left(
Wv(x_j) + b_j + \cF^{-1}_N \left( P \cF_N v \right)
\right),
\]
is simply the composition of 
\begin{itemize}
\item an affine mapping $\R^{d_v \times \cJ_N} \to \R^{d_v\times \cJ_N}$, and 
\item a componentwise application of the activation function $\sigma$.
\end{itemize}
In particular, the $\Psi$-FNO $\cN$, interpreted as a mapping 
\[
\R^{\cJ_N \times d_a} 
\to 
\R^{\cJ_N \times d_v}
\to 
\dots
\to
\R^{\cJ_N \times d_v}
\to 
\R^{\cJ_N \times d_u},
\]
can be represented by an ordinary neural network $\tilde{\beta}: \R^{\cJ_N \times d_a} \to \R^{\cJ_N \times d_u}$, with 
\[
\width(\tilde{\beta}) = |\cJ_N| d_v = \width(\cN),
\quad
\depth(\tilde{\beta}) = \depth(\cN).
\]
In fact, by suitably modifying the linear output layer of $\tilde{\beta}$, we can map the grid values encoded in the output $\tilde{\beta}_j(a) = \cN(a)(x_j)$ to the corresponding coefficients in a (real) trigonometric basis $\{\fb_k\}_{k\in \cK_N}$ with $\Span\{\fb_k\}_{k\in \cK_N} = \Span\{e^{i\langle k, x\rangle}\}_{k\in \cK_N}$; i.e. by modifying the linear output layer of $\tilde{\beta}$ (and re-indexing the components of the output), we obtain another neural network $\beta: \R^{\cJ_N\times d_a} \to \R^{\cK_N \times d_u}$, such that
\begin{align}\label{eq:ptwise}
\cN(a)(x_j) = \sum_{k\in \cK_N} \beta_k(a) \fb_k(x_j), \quad 
\forall \, j\in \cJ_N.
\end{align}
Since $\tilde{\beta}$ and $\beta$ only differ in their output layers, we clearly have
\[
\width({\beta}) =\width(\tilde{\beta}) = \width(\cN),
\quad
\depth({\beta}) =\depth(\tilde{\beta}) = \depth(\cN).
\]
Since $\{\fb_k\}_{k\in \cK_N}$ has the same span as $\{e^{i\langle k, x\rangle}\}_{k\in \cK_N}$, it follows form \eqref{eq:ptwise}, that 
\[
\cN(a)(x) = \sum_{k\in \cK_N} \beta_k(a) \fb_k(x), \quad 
\forall \, x\in \T^d.
\]
To prove the claim, it thus suffices to observe that there exists a neural network 
\[
\tau: \R^d \to \R^{\cK_N},
\quad
x \mapsto \{\tau_k(x)\}_{k\in \cK_N},
\]
such that the DeepOnet defined by $(\beta,\tau)$ satisfies
\begin{align*}
\sup_{\Vert a \Vert_{L^\infty} \le B}
\sup_{x\in \T^d}
&\left|
\cN(a)(x) - \sum_{k\in \cK_N} \beta_k(a) \tau_k(x)
\right|
\\
&=
\sup_{\Vert a \Vert_{L^\infty} \le B}\sup_{x\in \T^d}
\left|
\sum_{k\in \cK_N} \beta_k(a) \left[\fb_k(x) - \tau_k(x) \right]
\right|
\\
&\le
(2N+1)^d
\left(
\sup_{\Vert a \Vert_{L^\infty} \le B} \max_{k\in \cK_N}
|\beta_k(a) |
\right)
\sup_{x\in \T^d}
\max_{k\in \cK_N}
\left|\fb_k(x) - \tau_k(x) \right|
\\
&\le
(2N+1)^d
\left(
\sup_{\Vert a \Vert_{L^\infty} \le B} \Vert \cN(a) \Vert_{L^2}
\right)
\sup_{x\in \T^d}
\max_{k\in \cK_N}
\left|\fb_k(x) - \tau_k(x) \right|
\\
&=
\bar{B}
\sup_{x\in \T^d}
\max_{k\in \cK_N}
\left|\fb_k(x) - \tau_k(x) \right|
.
\end{align*}
\end{proof}

\end{document}